\documentclass[a4paper]{article}

\usepackage[english]{babel}
\usepackage[utf8]{inputenc}
\usepackage[T1]{fontenc}

\usepackage{mathtools}
\usepackage{amsthm}
\usepackage{amsfonts}
\usepackage{amssymb}
\usepackage{stmaryrd}

\usepackage[margin=3cm]{geometry}
\usepackage[pdftex,pdfborder={0 0 0},linktoc=all]{hyperref}
\usepackage{pdfpages}
\usepackage{color}

\usepackage{subfig}
\usepackage{pgf,tikz}
\usetikzlibrary{arrows}
\usetikzlibrary{cd}
\usetikzlibrary{graphs}
\usepackage[figurename=Figure]{caption}

\theoremstyle{plain}
	\newtheorem{thm}{Theorem}[section]
	\newtheorem{cor}[thm]{Corollary}
	\newtheorem{lem}[thm]{Lemma}
	\newtheorem{prop}[thm]{Proposition}

\theoremstyle{definition}
	\newtheorem{dfn}[thm]{Definition}
	\newtheorem*{ques}{Question}

\theoremstyle{remark}
	\newtheorem{rem}[thm]{Remark}
	\newtheorem{ex}[thm]{Example}

\numberwithin{equation}{section}

\newcommand{\C}{\mathbb{C}}
\newcommand{\N}{\mathbb{N}}
\renewcommand{\P}{\mathbb{P}}

\newcommand{\R}{\mathbb{R}}
\renewcommand{\S}{\mathbb{S}}

\newcommand{\cC}{\mathcal{C}}
\newcommand{\cD}{\mathcal{D}}
\newcommand{\cG}{\mathcal{G}}
\newcommand{\cI}{\mathcal{I}}
\newcommand{\cJ}{\mathcal{J}}
\newcommand{\cMJ}{\mathcal{MJ}}
\newcommand{\cO}{\mathcal{O}}
\newcommand{\cP}{\mathcal{P}}

\newcommand{\ut}{\underline{t}}
\newcommand{\ux}{\underline{x}}
\newcommand{\uy}{\underline{y}}

\newcommand{\hux}{\hat{\underline{x}}}

\renewcommand{\bar}{\overline}
\renewcommand{\epsilon}{\varepsilon}
\renewcommand{\geq}{\geqslant}
\renewcommand{\leq}{\leqslant}
\renewcommand{\tilde}{\widetilde}
\renewcommand{\hat}{\widehat}

\newcommand{\bl}[2]{\Bl_{#2}\parentheses*{#1}}
\newcommand{\config}{(\R^n)^p \setminus \Delta_p}
\newcommand{\configC}{(\C^n)^p \setminus \Delta_p}
\newcommand{\dx}{\dmesure\!}
\newcommand{\esp}[2][]{\mathbb{E}_{#1}\squarebrackets*{#2}}
\newcommand{\espcond}[3][]{\mathbb{E}_{#1}\squarebrackets*{#2 \mvert #3}}
\newcommand{\gauss}[1]{\mathcal{N}\parentheses*{0,#1}}
\newcommand{\gaussC}[1]{\mathcal{N}_\C\parentheses*{0,#1}}
\newcommand{\gr}[2][{\R_{p-1}[X]}]{\Gr_{#2}\parentheses*{#1}}
\newcommand{\jac}[1]{\Jac\parentheses*{#1}}
\newcommand{\jacC}[1]{\Jac_\C\parentheses*{#1}}
\newcommand{\mvert}{\mathrel{}\middle|\mathrel{}}
\newcommand{\one}{\mathbf{1}}
\newcommand{\trans}[1]{\prescript{\text{t}}{}{#1}}
\newcommand{\var}[1]{\Var\parentheses*{#1}}
\newcommand{\varC}[1]{\Var_\C\parentheses*{#1}}

\DeclareMathOperator{\Bl}{Bl}

\DeclareMathOperator{\dmesure}{d}
\DeclareMathOperator{\ev}{ev}
\DeclareMathOperator{\Gr}{Gr}
\DeclareMathOperator{\Id}{Id}
\DeclareMathOperator{\Jac}{Jac}
\DeclareMathOperator{\jet}{j}
\DeclareMathOperator{\mj}{mj}
\DeclareMathOperator{\Span}{Span}
\DeclareMathOperator{\sym}{Sym}
\DeclareMathOperator{\Var}{Var}
\DeclareMathOperator{\Vol}{Vol}

\DeclarePairedDelimiter{\brackets}{\{}{\}}

\DeclarePairedDelimiter{\norm}{\lvert}{\rvert}
\DeclarePairedDelimiter{\Norm}{\lVert}{\rVert}
\DeclarePairedDelimiter{\parentheses}{(}{)}
\DeclarePairedDelimiterX{\prsc}[2]{\langle}{\rangle}{#1\,, #2}
\DeclarePairedDelimiter{\squarebrackets}{[}{]}
\DeclarePairedDelimiterX{\ssquarebrackets}[2]{\llbracket}{\rrbracket}{#1,#2}

\author{Michele Ancona\,\thanks{Michele Ancona, Université Côte d’Azur, CNRS, LJAD, France; e-mail: \url{michele.ancona@unice.fr}.} \and Thomas Letendre\,\thanks{Thomas Letendre, Université Paris-Saclay, CNRS, LMO, France; e-mail: \url{letendre@math.cnrs.fr}. Partially supported by the ANR Grant UniRanDom (ANR-17-CE40-0008) and a PEPS Grant from the CNRS INSMI.}}
\date{\today}
\title{Multijet bundles and application to the finiteness of moments for zeros of Gaussian fields}

\begin{document}

\maketitle

\begin{abstract}
We define a notion of multijet for functions on $\R^n$, which extends the classical notion of jets in the sense that the multijet of a function is defined by contact conditions at several points. For all $p \geq 1$ we build a vector bundle of $p$-multijets, defined over a well-chosen compactification of the configuration space of $p$ distinct points in $\R^n$. As an application, we prove that the linear statistics associated with the zero set of a centered Gaussian field on a Riemannian manifold have a finite $p$-th moment as soon as the field is of class~$\cC^p$ and its $(p-1)$-jet is nowhere degenerate. We prove a similar result for the linear statistics associated with the critical points of a Gaussian field and those associated with the vanishing locus of a holomorphic Gaussian field.
\end{abstract}




\tableofcontents


\section{Introduction}
\label{sec: Introduction}

This paper is concerned with two different but related problems. The first one is to define a natural notion of multijet for a $\cC^k$ function on $\R^n$, generalizing the usual notion of $k$-jet. By multijet we mean that we want to consider a collection of jets at different points in $\R^n$ and patch them together in a relevant way. The second one is to find natural conditions on a Gaussian field $f:\R^n \to \R^r$ ensuring that the $(n-r)$-dimensional volume of $f^{-1}(0) \cap \mathbb{B}$ admits finite higher moments, where $\mathbb{B}$ stands for the unit ball in $\R^n$. One way to tackle this second problem is by considering the multijet of the random field $f$. In the following, we give more details about our contributions concerning the previous two problems, as well as some variations on these questions.


\subsection{Multijet bundles}
\label{subsec: Multijet bundles}

Let us start by recalling some standard facts about jets. See~\cite{Sau1989} for background on this matter. Let $x \in \R^n$ and $k \geq 0$. Two smooth functions $f$ and $g$ on $\R^n$ are said to have the same $k$-jet at $x$ if $f-g$ vanishes at $x$, as well as all its partial derivatives up to order $k$. Having the same $k$-jet at $x$ is an equivalence relation on $\cC^\infty(\R^n)$, and the space $\cJ_k(\R^n)_x$ of $k$-jets at~$x$ is the set of equivalence classes for this relation. We denote by $\jet_k(f,x)$ the $k$-jet of $f$ at~$x$, that is, its class in $\cJ_k(\R^n)_x$. The map $\jet_k(\cdot,x)$ is a linear surjection from $\cC^\infty(\R^n)$ onto the  finite-dimensional vector space $\cJ_k(\R^n)_x$. Of course, $\jet_k(f,x)$ makes sense even if $f$ is only $\cC^k$ and defined on some neighborhood of~$x$.

Considering the family of $k$-jet spaces for all $x \in \R^n$, the set $\cJ_k(\R^n) = \bigsqcup_{x \in \R^n} \cJ_k(\R^n)_x$ is equipped with a natural vector bundle structure over $\R^n$. Then, if $\Omega \subset \R^n$ is open and $f:\Omega \to \R$ is $\cC^k$, the map $\jet_k(f,\cdot)$ is a local section of $\cJ_k(\R^n) \to \R^n$ over $\Omega$. These definitions are well-behaved with respect to smooth changes of coordinates, so one can define similarly the vector bundle of $k$-jets of functions on a manifold $M$. More generally, if $E \to M$ is a vector bundle over $M$, there is a corresponding vector bundle $\cJ_k(M,E) \to M$ of $k$-jets of sections of $E \to M$.

In this paper, we are interested in defining similarly a notion of multijet and the associated vector bundles. That is, we want to consider smooth functions on $\R^n$ up to an equivalence relation defined by the vanishing of some derivatives at several points.

Let us make this more precise. Let $p \geq 1$ and $\Delta_p = \brackets*{\parentheses{x_1,\dots,x_p} \in (\R^n)^p \mvert \exists i\neq j \ \text{s.t.} \ x_i = x_j}$ denote the diagonal in $(\R^n)^p$. Given $\ux = \parentheses{x_1,\dots,x_p} \notin \Delta_p$, we say that $f$ and $g$ have the same multijet at $\ux$ if $f(x_i)=g(x_i)$ for all $i \in \ssquarebrackets{1}{p}$ (here we use the notation $\ssquarebrackets{a}{b} = \squarebrackets{a,b} \cap \N$). This is an equivalence relation on functions, defined by the vanishing of $f-g$ on the set $\brackets{x_1,\dots,x_p} \subset \R^n$, that is, by $p$ independent linear conditions. Thus the corresponding set of classes is a vector space of dimension $p$, which we denote by $\cMJ_p(\R^n)_{\ux}$. We also denote by $\mj_p(f,\ux)$ the class of $f$ in this space, that is, its multijet at $\ux$.

As will be explained later, this defines a vector bundle $\cMJ_p(\R^n)$ of rank $p$ over $\config$. Moreover, for all $\ux \notin \Delta_p$ the linear map $\mj_p(\cdot,\ux):\cC^\infty(\R^n)\to \cMJ_p(\R^n)_{\ux}$ is surjective, and, for all smooth $f$, its multijet $\mj_p(f,\cdot)$ is smooth. We would like to extend this picture over the whole of $(\R^n)^p$. Note that the surjectivity conditions rule out defining $\cMJ_p(\R^n)$ as $\cJ_0(\R^n)^p$ with $\mj_p(f,\ux) = \parentheses*{\jet_0(f,x_i)}_{1 \leq i \leq p}$. When $\ux \notin \Delta_p$, the previous notion of multijet is defined by $p$ independent linear conditions: vanishing at each of the $x_i$. The main issue is that, when $\ux \in \Delta_p$, these conditions are no longer independent and we need to replace them by another $p$-tuple of independent conditions.

A first natural idea is to look at vanishing with multiplicities. In dimension $n=1$ this works very well. Let $\ux \in \R^p$ be a permutation of $(y_1,\dots,y_1,\dots,y_m,\dots,y_m)$ where $(y_j)_{1 \leq j \leq m} \in \R^m \setminus \Delta_m$ and $y_j$ appears exactly $(k_j+1)$ times. We say that $f$ and $g$ have the same multijet at $\ux$ if $(f-g)^{(k)}(y_j)=0$ for all $j \in \ssquarebrackets{1}{m}$ and $k \in \ssquarebrackets{0}{k_j}$. In this sense, having the same multijet is equivalent to having the same Hermite interpolating polynomials at $\ux$. Thus, we can define $\cMJ_p(\R)$ as the trivial bundle $\R_{p-1}[X] \times \R^p \to \R^p$ and $\mj_p(f,\ux)$ as the Hermite interpolating polynomial of $f$ at $\ux$.

If $n>1$, the previous approach fails already for $p=2$. Let us consider $x \in \R^n$ and the corresponding $\ux=(x,x)\in \Delta_2$. Asking for the vanishing of $f-g$ and its differential at $x$ gives us $(n+1)$ independent conditions, which define the $1$-jet space $\cJ_1(\R^n)_x$. This space has dimension $n+1 >2$; hence it is too large to be the $\cMJ_2(\R^n)_{\ux}$ we are looking for. The next natural idea is to ask only for the vanishing at $x$ of $f-g$ and one of its directional derivatives. But whatever choice of directional derivative we make will lead to $\mj_2(f,\cdot)$ not being continuous at~$x$ for most $f \in \cC^\infty(\R^n)$. Actually, one cannot extend $\cMJ_2(\R^n)$ nicely over $(\R^n)^2$ if $n>1$. However, we can extend it nicely over a larger space: the blow-up $\bl{(\R^n)^2}{\Delta_2}$ of $(\R^n)^2$ along $\Delta_2$. The key idea is that $\bl{(\R^n)^2}{\Delta_2}$ contains $(\R^n)^2 \setminus \Delta_2$ as a dense open subset and that points in the complement of $(\R^n)^2 \setminus \Delta_2$ can be described by the following data: a base point $x \in \R^n$ and a direction $u \in \R\P^{n-1}$. This data tells us exactly which directional derivative to consider at the corresponding point in the exceptional locus of $\bl{(\R^n)^2}{\Delta_2}$. We will come back to this example later; see Example~\ref{ex: multijet bundle}.

This long discussion shows that there is a natural way to define a multijet bundle $\cMJ_p(\R^n)$ over the configuration space $\config$, but that it does not extend nicely over $(\R^n)^p$ in general. The case $p=2$ hints that it might however be possible to define a natural multijet bundle over a slightly larger space, containing a copy of $\config$ as a dense open subset. Our first main contribution is to define such an object. Its main properties are gathered in the following statement, where $\cC^k(\R^n,V)$ denotes the space of $\cC^k$ functions from $\R^n$ to $V$.

\begin{thm}[Existence of multijet bundles]
\label{thm: multijet bundle}
Let $n \geq 1$ and $p \geq 1$ and let $V$ be a real vector space of finite dimension $r\geq 1$. There exist a smooth manifold $C_p[\R^n]$ of dimension $np$ without boundary and a smooth vector bundle $\cMJ_p(\R^n,V) \to C_p[\R^n]$ of rank $rp$ with the following properties.

\begin{enumerate}

\item \label{item: thm multijet pi} There exists a smooth proper surjection $\pi:C_p[\R^n] \to (\R^n)^p$ such that $\pi^{-1}\parentheses*{\config}$ is a dense open subset of $C_p[\R^n]$, and $\pi$ restricted to $\pi^{-1}\parentheses*{\config}$ is a $\cC^\infty$-diffeomorphism onto $\config$.

\item \label{item: thm multijet mj} There exists a map $\mj_p:\cC^{p-1}(\R^n,V) \times C_p[\R^n] \to \cMJ_p(\R^n,V)$ such that:
\begin{itemize}
\item for all $z \in C_p[\R^n]$, the linear map $\mj_p(\cdot,z):\cC^{p-1}(\R^n,V) \to \cMJ_p(\R^n,V)_z$ is surjective;
\item for all $f \in \cC^{l+p-1}(\R^n,V)$, the section $\mj_p(f,\cdot)$ of $\cMJ_p(\R^n,V) \to C_p[\R^n]$ is~$\cC^l$.
\end{itemize}

\item \label{item: thm multijet ev} Let $z \in C_p[\R^n]$ be such that $\pi(z)=(x_1,\dots,x_p) \notin \Delta_p$ then for all $f \in \cC^{p-1}(\R^n,V)$ we have:
\begin{equation*}
\mj_p(f,z) = 0 \iff \forall i \in \ssquarebrackets{1}{p}, f(x_i)=0.
\end{equation*}

\item \label{item: thm multijet localness} Let $z \in C_p[\R^n]$ be such that $\pi(z)$ is obtained as a permutation of $(y_1,\dots,y_1,\dots,y_m,\dots,y_m)$ where $y_j$ appears exactly $(k_j+1)$ times and $y_1,\dots,y_m$ are pairwise distinct vectors in $\R^n$. Then, there exists a linear surjection $\Theta_z : \prod_{i=1}^m \cJ_{k_j}(\R^n,V)_{y_j} \to \cMJ_p(\R^n,V)_z$ such that
\begin{equation*}
\forall f \in \cC^{p-1}(\R^n,V), \qquad \mj_p(f,z) = \Theta_z\parentheses*{\jet_{k_1}(f,y_1),\dots,\jet_{k_m}(f,y_m)}.
\end{equation*}

\end{enumerate}

\end{thm}

\begin{rem}
\label{rem: thm multijet}
In Theorem~\ref{thm: multijet bundle}, the manifold $C_p[\R^n]$ does not depend on $V$. Condition~\ref{item: thm multijet pi} means that we can consider $\config$ as a dense open subset in $C_p[\R^n]$. Condition~\ref{item: thm multijet mj} consists of properties that we expect any reasonable notion of multijet to satisfy. Condition~\ref{item: thm multijet ev} means that, as in the previous discussions, if $\pi(z) \notin \Delta_p$ then $\cMJ_p(\R^n)_z = \cC^{p-1}(\R^n,V)/\sim$, where $f \sim g$ if and only if $f(x_i)=g(x_i)$ for all $i \in \ssquarebrackets{1}{p}$. Condition~\ref{item: thm multijet localness} means that, more generally, $\mj_p(f,z)$ only depends on the collection of jets $\parentheses{\jet_{k_j}(f,y_j)}_{1 \leq j \leq m}$. In particular, $\mj_p(f,z)$ still makes sense if $f$ is only $\cC^{k_j}$ on some neighborhood of $y_j$. This last condition also means that we can think of $\mj_p(f,z)$ intuitively as a family of $p$ independent linear combinations of partial derivatives of $f$, up to order $k_j$ at $y_j$. However this family is neither explicit nor unique in general.
\end{rem}

Let us now introduce some definitions and notation.

\begin{dfn}[Multijets]
\label{def: multijets}
Let $\Omega \subset \R^n$ be open, we let $C_p[\Omega] = \pi^{-1}(\Omega^p)$ and denote by $\cMJ_p(\Omega,V) \to C_p[\Omega]$ the restriction of $\cMJ_p(\R^n,V)$ to $C_p[\Omega]$.
We call $\cMJ_p(\Omega,V) \to C_p[\Omega]$ the \emph{bundle of $p$-multijets} of functions from $\Omega$ to $V$. Its fiber $\cMJ_p(\R^n,V)_z$ above $z \in C_p[\Omega]$ is the \emph{space of $p$-multijets at $z$}. If $V=\R$, we drop it from the notation and write $\cMJ_p(\Omega) \to C_p[\Omega]$. Let $f:\Omega\to V$ be of class $\cC^{p-1}$, we call the section $\mj_p(f,\cdot)$ of $\cMJ_p(\Omega,V)$ the \emph{$p$-multijet of~$f$} and its value at $z \in C_p[\Omega]$ the \emph{$p$-multijet of $f$ at $z$}.
\end{dfn}

The manifold $C_p[\R^n]$ is what is called in the literature a ``compactification'' of the configuration space $\config$. We will use this terminology, even though it is ill-chosen in our case since $C_p[\R^n]$ is not compact. However $C_p[\R^n]$ contains a diffeomorphic copy of $\config$ as a dense open subset and it is equipped with a proper surjection onto $(\R^n)^p$ so that, in a sense, it is locally a compactification of $\config$.

Compactifications of configuration spaces are built to understand how a configuration (ordered or not) of $p$ distinct points can degenerate as these points converge toward one another. They are usually obtained by blowing up various pieces of the diagonal. Points in the exceptional locus then correspond to singular configurations, with some extra data encoding along which paths regular configurations are allowed to degenerate in order to reach this singular configuration. The hope is that the extra data attached to singular configurations is enough to lift the singularities of the problem under consideration. The simplest example of this kind is the blow-up $\bl{(\R^n)^2}{\Delta_2}$ discussed above. More evolved examples are the space defined by Le Barz in~\cite{LBar1988}, the compactification of Fulton--MacPherson~\cite{FM1994} (see also~\cite{AS1994,Sin2004}), Olver's multispace~\cite{Olv2001}, the polydiagonal compactification of Ulyanov~\cite{Uly2002}, the construction of Evain~\cite{Eva2005} using Hilbert schemes, and many others.

In dimension~$n=1$, most of the compactifications of configuration spaces that we found in the literature coincide and can be used to define multijets; see for example~\cite{Anc2021} where Olver's multispace is used. In higher dimensions they are different and none of them exactly suited our needs. Thus to the best of our knowledge, the manifold $C_p[\R^n]$ in Theorem~\ref{thm: multijet bundle} is a new addition to the previous list. We define it by resolving the singularities of some real algebraic variety, using Hironaka's theorem~\cite{Hir1964,Hir1964a}. In particular, $C_p[\R^n]$ is obtained by a sequence of blow-ups along $\Delta_p$. Note that this sequence of blow-ups is neither explicit nor unique. Actually, $C_p[\R^n]$ itself is not uniquely defined, but this is not an issue for the applications we have in mind.


\subsection{Finiteness of moments for zeros of Gaussian fields}
\label{subsec: Finiteness of moments for zeros of Gaussian fields}

Let us now describe our contributions concerning zeros of Gaussian fields. Let $n \geq 1$ and let $r \in \ssquarebrackets{1}{n}$. In the following $n$ will always denote the dimension of the ambient space and $r$ the codimension of the random objects we are interested in.

Let $\Omega \subset \R^n$ be open and let $f:\Omega \to \R^r$ be a centered Gaussian field of class $\cC^1$. We will always assume that $f$ is non-degenerate, in the sense that $\det \var{f(x)}>0$ for all $x \in \Omega$. Under this hypothesis the zero set $Z=f^{-1}(0)$ is almost surely
$(n-r)$-rectifiable; see~\cite{AAL2023}. As such, it admits a well-defined $(n-r)$-dimensional volume measure $\dx \Vol_Z$ induced by the Euclidean metric on $\R^n$. We denote by $\nu$ the random Radon measure on $\Omega$ defined by:
\begin{equation}
\label{eq: def nu}
\forall \phi \in \cC^0_c(\Omega), \qquad \prsc{\nu}{\phi} = \int_Z \phi(x) \dx \Vol_Z(x),
\end{equation}
where $\cC^0_c(\Omega)$ denotes the space of continuous functions on $\Omega$ with compact support.

Actually, $\prsc{\nu}{\phi}$ makes sense as an almost surely defined random variable as soon as $\phi \in L^\infty(\Omega)$ and has compact support. This kind of test-function includes $\cC^0_c(\Omega)$ and indicator functions of bounded Borel subsets, which are the examples we are most interested in. Random variables of the type $\prsc{\nu}{\phi}$ are called the linear statistics of $\nu$ (or of $f$). Understanding the distribution of these linear statistics is one way to understand the distribution of the random measure $\nu$, or equivalently of the random set $Z$. For example, if $B \subset \Omega$ is a bounded Borel set and $\one_B$ denotes its indicator function, then $\prsc{\nu}{\one_B}$ is the $(n-r)$-dimensional volume of $Z \cap B$.

In this setting, a classical question is to determine conditions on the field $f$ ensuring that its linear statistics admit finite moments. If $n=r=1$, such conditions were first obtained in~\cite{Bel1966}. More generally see~\cite[Thm.~3.6]{AW2009}, which holds even if $f$ is not Gaussian. If $n\geq r=1$, a similar result is proved in~\cite{AAD+2023}; see also~\cite[Thm.~4.4]{AAGL2019}. For a survey of previous results for hypersurfaces (i.e.,~$r=1$), we refer to~\cite[Chap.~3, Sect.~2.7]{AW2009} in dimension $n=1$ and to the introduction of~\cite{AAD+2023} in dimension $n\geq 1$. Note that~\cite[Thm.~1.2]{Pry2020a} implies the finiteness of all moments of the nodal length for some Gaussian fields in $\R^2$. This problem was also studied in~\cite{MV1993} for points in~$\R^2$.

Our second main result gives simple conditions on the field $f$ ensuring the finiteness of the $p$-th moments of its linear statistics in any dimension and codimension. These conditions are of two kinds: we require the field to be regular enough, and to be non-degenerate in the following sense.

\begin{dfn}[$p$-non-degeneracy]
\label{def: p non degeneracy}
Let $p \geq 1$ and let $f:\Omega \to \R^r$ be a $\cC^p$ centered Gaussian field. We say that the field $f$ is \emph{$p$-non-degenerate} if for all $x \in \Omega$ the centered Gaussian vector
\begin{equation*}
\parentheses*{\strut f(x),D_xf,\dots,D_x^pf} \in \bigoplus_{k=0}^p \sym^k(\R^n) \otimes \R^r
\end{equation*}
is non-degenerate, where $\sym^k(\R^n)$ denotes the space of symmetric $k$-linear forms on $\R^n$ and $D^k_xf \in \sym^k(\R^n) \otimes \R^r$ stands the $k$-th differential of $f$ at $x$.
\end{dfn}

\begin{rem}
\label{rem: p non degeneracy}
If $f=(f_1,\dots,f_r)$, the $p$-non-degeneracy condition means more concretely that for all $x \in \Omega$ the Gaussian vector $\parentheses*{\strut \partial^{\alpha}f_i(x)}_{1 \leq i \leq r; \norm{\alpha} \leq p}$ is non-degenerate, where we used multi-index notation (see Section~\ref{subsec: Spaces of functions, sections and jets}). More abstractly, this condition means that the $p$-jet $\jet_p(f,x)$ of $f$ is non-degenerate for all $x \in \Omega$.
\end{rem}

\begin{thm}[Finiteness of moments]
\label{thm: moments Rn}
Let $n \geq 1$, let $r \in \ssquarebrackets{1}{n}$ and let $p \geq 1$. Let $\Omega \subset \R^n$ be open, let $f:\Omega \to \R^r$ be a centered Gaussian field and let $\nu$ be defined as in Equation~\eqref{eq: def nu}. If $f$ is~$\cC^p$ and $(p-1)$-non-degenerate then $\esp{\norm{\prsc{\nu}{\phi}}^p}<+\infty$ for all $\phi \in L^\infty(\Omega)$ with compact support.
\end{thm}

\begin{ex}
\label{ex: thm moments Rn}
Let us give some examples of fields satisfying the assumptions of Theorem~\ref{thm: moments Rn}.
\begin{itemize}
\item The Bargmann--Fock field, i.e.,~the smooth stationary Gaussian field on $\R^n$ whose covariance function is $x \mapsto e^{-\frac{\Norm{x}^2}{2}}$, satisfies the hypotheses of Theorem~\ref{thm: moments Rn}.
\item Let $f:\R^n \to \R$ be a stationary $\cC^p$ centered Gaussian field. If the support of its spectral measure has non-empty interior then $f$ is $(p-1)$-non-degenerate.
\item In codimension~$r$, if $(f_i)_{1 \leq i \leq r}$ are $r$ independent $(p-1)$-non-degenerate $\cC^p$ Gaussian fields then so is $f=(f_1,\dots,f_r)$.
\item The Berry field, i.e.,~the smooth stationary Gaussian field $f$ on $\R^n$ whose spectral measure is the uniform measure on $\S^{n-1}$, is $1$-non-degenerate but not $2$-non-degenerate. Indeed it almost surely satisfies $\Delta f+f=0$, so that $(f(x),D_xf,D_x^2f)$ is degenerate for all $x \in \R^n$.
\end{itemize}
\end{ex}

We can consider the same question in a more geometric setting. Let $(M,g)$ be a Riemannian manifold of dimension $n \geq 1$ without boundary and let $E \to M$ be a smooth vector bundle of rank $r \in \ssquarebrackets{1}{n}$ over $M$. Let $s$ be a centered Gaussian field on $M$ with values in $E$, in the sense that $s$ is a random section of $E \to M$ such that for all $m \geq 1$ and all $x_1,\dots,x_m \in M$ the random vector $(s(x_1),\dots,s(x_m))$ is a centered Gaussian. We assume that $s$ is almost surely $\cC^1$ and that $\det \var{s(x)} >0$ for all $x \in M$.

As in the Euclidean setting, $Z=s^{-1}(0)$ is almost surely $(n-r)$-rectifiable. As before, we denote by~$\nu$ the random Radon measure on $M$ defined by integrating over $Z$ with respect to the $(n-r)$-dimensional volume measure $\dx \Vol_Z$ induced by $g$. For all $\phi \in L^\infty(M)$ with compact support, we define the linear statistic $\prsc{\nu}{\phi}$ as in Equation~\eqref{eq: def nu}. In this context Definition~\ref{def: p non degeneracy} adapts as follows.

\begin{dfn}[$p$-non-degeneracy for Gaussian sections]
\label{def: p non degeneracy sections}
Let $p \geq 1$ and let $s$ be a $\cC^p$ centered Gaussian field on $M$ with values in $E$. We say that $s$ is \emph{$p$-non-degenerate} if, for all $x \in M$, the centered Gaussian vector $\jet_p(s,x) \in \cJ_p(M,E)_x$ is non-degenerate.
\end{dfn}

\begin{thm}[Finiteness of moments for zeros of Gaussian sections]
\label{thm: moments M}
Let $p\geq 1$, let $s$ be a centered Gaussian field on $M$ with values in $E$ and let $\nu$ be defined as in Equation~\eqref{eq: def nu}. If $s$ is~$\cC^p$ and $(p-1)$-non-degenerate then $\esp{\norm{\prsc{\nu}{\phi}}^p}<+\infty$ for all $\phi \in L^\infty(M)$ with compact support.
\end{thm}

We are aware of the very recent paper~\cite{GS2024} by Gass and Stecconi, in which they prove a result similar to Theorem~\ref{thm: moments Rn}, as well as its analogue for zeros of Gaussian fields on a Riemannian manifold. Their work and ours are independent, and the proofs are different. Their idea is to compare the Kac--Rice densities (see~Section~\ref{subsec: Factorial moment measures and Kac--Rice densities}) of the field $f$ with those of a well-chosen Gaussian polynomial $P$. Then they deduce the result for $f$ from the result for $P$, which is a consequence of Bézout's theorem. Our proof follows a different path, as it relies on the multijet bundle that we defined in Theorem~\ref{thm: multijet bundle}. Our idea is to observe that the zero set of $F:(x_1,\dots,x_p) \mapsto (f(x_1),\dots,f(x_p))$ in the configuration space $\Omega^p \setminus \Delta_p$ is exactly the vanishing locus of the multijet $\mj_p(f,\cdot)$ restricted to $\Omega^p \setminus \Delta_p \subset C_p[\Omega]$. Instead of working with $F$ which degenerates along $\Delta_p$, we work with the field $\mj_p(f,\cdot)$ that we built to be non-degenerate everywhere. Then, we deduce Theorem~\ref{thm: moments Rn} from the Kac--Rice formula for the expectation (see Proposition~\ref{prop: Kac-Rice expectation}) applied to the $p$-multijet of $f$ and a compactness argument.


\subsection{Higher order multijets and holomorphic multijets}
\label{subsec: Higher order multijets and holomorphic multijets}

Let us now discuss two important variations on our main results Theorems~\ref{thm: multijet bundle} and~\ref{thm: moments M}. In Section~\ref{subsec: Multijet bundles}, we said that two functions $f$ and $g$ on $\R^n$ have the same $p$-multijet at a point $\ux=(x_1,\dots,x_p) \in \config$ if and only if $f$ and $g$ have the same value, i.e.,~the same $0$-jet, at $x_i$ for all $i \in \ssquarebrackets{1}{p}$. In a sense, the $p$-multijet of $f$ at $\ux$ is obtained by patching together the $0$-jets of $f$ at each of the $x_i$ in a relevant way. A natural generalization is to define a higher order multijet of $f$ at $\ux$ by patching together the $k$-jets of $f$ at each of the $x_i$. We define such a higher order multijet in Section~\ref{sec: Multijets adapted to a differential operator}. More generally, we define a multijet bundle adapted to a differential operator~$\cD$. The case of higher order multijets corresponds to $\cD=\jet_k$. The analogue of Theorem~\ref{thm: multijet bundle} in this framework is Theorem~\ref{thm: multijet bundle D} below. We use it to prove an analogue of Theorem~\ref{thm: moments M} adapted to $\cD$; see Theorem~\ref{thm: moments M D} for a general statement. In the special case where $\cD=D$ is the standard differential, the statement is the following.

\begin{thm}[Finiteness of moments for critical points]
\label{thm: moments critical points}
Let $M$ be a smooth manifold without boundary. Let $f:M \to \R$ be a centered Gaussian field and let $\nu_D$ denote the counting measure of its critical locus. Let $p \geq 1$, we assume that $f$ is $\cC^{2p}$ and $(2p-1)$-non-degenerate. Then, for all $\phi \in L^\infty_c(M)$, we have $\esp{\norm{\prsc{\nu_D}{\phi}}^p}<+\infty$.
\end{thm}

Another variation on Theorem~\ref{thm: multijet bundle} is to define holomorphic multijets for holomorphic maps. This is done in Section~\ref{sec: Multijets of holomorphic maps}, and more precisely in Theorem~\ref{thm: holomorphic multijet bundle}. This is used to prove a holomorphic version of Theorem~\ref{thm: moments M}. The general statement is given in Theorem~\ref{thm: moments M holo}. For a holomorphic Gaussian field on an open subset of $\C^n$, it takes the following form.

\begin{thm}[Finiteness of moments for zeros of holomorphic Gaussian fields]
\label{thm: moments holo}
Let $\Omega \subset \C^n$ be open and let $f:\Omega \to \C^r$ be a centered holomorphic Gaussian field, where $r \in \ssquarebrackets{1}{n}$. Let $\nu$ be as in Definition~\ref{def: nu}. Let $p \geq 1$, we assume that, for all $x \in \Omega$, the complex Gaussian vector
\begin{equation*}
\parentheses*{\strut f(x),D_xf,\dots,D_x^{p-1}f} \in \bigoplus_{k=0}^{p-1} \sym^k(\C^n) \otimes \C^r
\end{equation*}
is non-degenerate. Then, for all $\phi \in L^\infty_c(\Omega)$, we have $\esp{\norm{\prsc{\nu}{\phi}}^p}<+\infty$.
\end{thm}

Note that Theorems~\ref{thm: moments critical points} and~\ref{thm: moments holo} are not consequences of Theorem~\ref{thm: moments M}. Indeed, if $f:M \to \R$ is a smooth Gaussian field then $Df$ cannot be $1$-non-degenerate because $D^2f$ is symmetric. Similarly, if $M$ is a complex manifold and $s$ is holomorphic, then $s$ is never $1$-non-degenerate because it satisfies the Cauchy--Riemann Equations.

Gass and Stecconi~\cite{GS2024} proved, independently and by a different method, results analogous to Theorems~\ref{thm: moments critical points} and~\ref{thm: moments holo}. Actually, they prove Theorem~\ref{thm: moments critical points} under the weaker and optimal hypotheses that $f$ is $\cC^{p+1}$ and $p$-non-degenerate. The finiteness of the third moment for the number of critical points of a stationary Gaussian field on $\R^d$ was proved in~\cite[Thm.~1.6]{BMM2024}. For holomorphic Gaussian fields in dimension~$n=1$, see~\cite{NS2012}.


\subsection{Organization of the paper}
\label{subsec: Organisation of the paper}

In Section~\ref{sec: Notation partitions and function spaces} we gather useful notation that appears in several parts of the paper. In Section~\ref{sec: Divided differences and Kergin interpolation} we discuss Kergin interpolation, which is a multivariate polynomial interpolation appearing in the definition of multijets. Section~\ref{sec: Evaluation maps and their kernels} is dedicated to evaluations maps on spaces of polynomials, and more precisely the properties of their kernels. We define our multijet bundles and prove Theorem~\ref{thm: multijet bundle} in Section~\ref{sec: Definition of the multijet bundles}. Section~\ref{sec: Application to zeros of Gaussian fields} is concerned with the application of multijets to the finiteness of moments for the zeros of Gaussian fields and the proofs of Theorems~\ref{thm: moments Rn} and~\ref{thm: moments M}. Multijets adapted to a differential operator are discussed in Section~\ref{sec: Multijets adapted to a differential operator}, where we also prove the analogue of Theorem~\ref{thm: moments M} for critical points. Finally, holomorphic multijets are defined in Section~\ref{sec: Multijets of holomorphic maps}, where we prove the analogue of Theorem~\ref{thm: moments M} for holomorphic Gaussian fields.

\paragraph{Acknowledgments.} The authors thank Vincent Borrelli, Nicolas Vichery and Gabriel Rivière for independently suggesting to take a look at various compactifications of configuration spaces.


\section{Notation: partitions and function spaces}
\label{sec: Notation partitions and function spaces}

The goal of this section is to quickly introduce definitions and notation that will appear in different parts of the paper. We gather them here for the reader's convenience.


\subsection{Sets, partitions and diagonals}
\label{subsec: Sets, partitions and diagonals}

In this paper, we denote by $\N$ the set of non-negative integers. Let $a$ and $b \in \N$, we use the following notation for integer intervals $\ssquarebrackets{a}{b}=\squarebrackets{a,b}\cap \N$.

Let $A$ be a non-empty finite set. For simplicity, in all the notation introduced in this section, if $A = \ssquarebrackets{1}{p}$ we allow ourselves to replace $A$ by $p$ in the indices and exponents. We denote by $\norm{A}$ the cardinality of $A$. Let $M$ be any set, we denote by $M^A$ the Cartesian product of $\norm{A}$ copies of $M$ indexed by the elements of~$A$. A generic element of $M^A$ is usually denoted by $\ux = \parentheses{x_a}_{a \in A}$. If $\emptyset \neq B \subset A$, we denote by $\ux_B = \parentheses{x_a}_{a \in B}$.

\begin{dfn}[Large diagonal]
\label{def: large diagonal}
We denote by $\Delta_A$ the large diagonal in $M^A$, that is:
\begin{equation*}
\Delta_A = \brackets*{\parentheses{x_a}_{a \in A}\in M^A \mvert \exists a,b \in A \ \text{such that} \ a\neq b \ \text{and} \ x_a=x_b}.
\end{equation*} 
\end{dfn}

\begin{dfn}[Partitions]
\label{def: partitions}
Let $A$ be a non-empty and finite set, a \emph{partition} of $A$ is a family $\cI=\brackets{I_1,\dots,I_m}$ of non-empty disjoint subsets of $A$ such that $\bigsqcup_{i=1}^m I_i = A$. The subsets $I_1,\dots,I_m$ are called the \emph{cells} of $\cI$. Given $a \in A$, we denote by $\squarebrackets{a}_\cI$ the only cell of $\cI$ that contains $a$. Finally, we denote by $\cP_A$ the set of partitions of $A$.
\end{dfn}

\begin{dfn}[Clustering partition]
\label{def: clustering partition}
Let $\ux = \parentheses{x_a}_{a \in A} \in M^A$, we denote by $\cI(\ux) \in \cP_A$ the only partition such that, for all $a$ and $b \in A$ we have $x_a=x_b$ if and only if $\squarebrackets{a}_{\cI(\ux)}=\squarebrackets{b}_{\cI(\ux)}$.
\end{dfn}

\begin{ex}
\label{ex: clustering partition}
If $\ux=(x,\dots,x)$ then $\cI(\ux) = \brackets{A}$. If $\ux \in M^A \setminus \Delta_A$ then $\cI(\ux) = \brackets*{\brackets{a}\mvert a \in A}=\cI_0$.
\end{ex}

\begin{dfn}[Strata of the diagonal]
\label{def: strata of the diagonal}
For all $\cI \in \cP_A$, we set $\Delta_{A,\cI} = \brackets*{\ux \in M^A \mvert \cI(\ux)=\cI}$, so that $\Delta_{A,\cI_0}= M^A \setminus \Delta_A$ and $\Delta_A = \bigsqcup_{\cI \neq \cI_0} \Delta_{A,\cI}$.
\end{dfn}

\begin{dfn}[Diagonal inclusions]
\label{def: diagonal inclusions}
Let $\cI \in \cP_A$, we denote by $\iota_\cI: M^\cI \setminus \Delta_\cI \to \Delta_{A,\cI}$ the bijection defined by $\iota_\cI\parentheses*{\strut (y_I)_{I \in \cI}} = \parentheses*{y_{\squarebrackets{a}_\cI}}_{a \in A}$.
\end{dfn}


\subsection{Spaces of functions, sections and jets}
\label{subsec: Spaces of functions, sections and jets}

We use the following multi-index notation. Let $\alpha = (\alpha_1,\dots,\alpha_n) \in \N^n$, we denote its length by $\norm{\alpha} = \alpha_1+\dots+\alpha_n$. Let $\partial_i$ denote the $i$-th partial derivative in some product space, we denote by $\partial^\alpha = \partial_1^{\alpha_1} \cdots \partial_n^{\alpha_n}$. Finally, if $X=(X_1,\dots,X_n)$, we denote by $X^\alpha = X_1^{\alpha_1} \cdots X_n^{\alpha_n}$.

\begin{dfn}[Polynomials]
\label{def: polynomials}
We denote by $\R_d[X]$ the space of real polynomials in $n$ variables of degree at most $d$, where  $d \in \N$ and $X = (X_1,\dots,X_n)$ is multivariate.
\end{dfn}

\begin{dfn}[Symmetric forms and differentials]
\label{def: symmetric forms and differentials}
Let $k \in \N$, we denote by $\sym^k(\R^n)$ the space of symmetric $k$-linear forms on $\R^n$. Let $V$ be a finite-dimensional real vector space, then $\sym^k(\R^n) \otimes V$ is the space of symmetric $k$-linear maps from $\R^n$ to $V$. Given a $\cC^k$ map $f:\R^n \to V$, we denote by $D_x^kf \in \sym^k(\R^n) \otimes V$ its $k$-th differential at $x \in \R^n$.
\end{dfn}

Let $M$ and $N$ be two manifolds without boundary, for all $k \in \N \cup \brackets{\infty}$, we denote by $\cC^k(M,N)$ the space of $\cC^k$ maps from $M$ to $N$. If $V=\R$, we drop it from the notation and we simply write~$\cC^k(M)$. We denote by $L^1_\text{loc}(M)$ the space of locally integrable functions on $M$. We denote by $\cC^0_c(M)$ (resp.~$L^\infty_c(M)$) the space of continuous (resp.~$L^\infty$) functions on $M$ with compact support. Finally, for any Borel subset $B \subset M$, we denote by $\one_B:M \to \R$ its indicator function.

Let $E \to M$ be a vector bundle of finite rank over $M$, we denote by $E_x$ the fiber above $x \in M$. For all $k \in \N \cup \brackets{\infty}$, we denote by $\Gamma^k(M,E)$ the space of $\cC^k$ sections of $E \to M$.

\begin{dfn}[Jets]
\label{def: jets}
Let $k \in \N$, we denote by $\cJ_k(M,E) \to M$ the vector bundle of $k$-jets of sections of $E \to M$. If $E = V \times M$ is trivial with fiber $V$, we denote its $k$-jet bundle by $\cJ_k(M,V) \to M$. If $V=\R$, we simply write $\cJ_k(M) \to M$. Given $s \in \Gamma^k(M,E)$, we denote by $\jet_k(s,x) \in \cJ_k(M,E)_x$ its $k$-jet at $x \in M$.
\end{dfn}


\section{Divided differences and Kergin interpolation}
\label{sec: Divided differences and Kergin interpolation}

An important step in our construction of a multijet for $\cC^k$ functions is to reduce the problem to that of defining a multijet for polynomials. This is done by polynomial interpolation. In several variables, polynomial interpolation is rather ill-behaved, at least compared with the one variable case. However, a multivariate polynomial interpolation suiting our needs was defined by Kergin in~\cite{Ker1980}. A constructive version of his proof was then given in~\cite{MM1980}, using a multivariate version of the so-called divided differences. In this section, we give the definitions of these objects and recall their relevant properties. We refer to the survey~\cite{Lor2000} for more background on polynomial interpolation in $\R^n$.


\subsection{Divided differences}
\label{subsec: Divided differences}

In this section, we recall the definition of multivariate divided differences; see~\cite{MM1980}. Let $k \in \N$, we denote by $\sigma_k$ the standard simplex of dimension $k$, that is:
\begin{equation}
\label{eq: def k simplex}
\sigma_k = \brackets*{\ut = \parentheses*{t_0,\dots,t_k} \in [0,1]^{k+1} \mvert \sum_{i=0}^k t_i=1} \subset \R^{k+1}.
\end{equation}
The simplex $\sigma_k$ is a subset of $\brackets{\underline{t} \in \R^{k+1} \mid \sum t_i = 1}$, and we denote by $\nu_k$ the ($k$-dimensional) Lebesgue measure on this hyperplane, normalized so that $\nu_k(\sigma_k) = \frac{1}{k!}$. One can check that its restriction to $\sigma_k$ satisfies:
\begin{equation}
\label{eq: def nuk}
\int_{\sigma_k} \phi(\ut) \dx \nu_k(\ut) = \int_{\substack{t_1,\dots,t_k \geq 0 \\ \sum_{i=1}^k t_i \leq 1}} \phi\parentheses*{1-\sum_{i=1}^k t_i,t_1,\dots,t_k} \dx t_1 \dots \dx t_k,
\end{equation}
where $\dx t_1 \dots \dx t_k$ is the Lebesgue measure on $\R^k$. For any $\ux=\parentheses*{x_0,\dots,x_k} \in \parentheses*{\R^n}^{k+1}$, we denote by $\sigma(\ux)$ the convex hull of the $x_i$ and we define $\upsilon_{\ux}:\ut \mapsto \sum_{i=0}^k t_ix_i$ from $\sigma_k$ onto $\sigma(\ux)$. Recalling Definition~\ref{def: symmetric forms and differentials}, we have the following.

\begin{dfn}[Divided differences]
\label{def: divided differences}
Let $\ux = \parentheses*{x_i}_{0 \leq i \leq k}\in (\R^n)^{k+1}$ and let $f$ be a $\cC^k$ function defined on some open neighborhood of $\sigma(\ux)$ in $\R^n$. We define the \emph{divided difference} of $f$ at $\ux$ by:
\begin{equation*}
f[x_0,\dots,x_k] = \int_{\sigma_k} D_{\upsilon_{\ux}(\ut)}^kf \dx \nu_k(\ut) \in \sym^k(\R^n),
\end{equation*}
that is, as the average of $D^kf$ over $\sigma(\ux)$ with respect to the pushed-forward measure $\parentheses{\upsilon_{\ux}}_*(\nu_k)$.
\end{dfn}

\begin{rem}
\label{rem: divided differences}
\begin{itemize}
\item If $\ux = \parentheses{x,\dots,x}$ for some $x \in \R^n$ then $f[x,\dots,x] = \frac{1}{k!}D^k_xf$.
\item Definition~\ref{def: divided differences} is invariant under permutation of $\parentheses{x_0,\dots,x_k}$.
\item When $n=1$, Definition~\ref{def: divided differences} coincides with the classical definition of divided differences, under the canonical isomorphism $\sym^k(\R) \simeq \R$. This is known as the Hermite--Genocchi formula~\cite{MM1980}.
\end{itemize}
\end{rem}

\begin{lem}[Regularity of divided differences]
\label{lem: regularity of divided differences}
For all $\ux \in (\R^n)^{k+1}$, the map $f \mapsto f[x_0,\dots,x_k]$ is linear. Moreover, if $f$ is of class $\cC^{k+l}$ then $\ux \mapsto f[x_0,\dots,x_k]$ is of class $\cC^l$.
\end{lem}

\begin{proof}
The linearity with respect to $f$ is clear. The regularity with respect to $\ux$ is obtained by derivation under the integral, using Definition~\ref{def: divided differences} and Equation~\eqref{eq: def nuk}.
\end{proof}


\subsection{Kergin interpolation}
\label{subsec: Kergin interpolation}

This section is dedicated to Kergin interpolation. In the following, we recall the construction of Kergin interpolation in Micchelli--Milman~\cite{MM1980}, which relies on the divided differences introduced in Definition~\ref{def: divided differences}. We will use the notation introduced in Definition~\ref{def: polynomials}.

\begin{prop}[Kergin interpolation]
\label{prop: Kergin interpolation}
Let $\ux \in (\R^n)^p$ and let $f$ be a function of class~$\cC^{p-1}$ defined on some neighborhood of $\sigma(\ux)$ in $\R^n$. There exists a unique polynomial $K(f,\ux) \in \R_{p-1}[X]$ such that, for all non-empty $I \subset \ssquarebrackets{1}{p}$, we have $f\squarebrackets*{\ux_I} = \parentheses*{K(f,\ux)}\squarebrackets*{\ux_I}$. Moreover,
\begin{equation}
\label{eq: definition K}
K(f,\ux) = \sum_{k=1}^{p} f[x_1,\dots,x_k]\parentheses*{X-x_1,\dots,X-x_{k-1}}.
\end{equation}
\end{prop}

\begin{proof}
This is the content of~\cite[Thm.~12.5]{BHS1993} for $m=0$. See also~\cite{MM1980}.
\end{proof}

\begin{rem}
\label{rem: Kergin isomorphism}
In particular, Proposition~\ref{prop: Kergin interpolation} implies the following.
\begin{itemize}
\item The restriction of $K(\cdot,\ux)$ to $\R_{p-1}[X]$ is the identity.
\item If $x$ appears with multiplicity at least $k+1$ in $\ux$, then:
\begin{equation*}
D_x^kf = k! f[\underbrace{x,\dots,x}_{k+1 \ \text{times}}] = k! \parentheses*{\strut K(f,\ux)}[\underbrace{x,\dots,x}_{k+1 \ \text{times}}] = D_x^k \parentheses*{\strut K(f,\ux)}.
\end{equation*}
\item The map $P \mapsto \parentheses*{P[x_1,\dots,x_j]}_{1 \leq j \leq p}$ is an isomorphism from $\R_{p-1}[X]$ to $\bigoplus_{j=0}^{p-1} \sym^j(\R^n)$ whose inverse map is given by $(S_j)_{0 \leq j \leq p-1} \mapsto \sum_{j=0}^{p-1} S_j(X-x_1,\dots,X-x_j)$.
\end{itemize}
\end{rem}

\begin{dfn}[Kergin polynomial]
\label{def: Kergin polynomial}
The polynomial $K(f,\ux)$ from Proposition~\ref{prop: Kergin interpolation} is called the \emph{Kergin interpolating polynomial} of $f$ at $\ux$.
\end{dfn}

\begin{ex}
\label{ex: Kergin polynomial}
If $n=1$, then $K(f,\ux)$ is the Hermite interpolating polynomial of~$f$ at $\ux\in \R^p$. If $\ux=(x,\dots,x)$, then $K(f,\ux)$ is the Taylor polynomial of order $p-1$ of $f$ at $x \in \R^n$.
\end{ex}

\begin{lem}[Regularity of the Kergin polynomial]
\label{lem: regularity of the Kergin polynomial}
For all $\ux \in (\R^n)^p$, the map $K(\cdot,\ux)$ is linear. Moreover, if $f$ is $\cC^{l+p-1}$ then $K(f,\cdot)$ is of class $\cC^l$.
\end{lem}

\begin{proof}
This is a consequence of Lemma~\ref{lem: regularity of divided differences} and Equation~\eqref{eq: definition K}.
\end{proof}

We need to prove a form of compatibility in Kergin interpolation, when the set of interpolation points is refined. We will use this fact to prove that the multijet bundle we define below satisfies Condition~\ref{item: thm multijet localness} in Theorem~\ref{thm: multijet bundle}. The following lemma is stated using the clustering partition $\cI(\ux)$ from Definition~\ref{def: clustering partition}.

\begin{lem}[Compatibility in Kergin interpolation]
\label{lem: compatibility Kergin interpolation}
For all $\ux \in (\R^n)^p$ the linear map from $\R_{p-1}[X]$ to $\prod_{I \in \cI(\ux)}\R_{\norm{I}-1}[X]$ defined by $\parentheses*{\strut K(\cdot,\ux_I)}_{I \in \cI(\ux)}:P \mapsto \parentheses*{\strut K(P,\ux_I)}_{I \in \cI(\ux)}$ is surjective.
\end{lem}

\begin{proof}
Let $\ux \in (\R^n)^p$ and let us write $\cI = \cI(\ux)$ for simplicity. As explained at the end of Section~\ref{subsec: Sets, partitions and diagonals}, there exists a unique $\uy=(y_I)_{I \in \cI} \in (\R^n)^\cI\setminus \Delta_\cI$ such that $\ux = \iota_\cI(\uy)$. Let $(\chi_I)_{I \in \cI}$ be smooth functions on $\R^n$ with pairwise disjoint compact supports and such that $\chi_I$ is equal to $1$ in a neighborhood of $y_I$.

Let $(P_I)_{I \in \cI} \in \prod_{I \in \cI}\R_{\norm{I}-1}[X]$, we consider the function $f = \sum_{I \in \cI} \chi_I P_I\in \cC^\infty(\R^n)$. Let $P = K(f,\ux)$ and let us prove that $K(P,\ux_I) = P_I$ for all $I \in \cI$. For all $k \leq \norm{I}-1$ we have $D^k_{y_I}P = D^k_{y_I}f = D^k_{y_I}P_I$. Indeed $y_I$ appears with multiplicity $\norm{I}$ in $\ux$ (see Remark~\ref{rem: Kergin isomorphism}) and $f$ is equal to $P_I$ in a neighborhood of $y_I$. Recalling Example~\ref{ex: Kergin polynomial}, we know that $K(P,\ux_I)$ is the Taylor polynomial of order $\norm{I}-1$ at $y_I$ of $P$, and hence of $P_I$. Since $P_I \in \R_{\norm{I}-1}[X]$, we get $K(P,\ux_I)=P_I$.
\end{proof}


\section{Evaluation maps and their kernels}
\label{sec: Evaluation maps and their kernels}

The goal of this section is to study evaluation maps on spaces of polynomials and their kernels. Defining multijets is closely related to these objects. Indeed, let $n \geq 1$ and $p \geq 1$ and recall that $\Delta_p$ stands for the large diagonal in $(\R^n)^p$; see Definition~\ref{def: large diagonal}. As explained in the introduction, when $\ux \notin \Delta_p$ we want the multijet of a $\cC^{p-1}$ function $f$ at $\ux$ to be the class of $f$ in $\cC^{p-1}(\R^n)/\sim$, where $f \sim g$ if and only if $\parentheses*{f(x_i)}_{1 \leq i \leq p}=\parentheses*{g(x_i)}_{1 \leq i \leq p}$. The Kergin interpolation of Section~\ref{subsec: Kergin interpolation} shows that any such class can be represented by a polynomial. Hence, the space of $p$-multijets at~$\ux$ is canonically isomorphic to $\R_{p-1}[X] / \ker \ev_{\ux}$, where $\ev_{\ux}:P \mapsto \parentheses{P(x_1),\dots,P(x_p)}$.

\begin{dfn}[Grassmannian]
\label{def: Grassmannian}
Let $V$ be a vector space of finite dimension $N$ and $k \in \ssquarebrackets{0}{N}$, we denote by $\gr[V]{k}$ the \emph{Grassmannian} of vector subspaces of $V$ of \emph{codimension} $k$.
\end{dfn}

\begin{rem}
\label{rem Grassmannian}
Beware that this notation is slightly unusual, since in most textbooks $\gr[V]{k}$ stands for the Grassmannian of subspaces of dimension $k$.
\end{rem}

Let us denote by $\mathcal{L}_\text{reg}\parentheses*{V,\R^k} \subset V^* \otimes \R^k$ the open dense subset of linear surjective maps from $V$ to $\R^k$. The group $GL_k(\R)$ acts on $\mathcal{L}_\text{reg}\parentheses*{V,\R^k}$ by multiplication on the left. On the other hand, $L \mapsto \ker(L)$ defines a surjective map from $\mathcal{L}_\text{reg}\parentheses*{V,\R^k}$ to $\gr[V]{k}$, and $\ker(L_1)=\ker(L_2)$ if and only if there exists $M \in GL_k(\R)$ such that $L_2 = M L_1$. Thus, one can identify $\gr[V]{k}$ with the orbit space $\mathcal{L}_\text{reg}\parentheses*{V,\R^k}/ GL_k(\R)$ of the previous action. This is one of the many ways to describe $\gr[V]{k}$ as a smooth real algebraic manifold.

\begin{dfn}[Evaluation map]
\label{def: evaluation map}
Let $\ux \in (\R^n)^p$, we set $\ev_{\ux}:f \mapsto \parentheses*{f(x_1),\dots,f(x_p)}$ from any space of functions defined at the $x_i$ to $\R^p$. The source space will always be clear from the context.
\end{dfn}

\begin{lem}[Non-degeneracy of $\ev_{\ux}$]
\label{lem: non-degeneracy of evx}
Let $\ux \notin \Delta_p$, then $\ev_{\ux}:\R_{p-1}[X] \to \R^p$ is surjective.
\end{lem}

\begin{proof}
Since $\ux \notin \Delta_p$, we have $\cI(\ux)=\brackets*{\brackets{1},\dots,\brackets{p}}$ and $\ev_{\ux} = \parentheses*{\strut K(\cdot,x_i)}_{1 \leq i \leq p}$ under the canonical identification $\R_0[X] \simeq \R$. Hence this is just a special case of Lemma~\ref{lem: compatibility Kergin interpolation}. Alternatively, in the right basis, one can extract a Vandermonde matrix from that of $\ev_{\ux}$.
\end{proof}

Lemma~\ref{lem: non-degeneracy of evx} shows that the following map is well-defined from $\config$ to $\gr{p}$:
\begin{equation}
\label{eq: def G}
\cG:\ux \longmapsto \ker \ev_{\ux}.
\end{equation}

\begin{lem}[Algebraicity of $\cG$]
\label{lem: algebraicity of G}
The map $\cG:\config \to \gr{p}$ is algebraic.
\end{lem}

\begin{proof}
Recalling the previous discussion, we have $\mathcal{L}_\text{reg}(\R_{p-1}[X],\R^p)/GL_p(\R) \simeq \gr{p}$, where the isomorphism is obtained as the quotient map of $\ker:L \mapsto \ker(L)$. In particular,
\begin{equation*}
\ker:\mathcal{L}_\text{reg}(\R_{p-1}[X],\R^p) \longrightarrow \gr{p}\simeq \mathcal{L}_\text{reg}(\R_{p-1}[X],\R^p)/GL_p(\R)
\end{equation*}
is just the canonical projection, which is algebraic.

Writing $\ev:\ux \mapsto \ev_{\ux}$, we have $\cG = \ker \circ \ev$. Thus it is enough to prove that $\ev$ is algebraic from $\config$ to $\mathcal{L}_\text{reg}(\R_{p-1}[X],\R^p)$. In the basis of $\R_{p-1}[X]$ formed by the monomials $\parentheses*{X^\alpha}_{\norm{\alpha} < p}$, the matrix of $\ev_{\ux}$ is $\begin{pmatrix}
x_i^\alpha
\end{pmatrix}_{1 \leq i \leq p; \norm{\alpha} < p}$, which depends algebraically on $\ux$.
\end{proof}

Let $\ux \in \config$, we defined $\cG(\ux) \in \gr{p}$ by Equation~\eqref{eq: def G}. For any non-empty $I \subset \ssquarebrackets{1}{p}$, we define similarly
\begin{align}
\label{eq: def GI tilde}
\cG_I(\ux) &= \ker \ev_{\ux_I} \in \gr[{\R_{\norm{I}-1}[X]}]{\norm{I}} & &\text{and} & \tilde{\cG}_I(\ux) &= \ker \ev_{\ux_I} \in \gr{\norm{I}}.
\end{align}
Because of the interpolation properties of the Kergin polynomials (see Remark~\ref{rem: Kergin isomorphism}), we have that $\ev_{\ux_I} = \parentheses*{\ev_{\ux_I}}_{\vert \R_{\norm{I}-1}[X]} \circ K(\cdot,\ux_I)$ on $\R_{p-1}[X]$. Hence $\tilde{\cG}_I(\ux) = K(\cdot,\ux_I)^{-1}\parentheses*{\cG_I(\ux)}$. Since $K(\cdot,\ux_I)$ is surjective from $\R_{p-1}[X]$ to $\R_{\norm{I}-1}[X]$, this shows that $\tilde{\cG}_I(\ux)$ has indeed codimension $\norm{I}$, like $\cG_I(\ux)$.

This collection of subspaces satisfies some incidence relations that will be useful in the following. For all non-empty $I \subset \ssquarebrackets{1}{p}$, we have $\cG(\ux) \subset \tilde{\cG}_I(\ux)$. Actually, we can be more precise: for any $\cI \in \cP_p$, we have $\cG(\ux) = \bigcap_{I \in \cI} \tilde{\cG}_I(\ux)$, and this intersection is transverse by a codimension argument.

\begin{rem}
The map $\cG:\config \to \gr{p}$ does not admit an extension as a regular map from $(\R^n)^p$ to $\gr{p}$, except if $n=1$ or $p=1$, that is, if $\gr{p}$ is a point.

For example, when $n=2=p$, the Grassmannian $\gr{p}$ is the set of lines in $\R_1[X_1,X_2]$. Taking $x=R\parentheses{\cos \theta,\sin \theta} \in \R^2 \setminus \brackets{0}$ the reader can check that $\cG(0,x) = \Span\parentheses*{\strut X_1 \sin\theta - X_2 \cos\theta}$ which does not converge as $R \to 0$. However, in this case, $\cG(0,\cdot)$ extends to the blow-up $\bl{\R^2}{0}$ of $\R^2$ at $0$ and similarly $\cG$ extends smoothly to $\bl{(\R^2)^2}{\Delta_2}$. This suggests that, even though $\cG$ does not extend smoothly to $(\R^n)^p$, it might extend to a larger space.
\end{rem}


\section{Definition of the multijet bundles}
\label{sec: Definition of the multijet bundles}

In this section we define the vector bundle $\cMJ_p(\R^n,V) \to C_p[\R^n]$ of $p$-multijets for functions from~$\R^n$ to some finite-dimensional vector space $V$ and prove Theorem~\ref{thm: multijet bundle}. The singularity of $\cG$ along $\Delta_p$ makes it impossible to define such a bundle over $(\R^n)^p$, which is why we define it over a compactification $C_p[\R^n]$ of the configuration space $\config$.

The manifold $C_p[\R^n]$ does not depend on $V$. It is defined in~Section~\ref{subsec: Definition of the basis CpRn}. In the next two sections, we work in the case $V=\R$. All important ideas appear in this case but notation is slightly simpler. In Section~\ref{subsec: Definition of the bundle MPjRn}, we define the bundle $\cMJ_p(\R^n)$. In Section~\ref{subsec: Localness of multijets}, we prove that $p$-multijets are local, in the sense of Condition~\ref{item: thm multijet localness} in Theorem~\ref{thm: multijet bundle}. Finally, we define the bundle $\cMJ_p(\R^n,V)$ of multijets for vector valued maps and prove Theorem~\ref{thm: multijet bundle} in Section~\ref{subsec: Multijets of vector valued maps}.


\subsection{Definition of the basis \texorpdfstring{$C_p[\R^n]$}{}}
\label{subsec: Definition of the basis CpRn}

In this section, we define the basis $C_p[\R^n]$ over which our $p$-multijet bundles are defined. This is a smooth manifold, obtained a compactification of the configuration space $\config$ such that $\parentheses*{\cG_I}_{I \subset \ssquarebrackets{1}{p}}$ extends smoothly to $C_p[\R^n]$. Let us first introduce some notation. We denote by $\Pi_0$ the projection from the product space
\begin{equation*}
(\R^n)^p \times \prod_{\emptyset \neq I \subset \ssquarebrackets{1}{p}} \gr[{\R_{\norm{I}-1}[X]}]{\norm{I}}
\end{equation*}
onto the factor $(\R^n)^p$. Similarly, we denote by $\Pi_I$ the projection onto $\gr[{\R_{\norm{I}-1}[X]}]{\norm{I}}$. Then, let
\begin{equation}
\label{eq: def Sigma}
\Sigma = \brackets*{\parentheses*{\ux, \parentheses*{\strut\cG_I(\ux)}_{I \subset \ssquarebrackets{1}{p}}} \mvert \ux \in \config} \subset (\R^n)^p \times \prod_{\emptyset \neq I \subset \ssquarebrackets{1}{p}} \gr[{\R_{\norm{I}-1}[X]}]{\norm{I}}
\end{equation}
denote the graph of the map $\parentheses*{\cG_I}_{I \subset \ssquarebrackets{1}{p}}$. We denote by $\bar{\Sigma}$ the closure of $\Sigma$ in the product space on the right-hand side of Equation~\eqref{eq: def Sigma}.

\begin{lem}[Surjectivity of \texorpdfstring{$(\Pi_0)_{\vert \bar{\Sigma}}$}{}]
\label{lem: surjectivity of Pi0 on Sigma bar}
Let $\ux \in (\R^n)^p$, then there exists $z \in \bar{\Sigma}$ such that $\Pi_0(z)=\ux$.
\end{lem}

\begin{proof}
Let $\parentheses*{\underline{x_n}}_{n \in \N}$ be a sequence of points in $\config$ converging to $\ux$. Since Grassmannians are compact manifolds, up to extracting subsequences finitely many times, we can assume that for all non-empty $I \subset \ssquarebrackets{1}{p}$ there exists $G_I \in \gr[{\R_{\norm{I}-1}[X]}]{\norm{I}}$ such that $\cG_I\parentheses{\underline{x_n}} \xrightarrow[n \to +\infty]{} G_I$. Then
\begin{equation*}
\parentheses*{\underline{x_n},\parentheses*{\cG_I\parentheses{\underline{x_n}}}_{I \subset \ssquarebrackets{1}{p}}} \xrightarrow[n \to +\infty]{} \parentheses*{\ux, \parentheses*{G_I}_{I \subset \ssquarebrackets{1}{p}}}=z \in \bar{\Sigma}.\qedhere
\end{equation*}
\end{proof}

\begin{lem}[Location of the new points]
\label{lem: location of the new points}
We have $\bar{\Sigma} \setminus \Sigma \subset \Pi_0^{-1}(\Delta_p)$.
\end{lem}

\begin{proof}
Since $\Sigma$ is the graph of a continuous function on $\config$, it is closed in the open subset $\Pi_0^{-1}\parentheses*{\config}$. Hence $\bar{\Sigma} \cap \Pi_0^{-1}\parentheses*{\config} = \Sigma$ and $\bar{\Sigma} \setminus \Sigma \subset \Pi_0^{-1}(\Delta_p)$.
\end{proof}

\begin{lem}[Algebraicity of \texorpdfstring{$\Sigma$ and $\bar{\Sigma}$}{}]
\label{lem: algebraicity of Sigma bar}
The graph $\Sigma$ is a smooth real algebraic manifold and $(\Pi_0)_{\vert \Sigma}:\Sigma \to \config$ is an isomorphism. Moreover, $\bar{\Sigma}$ is a real algebraic variety whose singular locus is contained in $\bar{\Sigma} \setminus \Sigma$.
\end{lem}

\begin{proof}
By Lemma~\ref{lem: algebraicity of G}, the set $\Sigma$ is the graph of an algebraic map, hence a smooth real algebraic manifold. Besides, $\Pi_0$ is algebraic and its restriction to $\Sigma$ is the inverse of $\ux \mapsto \parentheses*{\ux, \parentheses*{\strut\cG_I(\ux)}_{I \subset \ssquarebrackets{1}{p}}}$. Thus $(\Pi_0)_{\vert \Sigma}$ is an algebraic isomorphism from $\Sigma$ onto $\config$. Since $\Sigma$ is real algebraic, so is its closure~$\bar{\Sigma}$. By Lemma~\ref{lem: location of the new points}, we know that $\bar{\Sigma} \cap \Pi_0^{-1}\parentheses*{\config} = \Sigma$ is smooth. Hence, the singular locus of $\bar{\Sigma}$ is contained in $\bar{\Sigma} \cap \Pi_0^{-1}(\Delta_p)= \bar{\Sigma} \setminus \Sigma$.
\end{proof}

\begin{ex}
\label{ex: Sigma bar}
In simple cases, we understand very well what $\bar{\Sigma}$ is.
\begin{itemize}
\item If $p=1$ and $n \geq 1$, then $\Delta_p = \emptyset$ and $\gr{p}=\brackets*{\brackets{0}}$, so that $\bar{\Sigma} = \Sigma = \R^n$.
\item If $n=1$ and $p \geq 1$, then $\gr[{\R_{\norm{I}-1}[X]}]{\norm{I}} = \brackets*{\brackets{0}}$ for all $I \subset \ssquarebrackets{1}{p}$ and $\bar{\Sigma} = \R^p$.
\item If $p=2$ and $n \geq 2$, then for $x \neq y$ in $\R^n$, we know that $\cG(x,y) \subset \R_1[X]$ is the subspace of affine forms on $\R^n$ vanishing at $x$ and $y$, i.e.,~on the affine line through $x$ and $y$. Thus $\cG(x,y)$ encodes this line. As $y \to x$, the accumulation points of $\cG(x,y)$ correspond to all the affine lines passing through $x$, and they encode ``the direction from which $y$ converges to~$x$''. In this case, one can check that $\bar{\Sigma} = \bl{(\R^n)^2}{\Delta_2}$.
\end{itemize}
\end{ex}

In the previous examples the variety $\bar{\Sigma}$ is smooth. Hence the following natural question.

\begin{ques}
Is $\bar{\Sigma}$ smooth for all $n \geq 1$ and $p \geq 1$?
\end{ques}

Lacking a positive answer to this question, since we want $C_p[\R^n]$ to be a smooth manifold, we will define it by resolving the singularities of~$\bar{\Sigma}$. The existence of a resolution of singularities is given by Hironaka's theorem~\cite{Hir1964,Hir1964a}. Our references on this matter are Kollàr~\cite{Kol2007} and Wlodarczyk~\cite{Wlo2005}. See also~\cite{Hau2003} for a softer introduction to this theory.

\begin{prop}[Resolution of singularities]
\label{prop: resolution of singularities}
There exists a smooth manifold $C_p[\R^n]$ without boundary of dimension $np$ and a smooth proper $\Pi:C_p[\R^n] \to (\R^n)^p \times \prod_{\emptyset \neq I \subset \ssquarebrackets{1}{p}} \gr[{\R_{\norm{I}-1}[X]}]{\norm{I}}$ such that:
\begin{enumerate}
\item \label{item: resolution surjective} $\Pi\parentheses*{C_p[\R^n]}=\bar{\Sigma}$;
\item \label{item: resolution dense} $\Pi^{-1}(\Sigma)$ is an open dense subset of $C_p[\R^n]$;
\item \label{item: resolution diffeo} $\Pi_{\vert \Pi^{-1}(\Sigma)}$ is $\cC^\infty$-diffeomorphism from $\Pi^{-1}(\Sigma)$ onto $\Sigma$.
\end{enumerate}
\end{prop}

\begin{proof}
We apply Hironaka's theorem~\cite[Thm.~3.27]{Kol2007} to resolve the singularities of~$\bar{\Sigma}$. Since $\bar{\Sigma}$ is algebraic by Lemma~\ref{lem: algebraicity of Sigma bar}, there exists a smooth real algebraic manifold $C_p[\R^n]$ and a projective morphism $\Pi:C_p[\R^n] \to \bar{\Sigma}$ such that $\Pi$ is an isomorphism over the smooth locus of $\bar{\Sigma}$.

In particular $C_p[\R^n]$ is smooth, the map $\Pi:C_p[\R^n] \to (\R^n)^p \times \prod_{\emptyset \neq I \subset \ssquarebrackets{1}{p}} \gr[{\R_{\norm{I}-1}[X]}]{\norm{I}}$ is smooth and proper, and $\Pi\parentheses*{C_p[\R^n]} \subset \bar{\Sigma}$. Since $\Sigma$ is contained in the smooth locus of $\bar{\Sigma}$, the restriction of $\Pi$ to $\Pi^{-1}(\Sigma)$ is an isomorphism; in particular Condition~\ref{item: resolution diffeo} is satisfied.

According to~\cite[Thm.~1.0.2]{Wlo2005}, the manifold $C_p[\R^n]$ and the projection $\Pi$ are obtained by a sequence of blow-ups along smooth submanifolds that do not intersect the regular locus of $\bar{\Sigma}$, and hence $\Sigma$. This ensures that Conditions~\ref{item: resolution surjective} and~\ref{item: resolution dense} are satisfied.
\end{proof}

The following corollary proves the existence of the manifold $C_p[\R^n]$ and the proper surjection $\pi:C_p[\R^n]\to (\R^n)^p$ satisfying Condition~\ref{item: thm multijet pi} in Theorem~\ref{thm: multijet bundle}.

\begin{cor}[Existence of the basis \texorpdfstring{$C_p[\R^n]$}{}]
\label{cor: existence of the basis CpRn}
There exists a smooth manifold $C_p[\R^n]$ without boundary of dimension $np$ and a smooth proper surjection $\pi:C_p[\R^n] \to (\R^n)^p$ such that:
\begin{enumerate}
\item \label{item: compactification} the open subset $\pi^{-1}\parentheses*{\config}$ is dense in $C_p[\R^n]$ and $\pi$ induces a $\cC^\infty$-diffeomorphism from this set onto $\config$;
\item \label{item: extension of GI} for any non-empty $I \subset \ssquarebrackets{1}{p}$, the map $\cG_I \circ \pi$ admits a unique smooth extension to $C_p[\R^n]$.
\end{enumerate}
\end{cor}

\begin{proof}
We consider $\Pi:C_p[\R^n] \to (\R^n)^p \times \prod_{\emptyset \neq I \subset \ssquarebrackets{1}{p}} \gr[{\R_{\norm{I}-1}[X]}]{\norm{I}}$ given by Proposition~\ref{prop: resolution of singularities} and we let $\pi = \Pi_0 \circ \Pi$. Since Grassmannians are compact, $\Pi_0$ is proper. Hence $\pi$ is smooth and proper because $\Pi$ and $\Pi_0$ are. The surjectivity of $\pi$ is given by Lemma~\ref{lem: surjectivity of Pi0 on Sigma bar} and Condition~\ref{item: resolution surjective} in Proposition~\ref{prop: resolution of singularities}.

Item~\ref{item: compactification} in Corollary~\ref{cor: existence of the basis CpRn} is a consequence of Lemmas~\ref{lem: location of the new points} and~\ref{lem: algebraicity of Sigma bar} and of Conditions~\ref{item: resolution dense} and~\ref{item: resolution diffeo} in Proposition~\ref{prop: resolution of singularities}. Let $I \subset \ssquarebrackets{1}{p}$ be non-empty. On the dense open subset $\pi^{-1}\parentheses*{\config}$ we have $\cG_I \circ \pi = \Pi_I \circ \Pi$ by definition. In the last equality, the right-hand side is well-defined and smooth on~$C_p[\R^n]$, which yields the unique extension we are looking for.
\end{proof}

Since it is defined using Hironaka's theorem, the manifold $C_p[\R^n]$ is not unique. However, the value of the smooth extension of $\cG_I \circ \pi = \Pi_I \circ \Pi$ at $z \in C_p[\R^n]$ only depends on $\Pi(z) \in \bar{\Sigma}$. So this extension does not really depend on the choice of a resolution of singularities. In the following we choose once and for all a realization of $\pi:C_p[\R^n] \to (\R^n)^p$ as in Corollary~\ref{cor: existence of the basis CpRn}. Thanks to Condition~\ref{item: compactification}, we can identify the configuration space $\config$ with its open dense pre-image by $\pi$. Under this identification, Condition~\ref{item: extension of GI} states that the maps $\parentheses*{\cG_I}_{I \subset \ssquarebrackets{1}{p}}$ extend smoothly to~$C_p[\R^n]$. So, from now on, we consider $\cG_I$ as a smooth map from $C_p[\R^n]$ to $\gr[{\R_{\norm{I}-1}[X]}]{\norm{I}}$.


\subsection{Definition of the bundle \texorpdfstring{$\cMJ_p(\R^n)$}{}}
\label{subsec: Definition of the bundle MPjRn}

Now that we have defined the base space $C_p[\R^n]$ of our multijet bundle, we can define the bundle itself. The purpose of this section is to construct the vector bundle $\cMJ_p(\R^n) \to C_p[\R^n]$ of multijets for functions from $\R^n$ to $\R$, and the associated multijet map. The construction for vector valued maps, explained in Section~\ref{subsec: Multijets of vector valued maps}, is basically a fiberwise direct sum of this simpler case.

Recall that we defined the following projections $C_p[\R^n] \xrightarrow[]{\Pi} \bar{\Sigma} \xrightarrow[]{\Pi_0} (\R^n)^p$ and that $\pi = \Pi_0 \circ \Pi$. Thanks to Corollary~\ref{cor: existence of the basis CpRn}, and under the identification discussed above, the map $\cG = \cG_{\ssquarebrackets{1}{p}}$ defined by Equation~\eqref{eq: def G} extends as a smooth map from $C_p[\R^n]$ to $\gr{p}$. Seen as a collection of subspaces of $\R_{p-1}[X]$ indexed by $C_p[\R^n]$, this means that $\cG$ defines a smooth vector sub-bundle of corank $p$ in the trivial bundle $\R_{p-1}[X] \times C_p[\R^n] \to C_p[\R^n]$. We define our multijet bundle as the quotient of this trivial bundle by $\cG$.

\begin{dfn}[Vector bundle of multijets]
\label{def: vector bundle of multijets}
Let $n \geq 1$ and $p\geq 1$, the \emph{vector bundle of multijets of order $p$ on $\R^n$} is the smooth vector bundle of rank $p$ over $C_p[\R^n]$ defined by:
\begin{equation*}
\cMJ_p(\R^n) = \parentheses*{\R_{p-1}[X] \times C_p[\R^n]}/ \cG.
\end{equation*}
In particular, for any $z \in C_p[\R^n]$, the fiber $\cMJ_p(\R^n)_z = \R_{p-1}[X]/\cG(z)$ only depends on $\Pi(z) \in \bar{\Sigma}$.
\end{dfn}

Recalling the definition of Kergin polynomials given in Section~\ref{subsec: Kergin interpolation}, we can now define the $p$-multijet of a $\cC^{p-1}$ function on $\R^n$.

\begin{dfn}[Multijet of a function]
\label{def: multijet of a function}
Let $f \in \cC^{p-1}(\R^n)$ and $z \in C_p[\R^n]$, the \emph{multijet of $f$ at $z$} is the element of $\cMJ_p(\R^n)_z$ defined as:
\begin{equation*}
\mj_p(f,z) = K(f,\pi(z)) \mod \cG(z).
\end{equation*}
In particular, as an element of $\R_{p-1}[X]/\cG(z)$, the multijet $\mj_p(f,z)$ only depends on $\Pi(z) \in \bar{\Sigma}$.
\end{dfn}

\begin{ex}
\label{ex: multijet bundle}
In Example~\ref{ex: Sigma bar} we saw that in simple cases $\bar{\Sigma}$ is smooth. In these cases we set $C_p[\R^n] =\bar{\Sigma}$ and we can describe the bundle $\cMJ_p(\R^n)\to C_p[\R^n]$ and the map $\mj_p$.
\begin{itemize}
\item If $p=1$, then $C_1[\R^n]=\R^n$ and $\cG:x \mapsto \brackets{0} \subset \R_0[X]\simeq \R$. Thus $\cMJ_1(\R^n)$ is the trivial bundle $\R \times \R^n \to \R^n$. Moreover, if $f \in \cC^0(\R^n)$ then $K(f,x) = f(x) \in \R_0[X] \simeq \R$ and $\mj_1(f,x) = f(x)$ for all $x \in \R^n$.

\item If $n=1$, then $C_p[\R]=\R^p$ and $\cG:\ux \mapsto \brackets{0} \subset \R_{p-1}[X]$. Thus $\cMJ_p(\R)$ is the trivial bundle $\R_{p-1}[X] \times \R^p \to \R^p$. If $f \in \cC^{p-1}(\R)$ then $\mj_p(f,\ux)=K(f,\ux)$ is the Hermite interpolating polynomial of $f$ at $\ux$; see Example~\ref{ex: Kergin polynomial}.

Given $\ux=(x_1,\dots,x_p) \notin \Delta_p$, Lemma~\ref{lem: non-degeneracy of evx} shows that $\ev_{\ux}:\R_{p-1}[X] \to \R^p$ is an isomorphism. We can then consider the Lagrange basis $\parentheses*{L_i(\ux)}_{1 \leq i \leq p}$ of $\R_{p-1}[X]$ which is the pre-image by $\ev_{\ux}$ of the canonical basis of $\R^p$. We then have $\mj_p(f,\ux)=K(f,\ux) = \sum_{i=1}^p f(x_i) L_i(\ux)$. Geometrically, this means that the map $(P,\ux) \mapsto \parentheses*{\ev_{\ux}(P),\ux}$ defines a local trivialization of $\cMJ_p(\R)\to C_p[\R]$ over $\R^p \setminus \Delta_p$ and that $\ux \mapsto \parentheses*{L_i(\ux)}_{1 \leq i \leq p}$ is the corresponding frame. Moreover, it is tautological that $\mj_p(f,\ux)$ reads as $\parentheses*{f(x_1),\dots,f(x_p)}$ in this trivialization.

In this example, one can also define a global trivialization of $\cMJ_p(\R)$ by considering the global frame of Newton polynomials $\ux \mapsto \parentheses*{N_k(\ux)}_{1 \leq k \leq p}$, where $N_k(\ux) = \prod_{1 \leq i <k} (X-x_i)$. By Equation~\eqref{eq: definition K} we have $K(f,\ux) = \sum_{k=1}^p f[x_1,\dots,x_k]N_k(\ux)$, so that $\mj_p(f,\ux)$ reads as $\parentheses*{f[x_1,\dots,x_k]}_{1 \leq k \leq p}$ in this trivialization, where the divided differences are the classical ones in dimension~$1$. In this setting, we used in~\cite{AL2021} a strategy that can be roughly summarized as replacing $\parentheses*{f(x_i)}_{1 \leq i \leq p}$ by $\parentheses*{f[x_1,\dots,x_k]}_{1 \leq k \leq p}$. Our present point of view shows that we were actually considering $\mj_p(f,\ux)$ all along, but read in different trivializations.

\item If $p=2$, we saw that $C_2[\R^n] = \bl{(\R^n)^2}{\Delta_2}$. Given $z \in C_2[\R^n]$, if $\pi(z) =(x_1,x_2) \notin \Delta_2$, we know that $\cG(z) \subset \R_1[X]$ is the subspace of affine forms vanishing on the line $L_z \subset \R^n$ through $x_1$ and $x_2$. It is then natural to think of the class of $P$  modulo $\cG(z)$ as its restriction to~$L_z$. Parametrizing $L_z$ by $t \mapsto x_1 + \frac{x_2-x_1}{\Norm{x_2-x_1}}t$, one can check that $P \mapsto P\parentheses*{x_1 + \frac{x_2-x_1}{\Norm{x_2-x_1}}T}$ induces an isomorphism $\cMJ_2(\R^n)_z \simeq \R_1[T] \simeq \R^2$, where $T$ is univariate and the second isomorphism is obtained by reading coordinates in the canonical basis $(1,T)$ of $\R_1[T]$.

Recalling Definition~\ref{def: divided differences}, we have:
\begin{equation*}
P[x_1,x_2]\cdot \frac{x_2-x_1}{\Norm{x_2-x_1}}T = \parentheses*{\int_0^1 D_{x_1 + t(x_2-x_1)}P\cdot(x_2-x_1) \dx t} \frac{T}{\Norm{x_2-x_1}} = \frac{P(x_2)-P(x_1)}{\Norm{x_2-x_1}}T,
\end{equation*}
and $P = K(P,x_1,x_2)$ is given by Equation~\eqref{eq: definition K}. Letting $\tilde{P}(z) = \frac{P(x_2)-P(x_1)}{\Norm{x_2-x_1}}$, we have:
\begin{equation*}
P\parentheses*{x_1 + \frac{x_2-x_1}{\Norm{x_2-x_1}}T} = K(P,x_1,x_2)\parentheses*{x_1 + \frac{x_2-x_1}{\Norm{x_2-x_1}}T} = P(x_1) + \tilde{P}(z)T.
\end{equation*}
Thus, the previous isomorphism $\cMJ_2(\R^n)_z \to \R^2$ is $\parentheses*{P \mod \cG(z)}\mapsto \parentheses*{P(\pi(z)_1),\tilde{P}(z)}$.

Let us now consider $z \in \pi^{-1}(\Delta_2)$. This exceptional divisor is the projectivized normal bundle of $\Delta_2$ in $(\R^n)^2$. So we can think of $z$ as a point in the diagonal, say $(x,x) \in \Delta_2$, and a line in $(\R^n)^2$ which is orthogonal to $\Delta_2$, say spanned by $(u,-u)$ with $u \in \S^{n-1}$. Then $z = \lim_{\epsilon \to 0}\parentheses{x+\epsilon u,x-\epsilon u}$ in $C_2[\R^n]$. By continuity, $\cG(z)$ is the space of affine forms vanishing on the line $L_z \subset \R^n$ parametrized by $t \mapsto x+tu$. As above, mapping $P$ to the coefficients of $P\parentheses*{x+Tu} =P(x) + \parentheses*{D_xP\cdot u}T \in \R_1[T]$ induces an isomorphism $\cMJ_2(\R^n)_z \to \R^2$. Letting $\tilde{P}(z)=D_xP\cdot u$, this isomorphism is again $\parentheses*{P \mod \cG(z)}\mapsto \parentheses*{P(\pi(z)_1),\tilde{P}(z)}$.

Actually, one can check that everything depends smoothly on the base point $z \in C_2[\R^n]$, so that the bundle map 
$\parentheses*{\strut P \mod \cG(z),z} \longmapsto \parentheses*{P(\pi(z)_1),\tilde{P}(z),z}$ defines a global trivialization $\cMJ_2(\R^n) \to \R^2 \times C_2[\R^n]$. If $f \in \cC^1(\R^n)$, with the same notation as above, $\mj_2(f,z)$ reads in this trivialization as $\parentheses*{f(x_1),\frac{f(x_2)-f(x_1)}{\Norm{x_2-x_1}}}$ if $z \notin \pi^{-1}(\Delta_2)$ and as $\parentheses*{f(x),D_xf\cdot u}$ otherwise.
\end{itemize}
\end{ex}

In these examples, the multijet bundle $\cMJ_p(\R^n) \to C_p[\R^n]$ is trivial. This raises the following.

\begin{ques}
Is $\cMJ_p(\R^n) \to C_p[\R^n]$ trivial for all $n \geq 1$ and $p \geq 1$?
\end{ques}

The following two lemmas prove that the bundle map $\mj_p:\cC^{p-1}(\R^n) \times C_p[\R^n] \to \cMJ_p(\R^n)$ satisfies Condition~\ref{item: thm multijet mj} and~\ref{item: thm multijet ev} in Theorem~\ref{thm: multijet bundle}.

\begin{lem}[Regularity of multijets]
\label{lem: regularity of multijets}
The map $\mj_p(\cdot,z):\cC^{p-1}(\R^n)\to \cMJ_p(\R^n)_z$ is a linear surjection for all $z \in C_p[\R^n]$. Additionally, let $l \geq 0$ and let $f\in \cC^{l+p-1}(\R^n)$, then $\mj_p(f,\cdot)$ is a section of class $\cC^l$ of $\cMJ_p(\R^n)\to C_p[\R^n]$.
\end{lem}

\begin{proof}
Let $z \in C_p[\R^n]$, the map $K(\cdot,\pi(z)):\cC^{p-1}(\R^n) \to \R_{p-1}[X]$ is linear by Lemma~\ref{lem: regularity of the Kergin polynomial}. It is also surjective since its restriction to $\R_{p-1}[X]$ is the identity. Since $\mj_p(\cdot,z)$ is the composition of $K(\cdot,\pi(z))$ with the canonical projection from $\R_{p-1}[X]$ onto $\cMJ_p(\R^n)_z$, it is a linear surjection.

Let $l \geq 0$ and let $f \in \cC^{l+p-1}(\R^n)$. By Lemma~\ref{lem: regularity of the Kergin polynomial}, we have $K(f,\cdot)\in \cC^l\parentheses*{(\R^n)^p,\R_{p-1}[X]}$. Since $\pi$ is smooth, we get $K(f,\cdot) \circ \pi \in \cC^l\parentheses*{C_p[\R^n],\R_{p-1}[X]}$. In other words, $K(f,\cdot) \circ \pi$ defines a section of class $\cC^l$ of the trivial bundle $\R_{p-1}[X] \times C_p[\R^n] \to C_p[\R^n]$. Since $\cG$ is a smooth sub-bundle of $\R_{p-1}[X] \times C_p[\R^n]$, projecting onto the quotient bundle $\cMJ_p(\R^n)$ does not decrease the regularity.
\end{proof}

\begin{lem}[Multijets and evaluation]
\label{lem: multijets and evaluation}
Let $z \in C_p[\R^n]$ be such that $\pi(z)=(x_1,\dots,x_p) \notin \Delta_p$. Then for all $f \in \cC^{p-1}(\R^n)$ we have $\mj_p(f,z)=0$ if and only if, for all $i \in \ssquarebrackets{1}{p}$, $f(x_i)=0$.
\end{lem}

\begin{proof}
Let us denote by $\ux = (x_1,\dots,x_p) = \pi(z) \notin \Delta_p$. For all $f \in \cC^{p-1}(\R^n)$, we have:
\begin{equation*}
\mj_p(f,z)=0 \iff K(f,\ux) \in \cG(\ux) \iff \ev_{\ux} \parentheses*{\strut K(f,\ux)} = 0 \iff \ev_{\ux} \parentheses*{f} = 0,
\end{equation*}
since $K(f,\ux)$ interpolates the values of $f$ on at $x_1,\dots,x_p$ (see Remark~\ref{rem: Kergin isomorphism}).
\end{proof}

Actually, we can describe more precisely the relation between multijets and evaluation outside of the diagonal. This will appear in the proof of Theorem~\ref{thm: moments Rn local} below. Thanks to Lemma~\ref{lem: non-degeneracy of evx}, the smooth bundle map $\ev:(P,\ux) \mapsto \parentheses*{\ev_{\ux}(P),\ux}$ from $\R_{p-1}[X] \times \parentheses*{\config}$ to $\R^p \times \parentheses*{\config}$ is surjective and its kernel is exactly the sub-bundle $\ker \ev = \cG$. Thus it induces a smooth bundle map $\tau:\cMJ_p(\R^n)_{\vert \config} \to \R^p \times \parentheses*{\config}$ defined by $\tau \parentheses*{P \mod \cG(\ux)} = \parentheses*{\ev_{\ux}(P),\ux}$, which is bijective. Thus $\tau$ defines a smooth local trivialization of $\cMJ_p(\R^n)$ over $\config$. Moreover, for all $f \in \cC^{p-1}(\R^n)$ and $z \in C_p[\R^n]$ such that $\ux = \pi(z) \notin \Delta_p$ we have:
\begin{equation*}
\tau\parentheses*{\mj_p(f,z)} = \tau\parentheses*{\strut K(f,\ux) \mod \cG(\ux)} = \parentheses*{\ev_{\ux}\parentheses*{\strut K(f,\ux)},\ux} = \parentheses*{\ev_{\ux}(f),\ux}.
\end{equation*}
Hence $\mj_p(f,z)$ simply reads as $\parentheses*{f(x_1),\dots,f(x_p)}$ in this trivialization.


\subsection{Localness of multijets}
\label{subsec: Localness of multijets}

The goal of this section is to prove that the multijet bundle $\cMJ_p(\R^n) \to C_p[\R^n]$ defined in the previous section satisfies Condition~\ref{item: thm multijet localness} in Theorem~\ref{thm: multijet bundle}. Let $z \in C_p[\R^n]$, let $\ux=\pi(z)$ and let $\cI = \cI(\ux)$ be as in Definition~\ref{def: clustering partition}. As explained in Section~\ref{subsec: Sets, partitions and diagonals}, there is a unique $\uy = \parentheses*{y_I}_{I \in \cI} \in \R^\cI \setminus \Delta_\cI$ such that $\ux = \iota_\cI(\uy)$. Recalling that we dropped $V=\R$ from the notation in the present case, we can restate Condition~\ref{item: thm multijet localness} in Theorem~\ref{thm: multijet bundle} as: there exists $\Theta_z : \prod_{I \in \cI} \cJ_{\norm{I}-1}(\R^n)_{y_I} \to \cMJ_p(\R^n)_z$ a linear surjection such that $\mj_p(f,z) = \Theta_z\parentheses*{\parentheses*{\jet_{\norm{I}-1}(f,y_I)}_{I \in \cI}}$ for all $f \in \cC^{p-1}(\R^n)$.

This property is fundamental. First, it shows that $\mj_p(f,z)$ is obtained by patching together (part of) the jets of order $(\norm{I}-1)$ of $f$ at $y_I$, which justifies the name multijet. Second, it shows that $\mj_p(f,z)$ only depends on the values of $f$ in arbitrarily small neighborhoods of the $y_I$. This is not obvious at all since the definition of $\mj_p(f,z)$ involves divided differences of~$f$, which are obtained by integrating on the whole convex hull $\sigma(\ux)$ of the $x_i$ (see Definition~\ref{def: divided differences}). In particular, it shows that $\mj_p(f,z)$ makes sense even if $f$ is only $\cC^{\norm{I}-1}$ in some neighborhood of $y_I$, for all $I \in \cI$. Hence Definition~\ref{def: multijets} makes sense even if $\Omega$ is not convex.

In the following, we consider what we think of as the $I$-th part of a multijet, where $I \subset \ssquarebrackets{1}{p}$. This is just a variation on what we did in Definitions~\ref{def: vector bundle of multijets} and~\ref{def: multijet of a function} and it is defined as follows.

\begin{dfn}[$I$-multijets]
\label{def: I multijets}
Let $n \geq 1$ and $p\geq 1$ and let $I \subset \ssquarebrackets{1}{p}$ be non-empty, we define the bundle of \emph{$I$-multijets} as the following smooth bundle of rank $\norm{I}$ over $C_p[\R^n]$:
\begin{equation*}
\cMJ_I(\R^n) = \parentheses*{\R_{\norm{I}-1}[X] \times C_p[\R^n]} / \cG_I.
\end{equation*}
Let $f\in\cC^{\norm{I}-1}(\R^n)$ and $z \in C_p[\R^n]$, we define by $\mj_I(f,z) = K(f,\pi(z)_I) \mod \cG_I(z) \in \cMJ_I(\R^n)_z$  the \emph{$I$-multijet} of $f$ at $z$.
\end{dfn}

As explained in Section~\ref{subsec: Definition of the basis CpRn}, for all $\emptyset \neq I \subset \ssquarebrackets{1}{p}$ we have a map $\cG_I:C_p[\R^n] \to \gr[{\R_{\norm{I}-1}[X]}]{\norm{I}}$ extending the one on $\config$. Let $z \in C_p[\R^n]$ and $\ux = \pi(z)$. As in Section~\ref{sec: Evaluation maps and their kernels} we define $\tilde{\cG}_I(z) = K(\cdot,\ux_I)^{-1}\parentheses*{\cG_I(z)}\in \gr{\norm{I}}$, where $K(\cdot,\ux_I):\R_{p-1}[X] \to \R_{\norm{I}-1}[X]$. Note that $\tilde{\cG}_I(z)$ has the same codimension as $\cG_I(z)$ since $K(\cdot,\ux_I)$ is surjective.

\begin{lem}[Compatibility of the \texorpdfstring{$\cG_I$}{}]
\label{lem: compatibility GI}
For all $I \subset \ssquarebrackets{1}{p}$ and $z \in C_p[\R^n]$ we have $\cG(z) \subset \tilde{\cG}_I(z)$.
\end{lem}

\begin{proof}
We know from Section~\ref{sec: Evaluation maps and their kernels} that $\cG(z) \subset \tilde{\cG}_I(z)$ for any $z \in \config \subset C_p[\R^n]$, that is, $K(\strut \cdot,\pi(z)_I)\parentheses{\strut \cG(z)} \subset \cG_I(z)$. This incidence relation is a closed condition. By construction, the subset $\config$ is dense in $C_p[\R^n]$ and both terms in the previous inclusion are continuous with respect to $z$; see Lemma~\ref{lem: regularity of the Kergin polynomial} and Corollary~\ref{cor: existence of the basis CpRn}. Hence the inclusion actually holds for any $z \in C_p[\R^n]$. Thus $\cG(z) \subset \tilde{\cG}_I(z)$ for all $z \in C_p[\R^n]$.
\end{proof}

Let $\emptyset\neq I \subset \ssquarebrackets{1}{p}$, let $z \in C_p[\R^n]$ and $\ux=\pi(z)$. We consider $\mj_I(\cdot,z):\R_{p-1}[X] \to \cMJ_I(\R^n)_z$ from Definition~\ref{def: I multijets}. This linear map is surjective as the composition of $K(\cdot,\ux_I)$ and the projection modulo $\cG_I(z)$. Moreover, $\ker(\mj_I(\cdot,z)) = \tilde{\cG}_I(z)$ contains $\cG(z)$ by Lemma~\ref{lem: compatibility GI}. Hence, $\mj_I(\cdot,z)$ induces a surjective linear map from $\cMJ_p(\R^n)_z = \R_{p-1}[X]/\cG(z)$ onto $\cMJ_I(\R^n)_z$ that we still denote by $\mj_I(\cdot,z)$. This is summarized in the following commutative diagram, where the vertical arrows are the canonical projections and all arrows are surjective.
\begin{equation}
\label{eq: diagram mjI}
\begin{tikzcd}[column sep=huge]
\R_{p-1}[X]
	\arrow{r}{K(\cdot,\ux_I)}
	\arrow{d}
	\arrow[sloped]{rd}{\mj_I(\cdot,z)} &
\R_{\norm{I}-1}[X]
	\arrow{d} \\
\cMJ_p(\R^n)_z
	\arrow{r}{\mj_I(\cdot,z)} &
\cMJ_I(\R^n)_z
\end{tikzcd}
\end{equation}

Note that $(P,z) \mapsto \parentheses*{K(P,\pi(z)_I),z}$ is a smooth bundle map over $C_p[\R^n]$ from $\R_{p-1}[X] \times C_p[\R^n]$ to $\R_{\norm{I}-1}[X] \times C_p[\R^n]$. Hence, the previous diagram~\eqref{eq: diagram mjI} defines a smooth surjective bundle map $\mj_I:\parentheses*{P  \mod \cG(z)} \mapsto \parentheses*{P \mod \cG_I(z)}$ from $\cMJ_p(\R^n)$ to $\cMJ_I(\R^n)$ over $C_p[\R^n]$.

\begin{dfn}[Partitioned multijet]
\label{def: partitioned multijet}
For all $\cI \in \cP_p$ and $z \in C_p[\R^n]$, we define a linear map from $\cMJ_p(\R^n)_z$ to $\prod_{I \in \cI} \cMJ_I(\R^n)_z$ by $\mj_\cI(\cdot,z): \alpha \mapsto \parentheses*{\mj_I(\alpha,z)}_{I \in \cI}$.
\end{dfn}

As above, $\mj_\cI:(\alpha,z) \mapsto \mj_\cI(\alpha,z)$ defines a smooth bundle map over $C_p[\R^n]$ from $\cMJ_p(\R^n)$ to $\bigoplus_{I \in \cI} \cMJ_I(\R^n)$, which is obtained as the quotient of $(P,z) \mapsto \parentheses*{\parentheses*{K(P,\pi(z)_I}_{I \in \cI},z}$. However, $\mj_\cI(\cdot,z)$ is not surjective in general. The following lemma proves its surjectivity in some cases.

\begin{lem}[Splitting of multijets]
\label{lem: splitting of multijets}
Let $z \in C_p[\R^n]$, let $\ux =\pi(z)$ and let $\cI(\ux)$ be defined as in Definition~\ref{def: clustering partition}. Then $\mj_{\cI(\ux)}(\cdot,z):\cMJ_p(\R^n)_z \to \prod_{I \in \cI(\ux)}\cMJ_I(\R^n)_z$ is an isomorphism.
\end{lem}

\begin{proof}
The map $\mj_{\cI(\ux)}(\cdot,z)$ is linear between two spaces of the same dimension $p=\sum_{I \in \cI(\ux)} \norm{I}$, so it is enough to prove its surjectivity. Let $\parentheses*{\alpha_I}_{I \in \cI(\ux)} \in \prod_{I \in \cI(\ux)} \cMJ_I(\R^n)_z$, for each $I \in \cI(\ux)$ there exists $P_I \in \R_{\norm{I}-1}[X]$ such that $\alpha_I = P_I \mod \cG_I(z)$. By Lemma~\ref{lem: compatibility Kergin interpolation}, there exists $P \in \R_{p-1}[X]$ such that $K(P,\ux_I) = P_I$ for all $I \in \cI(\ux)$. Let $\alpha=P \mod \cG(z) \in \cMJ_p(\R^n)_z$. Then, for all $I \in \cI(\ux)$, we have:
\begin{equation*}
\mj_I\parentheses*{\alpha,z} = \mj_I(P,z) = K(P,\ux_I) \mod \cG_I(z)= P_I \mod \cG_I(z) = \alpha_I.
\end{equation*}
Hence $\mj_{\cI(\ux)}(\alpha,z) = \parentheses*{\alpha_I}_{I \in \cI(\ux)}$, and $\mj_{\cI(\ux)}(\cdot,z)$ is indeed surjective.
\end{proof}

Let $k \in \N$ and let $x \in \R^n$. By definition, two maps $f$ and $g \in \cC^k(\R^n)$ have the same $k$-jet at $x$ if and only if they have the same Taylor polynomial of order $k$ at $x$. Let $\ux =(x,\dots,x)$, recalling Example~\ref{ex: Kergin polynomial}, the linear map $K\parentheses*{\cdot,\ux}:\cC^k(\R^n)\to \R_k[X]$ is surjective, and it induces an isomorphism $\cJ_k(\R^n)_x \simeq \R_k[X]$.

Let $z \in C_p[\R^n]$, let $\ux = \pi(z)$, let $\cI =\cI(\ux)$ and let $\parentheses*{y_I}_{I \in \cI}=\iota_{\cI}^{-1}(\ux)$; see Definitions~\ref{def: clustering partition} and~\ref{def: diagonal inclusions}. For all $I \in \cI$, the canonical isomorphism $\cJ_{\norm{I}-1}(\R^n)_{y_I} \simeq \R_{\norm{I}-1}[X]$ allows us to see the projection from $\R_{\norm{I}-1}[X]$ onto $\cMJ_I(\R^n)_z$ as a canonical linear surjection $\varpi_{z,I}:\cJ_{\norm{I}-1}(\R^n)_{y_I} \to \cMJ_I(\R^n)_z$.

\begin{dfn}[Gluing map]
\label{def: gluing map}
Let $z \in C_p[\R^n]$, let $\ux = \pi(z)$ and let $\parentheses*{y_I}_{I \in \cI}=\iota_{\cI}^{-1}(\ux)$, where $\cI =\cI(\ux)$. We define $\varpi_z:\parentheses*{\alpha_I}_{I \in \cI} \mapsto \parentheses*{\varpi_{z,I}(\alpha_I)}_{I \in \cI}$ from $\prod_{I \in \cI} \cJ_{\norm{I}-1}(\R^n)_{y_I}$ to $\prod_{I \in \cI} \cMJ_I(\R^n)_z$. We also define $\Theta_z=\mj_{\cI}(\cdot,z)^{-1} \circ \varpi_z$.
\end{dfn}

We can now check that $\Theta_z$ satisfies Condition~\ref{item: thm multijet localness} in Theorem~\ref{thm: multijet bundle}.

\begin{lem}[Localness of multijets]
\label{lem: localness of multijets}
For all $z \in C_p[\R^n]$, the map $\Theta_z$ is a linear surjection from $\prod_{I \in \cI} \cJ_{\norm{I}-1}(\R^n)_{y_I}$ to $\cMJ_p(\R^n)_z$. Moreover, it is the only map such that
\begin{equation*}
\forall f \in \cC^{p-1}(\R^n), \qquad \Theta_z\parentheses*{\parentheses*{\jet_{\norm{I}-1}(f,y_I)}_{I \in \cI}} = \mj_p(f,z).
\end{equation*}
\end{lem}

\begin{proof}
With the same notation as in Definition~\ref{def: gluing map}, for all $I \in \cI$ the map $\varpi_{z,I}$ is a linear surjection by definition. Hence so is $\varpi_z$. Since $\mj_\cI(\cdot,z)$ is an isomorphism by Lemma~\ref{lem: splitting of multijets}, the map $\Theta_z=\mj_{\cI}(\cdot,z)^{-1} \circ \varpi_z$ is also a linear surjection.

Let $f \in \cC^{p-1}(\R^n)$. For all $I \in \cI$, the image of $\jet_{\norm{I}-1}(f,y_I)$ under the canonical isomorphism $\cJ_{\norm{I}-1}(\R^n)_{y_I} \simeq \R_{\norm{I}-1}[X]$ is the Taylor polynomial $K(f,\ux_I)$. Hence
\begin{equation*}
\varpi_{z,I}\parentheses*{\jet_{\norm{I}-1}(f,y_I)}= K(f,\ux_I) \mod \cG_I(z) = \mj_I\parentheses*{K(f,\ux_I),z}.
\end{equation*}
Thus, we have $\varpi_z\parentheses*{\parentheses*{\jet_{\norm{I}-1}(f,y_I)}_{I \in \cI}} = \parentheses*{\strut \mj_I\parentheses*{K(f,\ux_I),z}}_{I \in \cI} = \mj_\cI\parentheses*{\mj_p(f,z),z}$, and finally $\Theta_z\parentheses*{\parentheses*{\jet_{\norm{I}-1}(f,y_I)}_{I \in \cI}} = \mj_p(f,z)$.

Since the $y_I$ are pairwise distinct, every element of $\prod_{I \in \cI} \cJ_{\norm{I}-1}(\R^n)_{y_I}$ can be realized as $\parentheses*{\jet_{\norm{I}-1}(f,y_I)}_{I \in \cI}$ for some $f \in \cC^\infty(\R^n)$. Hence the previous relation completely defines $\Theta_z$.
\end{proof}


\subsection{Multijets of vector valued maps}
\label{subsec: Multijets of vector valued maps}

So far we have only defined multijets of real-valued functions. In this section, we extend the previous construction to maps from $\R^n$ to some vector space $V$ of dimension $r \geq 1$. Let $\pi:C_p[\R^n] \to (\R^n)^p$ be given by Corollary~\ref{cor: existence of the basis CpRn} as before. We define $\cMJ_p(\R^n,V)\to C_p[\R^n]$ as the tensor product of the bundle $\cMJ_p(\R^n)\to C_p[\R^n]$ from Definition~\ref{def: vector bundle of multijets} with the trivial bundle $V \times C_p[\R^n] \to C_p[\R^n]$.

\begin{dfn}[Multijet bundle of vector valued maps]
\label{def: multijet bundle of vector valued maps}
Let $n \geq 1$ and $p\geq 1$; let $V$ be a real vector space of dimension $r \geq 1$, we define the \emph{bundle of $p$-multijets of $V$-valued maps on $\R^n$} as the following smooth bundle of rank $pr$ over $C_p[\R^n]$:
\begin{equation*}
\cMJ_p(\R^n,V) = \cMJ_p(\R^n) \otimes V.
\end{equation*}
\end{dfn}

\begin{dfn}[Multijet of a map]
\label{def: multijet of a map}
Let $(v_1,\dots,v_r)$ denote a basis of $V$. Let $z \in C_p[\R^n]$ and let $f = \sum_{i=1}^r f_i v_i \in \cC^{p-1}(\R^n,V)$, we define by $\mj_p(f,z) = \sum_{i=1}^r \mj_p(f_i,z) \otimes v_i \in \cMJ_p(\R^n,V)_z$ the $p$-multijet of $f$ at $z$.
\end{dfn}

\begin{lem}[Independence from the basis]
\label{lem: multijet intrinsic}
In Definition~\ref{def: multijet of a map}, the vector $\mj_p(f,z)$ does not depend on the choice of the basis $(v_1,\dots,v_r)$.
\end{lem}

\begin{proof}
Let $(w_1,\dots,w_r)$ be another basis of $V$. There exists a matrix $\parentheses*{a_{ij}}_{1 \leq i,j \leq r} \in GL_r(\R)$ such that $v_i = \sum_{j=1}^r a_{ij} w_j$ for all $i \in \ssquarebrackets{1}{r}$. Letting $g_j = \sum_{i=1}^r a_{ij}f_i$ for all $j \in \ssquarebrackets{1}{r}$, we get
\begin{equation*}
f= \sum_{i=1}^r f_i v_i = \sum_{1 \leq i,j \leq r} a_{ij}f_i w_j = \sum_{j=1}^r g_j w_j.
\end{equation*}
Then, by linearity of the $p$-multijet for functions, we have:
\begin{equation*}
\sum_{j=1}^r \mj_p(g_j,z) \otimes w_j = \sum_{1 \leq i,j \leq r} a_{ij} \mj_p(f_i,z) \otimes w_j = \sum_{i=1}^r \mj_p(f_i,z) \otimes v_i.\qedhere
\end{equation*}
\end{proof}

\begin{ex}
\label{ex: multijet Rr}
If $V=\R^r$, then for all $z \in C_p[\R^n]$ and $f=(f_1,\dots,f_r) \in \cC^{p-1}(\R^n,\R^r)$ we have
\begin{equation*}
\cMJ_p(\R^n,\R^r)_z = \cMJ_p(\R^n)_z \otimes \R^r \simeq \parentheses*{\cMJ_p(\R^n)_z}^r,
\end{equation*}
and under this canonical isomorphism $\mj_p(f,z) = \parentheses*{\mj_p(f_i,z)}_{1 \leq i \leq r}$, as one would expect.
\end{ex}

Let $x \in \R^n$ and $k \in \N$, we have a canonical isomorphism $\cJ_k(\R^n,V)_x \simeq \cJ_k(\R^n)_x \otimes V$. If $(v_1,\dots,v_r)$ is a basis of $V$, this isomorphism is totally determined by the fact that the image of  $\jet_k(f,x)$ is $\sum_{i=1}^r \jet_k(f_i,x)\otimes v_i$ for all $f= \sum_{i=1}^r f_i v_i \in \cC^k(\R^n,V)$. As in the proof of Lemma~\ref{lem: multijet intrinsic}, this does not depend on the choice of the basis $(v_1,\dots,v_r)$.

\begin{dfn}[Gluing map for vector valued multijets]
\label{def: gluing map for vector valued multijets}
Let $z \in C_p[\R^n]$, let $\ux = \pi(z)$, let $\cI =\cI(\ux)$ and let $\parentheses*{y_I}_{I \in \cI}=\iota_{\cI}^{-1}(\ux)$. Using the previous canonical isomorphisms, we have:
\begin{equation*}
\prod_{I \in \cI} \cJ_{\norm{I}-1}(\R^n,V)_{y_I} = \prod_{I \in \cI} \cJ_{\norm{I}-1}(\R^n)_{y_I} \otimes V = \parentheses*{\prod_{I \in \cI} \cJ_{\norm{I}-1}(\R^n)_{y_I}} \otimes V.
\end{equation*}
Recalling Definition~\ref{def: gluing map}, we define a linear map $\Theta_z:\prod_{I \in \cI} \cJ_{\norm{I}-1}(\R^n,V)_{y_I} \to \cMJ_p(\R^n,V)_z$ by $\Theta_z\parentheses*{\alpha \otimes v} = \Theta_z(\alpha) \otimes v$ for all $\alpha \in \prod_{I \in \cI} \cJ_{\norm{I}-1}(\R^n)_{y_I}$ and all $v \in V$.
\end{dfn}

We now have everything we need to prove Theorem~\ref{thm: multijet bundle}.

\begin{proof}[Proof of Theorem~\ref{thm: multijet bundle}]
The base space $C_p[\R^n]$ and the projection $\pi$ are given by Corollary~\ref{cor: existence of the basis CpRn}. In particular, they satisfy Condition~\ref{item: thm multijet pi} in Theorem~\ref{thm: multijet bundle}. Definition~\ref{def: multijet of a map} and Lemma~\ref{lem: regularity of multijets} show that $\mj_p$ satisfies Condition~\ref{item: thm multijet mj}. Similarly, Condition~\ref{item: thm multijet ev} is satisfied thanks to Lemma~\ref{lem: multijets and evaluation} and the definition of $\mj_p$ for $V$-valued maps.

Let us check that the linear maps $\Theta_z$ from Definition~\ref{def: gluing map for vector valued multijets} satisfy Condition~\ref{item: thm multijet localness}. Let $z \in C_p[\R^n]$, let $\ux = \pi(z)$, let $\cI =\cI(\ux)$ and let $\parentheses*{y_I}_{I \in \cI}=\iota_{\cI}^{-1}(\ux)$. Let us also denote by $(v_1,\dots,v_r)$ a basis of~$V$. Let $\alpha \in \cMJ_p(\R^n,V)_z$, there exists $\alpha_1,\dots,\alpha_r \in \cMJ_p(\R^n)_z$ such that $\alpha = \sum_{i=1}^r \alpha_i \otimes v_i$. By Lemma~\ref{lem: localness of multijets}, for each $i \in \ssquarebrackets{1}{r}$, there exists $\beta_i \in \prod_{I \in \cI} \cJ_{\norm{I}-1}(\R^n)_{y_I}$ such that $\alpha_i = \Theta_z(\beta_i)$. Hence, $\Theta_z \parentheses*{\sum_{i=1}^r \beta_i \otimes v_i} = \sum_{i=1}^r \Theta_z(\beta) \otimes v_i = \alpha$, and $\Theta_z$ is indeed surjective.

Finally, let us consider $f=\sum_{i=1}^r f_i v_i \in \cC^{p-1}(\R^n,V)$. Then, by Lemma~\ref{lem: localness of multijets} once again,
\begin{align*}
\Theta_z\parentheses*{\parentheses*{\jet_{\norm{I}-1}(f,y_I)}_{I \in \cI}} &= \Theta_z\parentheses*{\sum_{i=1}^r \parentheses*{\jet_{\norm{I}-1}(f_i,y_I)}_{I \in \cI} \otimes v_i} = \sum_{i=1}^r \Theta_z\parentheses*{\parentheses*{\jet_{\norm{I}-1}(f_i,y_I)}_{I \in \cI}} \otimes v_i\\
&= \sum_{i=1}^r \mj_p(f_i,z) \otimes v_i = \mj_p(f,z).\qedhere
\end{align*}
\end{proof}

\begin{rem}
\label{rem: typography}
Another way to define $\cMJ_p(\R^n,V)$ and $\mj_p$ is the following. If $f:\R^n \to V$ is regular enough, then the divided differences from Definition~\ref{def: divided differences} still make sense, only this time $f[x_0,\dots,x_k] \in \sym^k(\R^n) \otimes V$. Then, one can still define $K(f,\ux)$ as in Proposition~\ref{prop: Kergin interpolation}, and it defines an element of $\R_{p-1}[X] \otimes V$ that interpolates the divided differences of $f$. Similarly, everything we did from Section~\ref{subsec: Divided differences} to Section~\ref{subsec: Localness of multijets} can be adapted to the case of $V$-valued maps, simply by tensoring each vector space by $V$, and each linear map by $\Id_V$. One can check that we recover the same objects as in Definitions~\ref{def: multijet bundle of vector valued maps} and~\ref{def: multijet of a map}, up to canonical isomorphisms.
\end{rem}


\section{Application to zeros of Gaussian fields}
\label{sec: Application to zeros of Gaussian fields}

This section is concerned with our application of multijet bundles to Gaussian fields. In Section~\ref{subsec: Gaussian vectors and Gaussian sections}, we describe the local model for the Gaussian fields with values in a vector bundle that we consider. In Section~\ref{subsec: Bulinskaya lemma and Kac--Rice formula for the expectation}, we prove a Bulinskaya-type lemma and a Kac--Rice formula for the zeros of these fields. Section~\ref{subsec: Factorial moment measures and Kac--Rice densities} is dedicated to the definition of the Kac--Rice densities of order larger than $2$. We also relate the properties of these functions with the moments of the linear statistics associated with our field. Finally, we prove Theorems~\ref{thm: moments Rn} and~\ref{thm: moments M} in Section~\ref{subsec: Proof of finiteness of moments}, using the multijet bundles defined in Theorem~\ref{thm: multijet bundle}.


\subsection{Gaussian vectors and Gaussian sections}
\label{subsec: Gaussian vectors and Gaussian sections}

In this section, we briefly recall some notation and conventions concerning Gaussian vectors. Then we describe the local model for Gaussian fields with values in a vector bundle.  We will mostly consider centered random vectors in finite-dimensional vector spaces, so we restrict ourselves to this setting. In the following, $V$ is a finite-dimensional real vector space.

\begin{dfn}[Gaussian vector]
\label{def: Gaussian vector}
We say that a random vector $X$ with values in $V$ is a \emph{centered Gaussian vector} if, for all $\eta \in V^*$, the real random variable $\eta(X)$ is a centered Gaussian in $\R$.
\end{dfn}

In particular, a centered Gaussian vector in $V$ has finite moments up to any order. Let us assume that $V$ is endowed with an inner product $\prsc{\cdot}{\cdot}$. Then for all $v \in V$, we define $v^* = \prsc{v}{\cdot} \in V^*$.

\begin{dfn}[Variance operator]
\label{def: variance operator}
Let $X$ be a centered Gaussian vector in $\parentheses*{V,\prsc{\cdot}{\cdot}}$, then its \emph{variance operator} is the non-negative self-adjoint endomorphism $\var{X} = \esp{X \otimes X^*}$ of~$V$. We say that $X$ is \emph{non-degenerate} if $\var{X}$ is invertible.
\end{dfn}

Recall that a centered Gaussian vector in $\parentheses*{V,\prsc{\cdot}{\cdot}}$ is completely determined by its variance. In the following, we denote by $\gauss{\Lambda}$ the centered Gaussian distribution of variance $\Lambda$, and by $X \sim \gauss{\Lambda}$ the fact that $X$ follows this distribution.

\begin{dfn}[Gaussian field]
\label{def: Gaussian field}
Let $E \to M$ be a vector bundle over some manifold $M$, we say that a random section $s$ of $E \to M$ is a \emph{centered Gaussian field} if for all $m \geq 1$ and all $x_1,\dots,x_m$ the random vector $\parentheses{s(x_1),\dots,s(x_m)}$ is a centered Gaussian. We say that this field is \emph{non-degenerate} if $s(x)$ is non-degenerate for all $x \in M$.
\end{dfn}

If the centered Gaussian field $s$ is $\cC^p$, then its jet $\jet_k(s,x)$ is a centered Gaussian for all $x \in M$. Thus, the definition of $p$-non-degeneracy of the field makes sense; see~Definition~\ref{def: p non degeneracy sections}. Note that $0$-non-degenerate simply means non-degenerate.

Since this will appear in several places later on, let us describe the local model for Gaussian fields in this context. Let $x_0 \in M$, there exists a chart $(U,\varphi)$ of $M$ around $x_0$. That is, $\varphi:U \to \Omega$ is a diffeomorphism between an open neighborhood $U$ of $x_0$ and an open subset $\Omega \subset \R^n$. Up to reducing $U$, we can assume that $E$ is trivial over $U$, i.e.,~there exists a trivialization $\tau:E_{\vert U} \to \R^r \times U$. Letting $\tau_\varphi = (\Id,\varphi)\circ \tau$, we have the following commutative diagram, where arrows on the top row are bundle maps covering the maps on the bottom row.

\begin{equation}
\label{eq: diagram trivialization}
\begin{tikzcd}[column sep=huge]
E_{\vert U}
	\arrow{r}{\tau}
	\arrow[out=30, in=150, swap]{rr}{\tau_\varphi}
	\arrow{d}&
\R^r \times U
	\arrow{r}{(\Id,\varphi)}
	\arrow{d}&
\R^r \times \Omega
	\arrow{d}\\
U
	\arrow{r}{\Id}&
U
	\arrow{r}{\varphi}&
\Omega
\end{tikzcd}
\end{equation}

Let $s$ be a local section of $E_{\vert U}$, then $\tau_\varphi \circ s \circ \varphi^{-1}$ is a section of the trivial bundle on the right-hand side of~\eqref{eq: diagram trivialization}. Hence there exists a map $f:\Omega \to \R^r$ such that $\tau_\varphi \circ s \circ \varphi^{-1}=\parentheses*{f,\Id}$. For all $x \in \Omega$, the vector $f(x)$ is the image of $s(\varphi^{-1}(x))$ by a linear bijection. Thus, if $s:M \to E$ is a centered Gaussian field, its restriction to $U$ corresponds to a centered Gaussian field $f:\Omega \to \R^r$. Moreover, $f$ has the same regularity as $s$.

The local trivializations on the diagram~\eqref{eq: diagram trivialization} induce a similar picture for jet bundles so that we have a local trivialization $\cJ_p(U,E_{\vert U}) \simeq \cJ_p(\Omega,\R^r)$, under which $\jet_p(f,x)$ corresponds to $\jet_p(s,\varphi^{-1}(x))$. Thus, the Gaussian section $s$ is $p$-non-degenerate in the sense of Definition~\ref{def: p non degeneracy sections} if and only if $f$ is $p$-non-degenerate in the sense of Definition~\ref{def: p non degeneracy}. If this is the case, up to replacing $\Omega$ by a smaller $\Omega'$ such that $\bar{\Omega'} \subset \Omega$ is compact, we can assume that $f$ is uniformly $p$-non-degenerate, in the sense that $\det\var{\jet_p(f,x)}$ is bounded from below on $\Omega$. This local picture is summarized in the following lemma.

\begin{lem}[Local model for Gaussian fields]
\label{lem: local picture}
Let $s:M\to E$ be a centered $p$-non-degenerate Gaussian field. For all $x_0 \in M$, there exist an open neighborhood $U$ of $x_0$ and a local trivialization of the form~\eqref{eq: diagram trivialization} such that $s$ reads in local coordinates as a centered Gaussian field $f:\Omega \to \R^r$ of the same regularity as~$s$ and which is uniformly $p$-non-degenerate on $\Omega$.
\end{lem}


\subsection{Bulinskaya lemma and Kac--Rice formula for the expectation}
\label{subsec: Bulinskaya lemma and Kac--Rice formula for the expectation}

Let $(M,g)$ be a Riemannian manifold of dimension $n\geq 1$ without boundary and let $E \to M$ be a vector bundle of rank $r \in \ssquarebrackets{1}{n}$. We consider a non-degenerate centered Gaussian field $s:M\to E$, in the sense of Definition~\ref{def: Gaussian field}. The goal of this section is to state a Bulinskaya-type lemma and a Kac--Rice formula for the expectation of the linear statistics of~$s$.

When $M$ is an open subset of $\R^n$ and $E=\R^r \times M$ is trivial, these results are proved by Armentano--Azaïs--Leòn in~\cite[Prop.~2.1 and Thm.~2.2]{AAL2023}. They are extended to fields on submanifolds of $\R^N$ in~\cite[Sect.~9.1]{AAL2023}. In the following we check the results of~\cite{AAL2023} can be adapted to the case of Gaussian sections.

\begin{rem}
\label{rem: complicated setting}
Some readers are most interested in the zeros of Gaussian fields from $\R^n$ to $\R^r$ and the present geometric setting may seem overly complicated to them. Let us stress that, even in the simpler setting of fields from $\R^n$ to $\R^r$, our proof of Theorem~\ref{thm: moments Rn} uses the Kac--Rice formula in the more general setting we are studying here.
\end{rem}

In order to state the Bulinskaya lemma and the Kac--Rice formula, we need the following.

\begin{dfn}[Jacobian determinant]
\label{def: Jacobian}
Let $L:V \to V'$ be a linear map between Euclidean spaces and let $L^*$ denote its adjoint map. The \emph{Jacobian} of $L$ is defined as $\jac{L}=\det\parentheses{LL^*}^\frac{1}{2}$.
\end{dfn}

\begin{rem}
\label{rem: Jacobian}
We have $\jac{L} \geq 0$, and $\jac{L}>0$ if and only if $L$ is surjective. In particular, the fact that $\jac{L}=0$ depends only on $L$ and not on the Euclidean structures on $V$ and $V'$. Thus the condition $\jac{L}=0$ makes sense even if no inner product is specified.
\end{rem}

\begin{prop}[Weak Bulinskaya lemma]
\label{prop: weak Bulinskaya}
Let $\nabla$ be a connection on $E \to M$. If the centered Gaussian field $s:M\to E$ is $\cC^1$ and non-degenerate, the $(n-r)$-dimensional Hausdorff measure of $\brackets*{x \in M \mvert s(x)=0 \ \text{and} \ \jac{\nabla_xs}=0}$ is almost surely $0$.
\end{prop}

\begin{rem}
\label{rem: Bulinskaya nabla}
If $s(x)=0$ then $\nabla_xs$ does not depend on $\nabla$. Hence the random set we are interested in Proposition~\ref{prop: weak Bulinskaya} does not depend on the choice $\nabla$.
\end{rem}

\begin{proof}[Proof of Proposition~\ref{prop: weak Bulinskaya}]
We can cover $M$ by countably many open trivialization domains of the type described in Lemma~\ref{lem: local picture}. Then it is enough to prove the result in each of these domains.

Let $U \subset M$ be as Lemma~\ref{lem: local picture}. In local coordinates, the restriction of $s$ reads as a uniformly non-degenerate $\cC^1$ centered Gaussian field $f:\Omega \to \R^r$, where $\Omega \subset \R^n$ is open. Moreover, for any $x \in \Omega$ such that $f(x)=0$, the covariant derivative of $s$ reads as $D_xf$, independently of the choice of $\nabla$. Thus, we are left with proving that the $(n-r)$-dimensional Hausdorff measure of
\begin{equation*}
\brackets*{x \in \Omega \mvert f(x)=0 \ \text{and}\ \jac{D_xf}=0}
\end{equation*}
is almost surely $0$, which is given by~\cite[Prop.~2.1]{AAL2023}.
\end{proof}

Let us assume from now on that the centered Gaussian field $s:M \to E$ is $\cC^1$ and non-degenerate. We denote its zero set by $Z = s^{-1}(0)$. Let us define $Z_\text{sing} = \brackets*{x \in Z \mvert \jac{\nabla_xs}=0}$ and $Z_\text{reg} = Z \setminus Z_\text{sing}$. By Proposition~\ref{prop: weak Bulinskaya}, the $(n-r)$-dimensional Hausdorff measure of the singular part $Z_\text{sing}$ is almost surely $0$. On the other hand, the regular part $Z_\text{reg}$ is the set of points where~$s$ vanishes transversally. As such, it is a (possibly empty) $\cC^1$ submanifold of~$M$ of codimension $r$ without boundary. Thus, $Z$ is almost surely the union of an open (in $Z$) regular part $Z_\text{reg}$ of dimension $n-r$, and a negligible singular part $Z_\text{sing}$ that we can think of as a set of lower dimension. Let us mention that, under additional assumptions on the field, the singular part is almost surely empty.

\begin{prop}[Strong Bulinskaya lemma]
\label{prop: strong Bulinskaya}
If the centered Gaussian field $s:M \to E$ is $\cC^2$ and $1$-non-degenerate, then $Z_\text{sing}=\emptyset$ almost surely.
\end{prop}

\begin{proof}
As in the proof of Proposition~\ref{prop: weak Bulinskaya}, it is enough to prove the result in local coordinates given by Lemma~\ref{lem: local picture}. In these coordinates, $s$ reads as a $\cC^2$ centered Gaussian field $f:\Omega \to \R^r$ which is uniformly $1$-non-degenerate, that is, $\det\var{f(x),D_xf}$ is bounded from below on $\Omega$. Then the result follows from~\cite[Prop.~6.12]{AW2009}.
\end{proof}

The Riemannian metric $g$ induces a metric on $Z_\text{reg}$, which in turn defines an $(n-r)$-dimensional Riemannian volume measure $\dx \Vol_Z$. This measure coincides with the $(n-r)$-dimensional Hausdorff measure on $Z$. In the following, we consider this measure as a Radon measure on $M$ defined as follows. Recall that the space of Radon measures is the topological dual of $\cC^0_c(M)$, and that being a non-negative Radon is equivalent to being a Borel measure which is finite on compact subsets.

\begin{dfn}[Random measure associated with $Z$]
\label{def: nu}
We denote by $\nu$ the random non-negative Radon measure on $M$ defined by:
\begin{equation*}
\forall \phi \in \cC^0_c(M), \qquad \prsc{\nu}{\phi} = \int_{Z_\text{reg}} \phi(x) \dx \Vol_Z(x).
\end{equation*}
We define $\prsc{\nu}{\phi}$ similarly if $\phi$ is non-negative Borel function (in which case $\prsc{\nu}{\phi} \in [0,+\infty]$) or if $\phi$ is a Borel function such that $\prsc{\nu}{\norm{\phi}}<+\infty$ almost surely.
\end{dfn}

\begin{ex}
\label{ex: nu}
If $n=r$ then $Z$ is almost surely locally finite. In this case $\nu = \sum_{x \in Z} \delta_x$ is the random counting measure of this point process.
\end{ex}

Let us go back to the local model of Lemma~\ref{lem: local picture}. Around any $x_0 \in M$ there exists a chart~$(U,\varphi)$ and a local trivialization of the kind described by~\eqref{eq: diagram trivialization}. Since $s$ is $\cC^1$ and non-degenerate, it corresponds in local coordinates to a $\cC^1$ non-degenerate Gaussian field $f:\Omega \to \R^r$. We still denote by $Z$ (resp.~$Z_\text{reg}$) the image of $Z$ (resp.~$Z_\text{reg}$) by $\varphi$, which is the zero set of $f$ (resp.~its regular part). Similarly, we still denote by $g$ (resp.~$\dx \Vol_Z$) the push-forward to $\Omega$ of the metric $g$ (resp.~of the measure $\dx \Vol_Z$), and we identify test-functions on~$U$, with test-functions on $\Omega$. Thus, if $\phi \in L^\infty_c(U)$ we have:
\begin{equation*}
\prsc{\nu}{\phi} = \int_{Z_\text{reg}} \phi(x) \dx \Vol_Z(x),
\end{equation*}
where we think of everything on the right-hand side as defined on $\Omega$. Now, the measure $\dx \Vol_Z$ is the $(n-r)$-dimensional Riemannian volume on $Z \subset \Omega \subset \R^n$ induced by $g$. In the following, we will need to understand how it compares with the Riemannian volume $\dx \Vol^0_Z$ on $Z$ induced by the Euclidean metric. This is the purpose of what comes next.

\begin{dfn}[Riemannian densities]
\label{def: Riemannian densities}
Let $x \in \Omega$ and let $G$ be a subspace of $\R^n$, we denote by $\det\parentheses*{g(x)_{\vert G}}$ the determinant of the restriction to $G$ of the inner product $g(x)$, in any basis of $G$ which is orthonormal for the Euclidean inner product of $\R^n$.

For all $r \in \ssquarebrackets{0}{n}$, we denote by $\gamma_r:\Omega \times \gr[\R^n]{r} \to (0,+\infty)$ the continuous map defined by $\gamma_r:(x,G) \mapsto\det\parentheses*{g(x)_{\vert G}}^\frac{1}{2}$. We also write $\gamma:x \mapsto \gamma_0(x,\R^n)$ for simplicity.
\end{dfn}

\begin{lem}[Comparing volumes]
\label{lem: comparing volumes}
Let $Z$ be a submanifold of codimension $r$ of $\Omega$ and let $\dx \Vol_Z$ (resp.~$\dx \Vol_Z^0$) denote the $(n-r)$-dimensional Riemannian volume on $Z$ induced by $g$ (resp.~the Euclidean metric). Then $\dx \Vol_Z$ admits the density $x \mapsto \gamma_r\parentheses*{x,T_xZ}$ with respect to $\dx \Vol_Z^0$. In particular $\dx \Vol_\Omega$ admits the density $\gamma$ with respect to the Lebesgue measure on $\Omega$.
\end{lem}

\begin{proof}
This follows directly from the definition of the Riemannian volume measures; see \cite[Chap.~3]{Lee2018} for example.
\end{proof}

We can now state and prove the Kac--Rice formula for the expectation of the linear statistics in our setting of Gaussian fields on $M$ with values in a vector bundle $E$.

\begin{dfn}[Kac--Rice density]
\label{def: rho1}
Let $\rho_1:M \to [0,+\infty)$ be defined by:
\begin{equation*}
\rho_1 : x \longmapsto \frac{\espcond{\jac{\nabla_xs}}{s(x)=0}}{\det\parentheses*{\strut 2\pi \var{s(x)}}^\frac{1}{2}},
\end{equation*}
where the numerator stands for the conditional expectation of $\jac{\nabla_xs}$ given that $s(x)=0$.
\end{dfn}

\begin{rem}
\label{rem: rho1}
Since $s$ is non-degenerate and $\cC^1$, the function $\rho_1$ is well-defined and continuous. Moreover, it does not depend on the choice of $\nabla$, nor on the choice of a metric on~$E$.
\end{rem}

\begin{prop}[Kac--Rice formula for the expectation]
\label{prop: Kac-Rice expectation}
Let $s$ be a non-degenerate $\cC^1$ centered Gaussian field. Then, for any Borel function $\phi:M \to \R$ which is non-negative or such that $\phi \rho_1 \in L^1(M)$, we have:
\begin{equation*}
\esp{\prsc{\nu}{\phi}} = \int_M \phi(x) \rho_1(x) \dx\Vol_M(x),
\end{equation*}
i.e.,~$\esp{\nu}$ is the Radon measure on $M$ with density $\rho_1$ with respect to the Riemannian volume $\dx\Vol_M$.
\end{prop}

\begin{rem}
\label{rem: Kac-Rice}
For any Borel maps $\phi_1$ and $\phi_2$, we have
\begin{equation*}
\esp{\norm*{\strut\prsc{\nu}{\phi_1}-\prsc{\nu}{\phi_2}}} \leq \esp{\prsc*{\nu}{\norm*{\phi_1-\phi_2}}} = \int_M \norm{\phi_1(x)-\phi_2(x)} \rho_1(x) \dx\Vol_M(x).
\end{equation*}
Then, if $\phi_1=\phi_2$ almost everywhere on $M$ we have $\prsc{\nu}{\phi_1}=\prsc{\nu}{\phi_2}$ almost surely. Thus, $\prsc{\nu}{\phi}$ makes sense as a random variable even if $\phi$ is only defined up to modification on a negligible set.
\end{rem}

\begin{proof}[Proof of Proposition~\ref{prop: Kac-Rice expectation}]
By a partition of unity argument, it enough to prove the result if~$\phi$ is compactly supported in an open domain $U$ satisfying the same properties as in Lemma~\ref{lem: local picture}. In this case, in local coordinates, the field $s$ corresponds to a non-degenerate $\cC^1$ centered Gaussian field $f:\Omega \to \R^r$ with $\Omega \subset \R^n$ open. Thanks to Remark~\ref{rem: rho1}, we can assume that $\nabla$ corresponds in this trivialization to the standard derivation for maps from $\Omega$ to $\R^r$, and that the metric on~$E$ corresponds to the canonical inner product on $\R^r$. Identifying $Z_\text{reg}$, the metric $g$, the measure $\dx\Vol_Z$ and the test-function $\phi$ with their images in the trivialization, we have reduced our problem to proving the result for the vanishing locus of $f$ with the volume measures induced by $g$.

By Lemma~\ref{lem: comparing volumes}, we have:
\begin{equation*}
\esp{\prsc{\nu}{\phi}} = \esp{\int_{Z_\text{reg}} \phi(x) \dx \Vol_Z(x)} = \esp{\int_{Z_\text{reg}} \phi(x)\gamma_r(x,\ker D_xf) \dx \Vol_Z^0(x)}.
\end{equation*}
For all $x \in \Omega$ and $\lambda \in \cC^0\parentheses*{\Omega,\mathcal{L}(\R^n,\R^r)}$ we define $\Psi(x,\lambda) = \phi(x)\gamma_r(x,\ker \lambda(x))\one_O\parentheses*{\lambda(x)}$, where $O = \brackets*{L \in \mathcal{L}(\R^n,\R^r) \mvert \jac{L}>0}$. Since $O$ is open and the maps $\ker:O\to \gr[\R^n]{r}$ and $\gamma_r$ are continuous, the map $\Psi$ is lower semi-continuous with respect to each variable, where $\cC^0\parentheses*{\Omega,\mathcal{L}(\R^n,\R^r)}$ is equipped with the weak topology. Thus, we can apply the Euclidean Kac--Rice formula with weight from~\cite[Thm.~7.1 and Rem.~8]{AAL2023} to $\Psi(x,Df)$. This yields
\begin{equation*}
\esp{\prsc{\nu}{\phi}} = \int_\Omega \phi(x)\frac{\espcond{\gamma_r(x,\ker D_xf)\Jac^0\parentheses*{D_xf}}{f(x)=0}}{\det\parentheses*{2\pi \var{f(x)}}^\frac{1}{2}}\dx x,
\end{equation*}
where $\Jac^0$ means that we computed the Jacobian with respect to the Euclidean metric on $\R^n$.

To conclude, we need to compare $\Jac^0$ with the Jacobian $\Jac$ with respect to $g$. This is the content of Lemma~\ref{lem: comparing Jacobians} below, which yields that $\gamma_r(x,\ker D_xf)\Jac^0\parentheses*{D_xf} = \gamma(x)\jac{D_xf}$. Since $\gamma(x)$ is deterministic, by Lemma~\ref{lem: comparing volumes} we have:
\begin{equation*}
\esp{\prsc{\nu}{\phi}} = \int_\Omega \phi(x)\frac{\espcond{\Jac\parentheses*{D_xf}}{f(x)=0}}{\det\parentheses*{2\pi \var{f(x)}}^\frac{1}{2}}\gamma(x)\dx x = \int_\Omega \phi(x)\rho_1(x) \dx \Vol_\Omega(x).
\end{equation*}
This proves that the result holds locally, that is, for a field $f:\Omega \to \R^r$, with the volume measures induced by any Riemannian metric on $\Omega$, which concludes the proof.
\end{proof}

\begin{lem}[Comparing Jacobians]
\label{lem: comparing Jacobians}
Let $x \in \Omega$ and let $L:\R^n \to \R^r$ be a surjective linear map. With the same notation as above, we have $\gamma_r(x,\ker L) \Jac^0(L) = \gamma(x) \jac{L}$.
\end{lem}

\begin{proof}
We denote by $L^*_g$ (resp.~$L^*_0$) the adjoint of $L$ with respect to the inner product $g(x)$ (resp.~the Euclidean inner product). In a Euclidean orthonormal basis adapted to $\ker(L)^\perp \oplus \ker(L)$, the matrix of $g(x)$ is symmetric of the form $\parentheses*{\begin{smallmatrix}A & \trans{B} \\ B  & C \end{smallmatrix}}$ with $A$ and $C$ positive-definite, the matrix of $L$ is $\parentheses*{\begin{smallmatrix}F & 0 \end{smallmatrix}}$, that of $L_0^*$ is $\parentheses*{\begin{smallmatrix}\trans{F} \\ 0 \end{smallmatrix}}$ and that of $L^*_g$ is $\parentheses*{\begin{smallmatrix}X \\ Y \end{smallmatrix}}$. We have $\parentheses*{\begin{smallmatrix}\trans{F} \\ 0 \end{smallmatrix}} = \parentheses*{\begin{smallmatrix}A & \trans{B} \\ B & C \end{smallmatrix}}\parentheses*{\begin{smallmatrix}X \\ Y \end{smallmatrix}}$, which leads to $\trans{F} = \parentheses*{A-\trans{B}C^{-1}B}X$. Hence,
\begin{equation*}
\det(LL_0^*)=\det(F \trans{F}) = \det\parentheses*{A-\trans{B}C^{-1}B} \det(FX) = \det\parentheses*{A-\trans{B}C^{-1}B}\det(LL_g^*),
\end{equation*}
and $\gamma_r(x,G)\Jac^0(L) = \det(C)^\frac{1}{2}\det\parentheses*{A-\trans{B}C^{-1}B}^\frac{1}{2}\jac{L}$. Since $A-\trans{B}C^{-1}B$ is the Schur complement of $C$ in the matrix of $g(x)$, we have $\det(C)^\frac{1}{2}\det\parentheses*{A-\trans{B}C^{-1}B}^\frac{1}{2}=\gamma(x)$.
\end{proof}


\subsection{Factorial moment measures and Kac--Rice densities}
\label{subsec: Factorial moment measures and Kac--Rice densities}

As in the previous section, we consider a non-degenerate $\cC^1$ centered Gaussian field $s:M \to E$ which is a random section of some vector bundle $E \to M$. Recall that $n= \dim(M)$ and $r \in \ssquarebrackets{1}{n}$ is the rank of $E$. We are interested in the finiteness of the moments of the linear statistics $\prsc{\nu}{\phi}$ with $\phi \in L^\infty_c(M)$; see Definition~\ref{def: nu}. In this section, we introduce the factorial moment measures and Kac--Rice densities of the fields $s$, and we relate them to higher moments of the linear statistics.

In the following, $p$ will always denote the order of the moment we are considering. Recall that, under our hypothesis on $s$, the random measure $\nu$ is almost surely a non-negative Radon measure on $M$; see Remark~\ref{rem: Kac-Rice}.

\begin{dfn}[Product measures]
\label{def: product measures}
Let $p\geq 1$, we denote by $\nu^{\otimes p}$ the product measure of $\nu$ with itself $p$-times. We also denote by $\nu^{[p]}$ the restriction of $\nu^{\otimes p}$ to $M^p \setminus \Delta_p$. That is, for any test-function $\Phi$:
\begin{align*}
\prsc*{\nu^{\otimes p}}{\Phi} &= \int_{Z_\text{reg}^p} \Phi(\ux) \dx\Vol_Z^{\otimes p}(\ux) & &\text{and} & \prsc*{\nu^{[p]}}{\Phi} &= \int_{Z_\text{reg}^p \setminus \Delta_p} \Phi(\ux) \dx\Vol_Z^{\otimes p}(\ux).\end{align*}
Almost surely, these measures are Radon measures on $M^p$. More generally, if we want to consider product spaces indexed by a non-empty finite set $A$ instead of $\ssquarebrackets{1}{p}$, we denote by $\nu^{\otimes A}$ the product measure of $\nu$ with itself $\norm{A}$ times on $M^A$ and by $\nu^{[A]}$ its restriction to $M^A\setminus \Delta_A$.
\end{dfn}

The following lemma describes the relation between $\nu^{\otimes p}$ and $\nu^{[p]}$, using the notation introduced in Section~\ref{subsec: Sets, partitions and diagonals}.

\begin{lem}[Relation between \texorpdfstring{$\nu^{\otimes p}$ and $\nu^{[p]}$}{}]
\label{lem: relation between nu p}
Let $p \geq 1$. If $r<n$ then $\nu^{\otimes p} = \nu^{[p]}$. If $r=n$ then $\nu^{\otimes p} = \sum_{\cI \in \cP_p} \parentheses*{\iota_\cI}_*\parentheses*{\nu^{[\cI]}}$.
\end{lem}

\begin{proof}
If $r<n$ then $Z_\text{reg}$ is a $\cC^1$ submanifold of positive dimension in $M$. In particular, the large diagonal in $\parentheses*{Z_\text{reg}}^p$ has positive codimension, and hence is negligible for $\dx \Vol_Z^{\otimes p}$. Thus $\nu^{\otimes p} = \nu^{[p]}$ in this case.

If $r=n$ then $Z_{reg}$ is a locally finite set and $\nu$ is its counting measure. Similarly, $\nu^{\otimes p}$ is the counting measure of the locally finite $\parentheses*{Z_\text{reg}}^p$, and $\Delta_p$ in no longer negligible for this measure. If $r=n=1$, we proved in~\cite[Lem.~2.7]{AL2021} that $\nu^{\otimes p} = \sum_{\cI \in \cP_p} \parentheses*{\iota_\cI}_*\parentheses*{\nu^{[\cI]}}$. The proof is purely combinatorics and it extends immediately to the case $r=n \geq 1$.
\end{proof}

Our interest in these measures is that, by the Fubini theorem, for all $\phi \in L^\infty_c(M)$ we have $\esp{\prsc{\nu}{\phi}^p} = \esp{\prsc*{\nu^{\otimes p}}{\phi^{\otimes p}}} = \prsc*{\esp{\nu^{\otimes p}}}{\phi^{\otimes p}}$, where $\phi^{\otimes p}:(x_1,\dots,x_p) \mapsto \phi(x_1)\dots\phi(x_p)$. Thus, the measure $\esp{\nu^{\otimes p}}$ is closely related with the computation of moments of linear statistics. For technical reasons, it is more convenient to consider $\esp{\nu^{[p]}}$ instead.

\begin{dfn}[Moment measures]
\label{def: moment measures}
Let $p \geq 1$, the measure $\esp{\nu^{\otimes p}}$ is called the \emph{$p$-th moment measure} of the field $s$ and $\esp{\nu^{[p]}}$ is called its \emph{$p$-th factorial moment measure}.
\end{dfn}

\begin{dfn}[Kac--Rice density of order $p$]
\label{def: rhop}
Let $p \geq 1$ and let us assume that the random vector $\parentheses*{s(x_1),\dots,s(x_p)}$ is non-degenerate for all $(x_1,\dots,x_p) \in M^p \setminus \Delta_p$. Then we define
\begin{equation*}
\rho_p:\parentheses*{x_1,\dots,x_p} \longmapsto \frac{\espcond{\strut \prod_{i=1}^p \jac{\nabla_{x_i}s}}{\forall i \in \ssquarebrackets{1}{p}, s(x_i)=0}}{\det\parentheses*{2\pi \var{s(x_1),\dots,s(x_p)}\strut}^\frac{1}{2}}
\end{equation*}
from $M^p\setminus\Delta_p$ to $[0,+\infty)$, where the numerator is the conditional expectation of $\prod_{i=1}^p \jac{\nabla_{x_i}s}$ given that $s(x_i)=0$ for all $i \in \ssquarebrackets{1}{p}$.
\end{dfn}

Once again, $\rho_p$ is well-defined and continuous on $M^p \setminus \Delta_p$ thanks to our non-degeneracy hypothesis. However, its expression is singular along $\Delta_p$.  In particular, $\rho_p$ is in general not bounded, which raises the question of its local integrability near $\Delta_p$. For example, if $f:\R^n \to \R$ is a non-degenerate enough stationary Gaussian field and $p=2$, one can check that as $y \to x$, the corresponding Kac--Rice density $\rho_2(x,y)$ behaves like $\Norm{y-x}$ if $n=1$ and like $\frac{1}{\Norm{y-x}}$ if $n>1$.
 
\begin{prop}[Kac--Rice formula for the $p$-th factorial moment]
\label{prop: Kac--Rice moments}
Let $s$ be a $\cC^1$ centered Gaussian field such that $\parentheses*{s(x_1),\dots,s(x_p)}$ is non-degenerate for all $(x_1,\dots,x_p) \in M^p \setminus \Delta_p$. Then, for any Borel function $\Phi:M^p \to \R$ which is non-negative or such that $\Phi \rho_p \in L^1(M^p)$, we have:
\begin{equation*}
\esp{\prsc*{\nu^{[p]}}{\Phi}} = \int_{M^p} \Phi(\ux) \rho_p(\ux) \dx\Vol_M^{\otimes p}(\ux),
\end{equation*}
i.e.,~$\esp{\nu^{[p]}}$ is the measure on $M^p$ with density $\rho_p$ with respect to the Riemannian volume $\dx\Vol_M^{\otimes p}$.
\end{prop}

\begin{proof}
Let us consider $S:(x_1,\dots,x_p) \mapsto \parentheses{s(x_1),\dots,s(x_p)}$ on $M^p \setminus \Delta_p$, which is a random section of the restriction over $M^p \setminus \Delta_p$ of the vector bundle $E^p \to M^p$. This is a non-degenerate $\cC^1$ centered Gaussian field on $M^p \setminus \Delta_p$, and $\nu^{[p]}$ is the measure of integration over its zero set. Bearing in mind that $\Delta_p$ is negligible in $M^p$ for $\dx\Vol_M^{\otimes p}$, the result follows from Proposition~\ref{prop: Kac-Rice expectation} applied to $S$.
\end{proof}

The following proposition relates the properties of the Kac--Rice densities, the moment measures and the moments of linear statistics.

\begin{prop}[Relation between moments, measures and densities]
\label{prop: relation between moments measures and densities}
Let $p\geq 1$ and let $s$ be a $\cC^1$ centered Gaussian field such that $\parentheses*{s(x_1),\dots,s(x_p)}$ is non-degenerate for all $(x_1,\dots,x_p) \notin \Delta_p$. Then the following four properties are equivalent.
\begin{enumerate}
\item \label{item: relations linear stats} For all $\phi \in L^\infty_c(M)$, we have $\esp{\strut\norm*{\prsc{\nu}{\phi}}^p}<+\infty$.
\item \label{item: relations moment measures} For all $k \in \ssquarebrackets{1}{p}$, the moment measure $\esp{\nu^{\otimes k}}$ is Radon on $M^k$, i.e.,~finite on compact sets.
\item \label{item: relations factorial moment measures} For all $k \in \ssquarebrackets{1}{p}$, the factorial moment measure $\esp{\nu^{[k]}}$ is Radon on $M^k$.
\item \label{item: relations densities} For all $k \in \ssquarebrackets{1}{p}$, the Kac--Rice density satisfies $\rho_k \in L^1_\text{loc}(M^k)$.
\end{enumerate}
\end{prop}

\begin{proof}
Let us assume~\eqref{item: relations linear stats}. Let $k \in \ssquarebrackets{1}{p}$ and let $K_0 \subset M^k$ be compact. There exists a compact set $K \subset M$ such that $K_0 \subset K^k$. Then,
\begin{equation*}
\prsc*{\esp{\nu^{\otimes k}}}{\one_{K_0}} \leq \prsc*{\esp{\nu^{\otimes k}}}{\one_{K}^{\otimes k}} = \esp{\prsc*{\nu}{\one_K}^k} = \esp{\norm*{\prsc*{\nu}{\one_K}}^k}.
\end{equation*}
Since $\one_K \in L^\infty_c(M)$, the $p$-th absolute moment of $\prsc{\nu}{\one_K}$ is finite; hence so is its $k$-th absolute moment. Thus $\prsc*{\esp{\nu^{\otimes k}}}{\one_{K_0}}<+\infty$ for all compact $K_0$ and~\eqref{item: relations moment measures} is satisfied.

If~\eqref{item: relations moment measures} is satisfied then so is~\eqref{item: relations factorial moment measures}. Indeed, for any $k \in \ssquarebrackets{1}{p}$, the measure $\nu^{[k]}$ is the restriction of $\nu^{\otimes k}$ to $M^k \setminus \Delta_k$. Thus, for any compact $K \subset M^k$ we have:
\begin{equation*}
\prsc*{\esp{\nu^{[k]}}}{\one_K} = \esp{\prsc*{\nu^{[k]}}{\one_K}} \leq \esp{\prsc*{\nu^{\otimes k}}{\one_K}} = \prsc*{\esp{\nu^{\otimes k}}}{\one_K} <+\infty.
\end{equation*}

If~\eqref{item: relations factorial moment measures} is satisfied, let $k \in \ssquarebrackets{1}{p}$ and let $K \subset M^k$ be a compact. By Proposition~\ref{prop: Kac--Rice moments} we have:
\begin{equation*}
\int_K \rho_k(\ux) \dx\Vol_M^{\otimes p}(\ux) = \int_M \one_K(\ux)\rho_k(\ux) \dx\Vol_M^{\otimes p}(\ux) = \esp{\prsc*{\nu^{[k]}}{\one_K}} = \prsc*{\esp{\nu^{[k]}}}{\one_K} <+\infty.
\end{equation*}
Thus $\rho_k$ is integrable on any compact set, that is, $\rho_k \in L^1_\text{loc}(M^k)$. This proves~\eqref{item: relations densities} in this case.

Finally, let us assume that~\eqref{item: relations densities} holds. Let $\phi \in L^\infty_c(M)$, and let us denote by $K \subset M$ its compact support. We have $\esp{\norm*{\prsc{\nu}{\phi}}^p}\leq \esp{\prsc*{\nu}{\norm{\phi}}^p} \leq \Norm{\phi}_\infty^p \esp{\prsc{\nu}{\one_K}^p}$, so it is enough to prove that $\esp{\prsc{\nu}{\one_K}^p} = \esp{\prsc*{\nu^{\otimes p}}{\one_{K^p}}}$ is finite. By Lemma~\ref{lem: relation between nu p}, whether $r=n$ or not, we have:
\begin{equation*}
\esp{\prsc*{\nu^{\otimes p}}{\one_{K^p}}} \leq \sum_{\cI \in \cP_p} \esp{\prsc*{\nu^{[\cI]}}{\one_{K^p} \circ \iota_\cI}} = \sum_{\cI \in \cP_p} \prsc*{\esp{\nu^{[\cI]}}}{\one_{K^\cI}}.
\end{equation*}
Then, the Kac--Rice formula for moments and the local integrability of the $(\rho_k)_{1 \leq k \leq p}$ yields
\begin{equation*}
\esp{\prsc{\nu}{\one_K}^p} \leq \sum_{\cI \in \cP_p}\prsc*{\esp{\nu^{[\cI]}}}{\one_{K^\cI}} = \sum_{\cI \in \cP_p} \int_{K^{\norm{\cI}}} \rho_{\norm{\cI}}(\ux)\dx\Vol_M^{\otimes \norm{\cI}}(\ux) <+\infty,
\end{equation*}
which proves~\eqref{item: relations linear stats} and concludes the proof.
\end{proof}


\subsection{Proofs of Theorems~\ref{thm: moments Rn} and~\ref{thm: moments M}: finiteness of moments}
\label{subsec: Proof of finiteness of moments}

The goal of this section is to prove~Theorems~\ref{thm: moments Rn} and~\ref{thm: moments M}, which give simple conditions for the finiteness of the moments of the linear statistics of a Gaussian field. We begin by proving a local version of Theorem~\ref{thm: moments Rn}, under a non-degeneracy hypothesis for the multijets of the field. This is Theorem~\ref{thm: moments Rn local} below. Then we deduce Theorem~\ref{thm: moments M} from Theorem~\ref{thm: moments Rn local}, in the case of Gaussian fields with value in a vector bundle. Finally, Theorem~\ref{thm: moments Rn} is obtained as a special case of Theorem~\ref{thm: moments M}.

Let $\Omega \subset \R^n$ be open. Recall that $\cMJ_p(\Omega,\R^r) \to C_p[\Omega]$ is defined in Definition~\ref{def: multijets} as the restriction over $C_p[\Omega] \subset C_p[\R^n]$ of the vector bundle $\cMJ_p(\R^n,\R^r) \to C_p[\R^n]$ from Theorem~\ref{thm: multijet bundle}.

\begin{thm}[Finiteness of moments, local version]
\label{thm: moments Rn local}
Let $f:\Omega \to \R^r$ be a centered Gaussian field and $\nu$ be as in Definition~\ref{def: nu}. Let $p \geq 1$, if $f$ is $\cC^p$ and for all $k \in \ssquarebrackets{1}{p}$ the Gaussian field $\mj_k(f,\cdot):C_k[\Omega] \to \cMJ_k(\Omega,\R^r)$ is non-degenerate, then the four equivalent statements in Proposition~\ref{prop: relation between moments measures and densities} hold.
\end{thm}

\begin{proof}
Let $f:\Omega \to \R^r$ be a $\cC^p$ centered Gaussian field such that $\mj_k(f,\cdot):C_k[\Omega] \to \cMJ_k(\Omega,\R^r)$ is non-degenerate for all $k \in \ssquarebrackets{1}{p}$.

\paragraph{Step 1: Gaussianity and non-degeneracy of the multijets.}
Since $f$ is $\cC^p$, for all $k \in \ssquarebrackets{1}{p}$ we have $\mj_k(f,\cdot) \in \Gamma^1\parentheses*{\strut C_k[\Omega],\cMJ_k(\Omega,\R^r)}$ because of Condition~\ref{item: thm multijet mj} in Theorem~\ref{thm: multijet bundle}. Since $f$ is centered and Gaussian, so is any finite collection of jets of $f$. Then, for all $m \geq 1$ and all $z_1,\dots,z_m \in C_k[\Omega]$ we have that $\parentheses*{\mj_k(f,z_1),\dots,\mj_k(f,z_1)}$ is a centered Gaussian. Indeed, by Condition~\ref{item: thm multijet localness} in Theorem~\ref{thm: multijet bundle}, this is the image of a centered Gaussian by a linear map. Thus, $\mj_k(f,\cdot)$ is a non-degenerate $\cC^1$ centered Gaussian field on $C_k[\Omega]$ with values in $\cMJ_k(\Omega,\R^r)$.

Let $z \notin \pi^{-1}(\Delta_p)$ and let $\ux=(x_1,\dots,x_p) = \pi(z)$. By Condition~\ref{item: thm multijet localness} in Theorem~\ref{thm: multijet bundle}, the map $\Theta_z$ is a linear surjection. A dimension argument, shows that it is actually a bijection. Thus
\begin{equation*}
\parentheses*{\strut f(x_1),\dots,f(x_p)} = \parentheses*{\strut \jet_0(f,x_1),\dots,\jet_0(f,x_p)} = \Theta_z^{-1}\parentheses*{\mj_p(f,z)},
\end{equation*}
which proves that $\parentheses*{\strut f(x_1),\dots,f(x_p)}$ is non-degenerate. Thus, the hypotheses of Proposition~\ref{prop: relation between moments measures and densities} are satisfied, and the four statements appearing in this proposition are indeed equivalent.

\paragraph{Step 2: Comparing zeros of \texorpdfstring{$f$ and $\mj_k(f,\cdot)$}{}.}
Let $k \in \ssquarebrackets{1}{p}$, in the following we are going to prove that $\esp{\nu^{[k]}}$ is a Radon measure on~$\Omega^k$, which is enough to conclude the proof. In the following, we say that a subset of $C_k[\Omega]$ (resp.~$\Omega^k$) is negligible if its $k(n-r)$-dimensional Hausdorff measure is $0$.

Let us consider the Gaussian field $\mj_k(f,\cdot):C_k[\Omega] \to \cMJ_k(\Omega,\R^r)$. We have checked above that it satisfies the hypotheses of Proposition~\ref{prop: weak Bulinskaya}. Let $X \subset C_k[\Omega]$ denote the zero set of $\mj_k(f,\cdot)$. As in Section~\ref{subsec: Bulinskaya lemma and Kac--Rice formula for the expectation}, we define $X_\text{sing} = \brackets*{\strut z \in C_k[\Omega] \mvert \mj_k(f,x)=0 \ \text{and} \ \jac{\nabla_z \mj_k(f,\cdot)}=0}$ and $X_\text{reg} = X \setminus X_\text{sing}$. Recall that $X_\text{reg}$ is a $\cC^1$ submanifold of codimension $kr$ and that $X_\text{sing}$ is almost surely negligible by Proposition~\ref{prop: weak Bulinskaya}. Let $Y = X \cap \pi^{-1}\parentheses*{\Omega^k \setminus \Delta_k}$, we also let $Y_\text{sing}=Y \cap X_\text{sing}$ and $Y_\text{reg} = Y \cap X_\text{reg}$.

Recalling that $Z = f^{-1}(0) \subset \Omega$, for all $z \in C_k[\Omega]$ we have $z \in Y$ if and only if $\pi(z) \in Z^k \setminus \Delta_k$; see Condition~\ref{item: thm multijet ev} in Theorem~\ref{thm: multijet bundle}. By Condition~\ref{item: thm multijet pi} in the same theorem, the restriction of $\pi$ to $\pi^{-1}\parentheses*{\Omega^k \setminus \Delta_k}$ is a diffeomorphism. Hence $\pi(Y) = Z^k \setminus \Delta_k$, the set $\pi(Y_\text{reg})$ is a $\cC^1$ submanifold of $\Omega^k \setminus \Delta_k$, and $\pi(Y_\text{sing})$ is almost surely negligible. Since $Z^k \setminus \parentheses*{Z_\text{reg}}^k$ is also almost surely negligible, the submanifolds $Z_\text{reg}^k \setminus \Delta_k$ and $\pi(Y_\text{reg})$ are almost surely the same, up to a negligible set (actually the reader can check that $\pi(Y_\text{reg}) = Z_\text{reg}^k \setminus \Delta_k$ using the trivialization $\tau$ introduced at the end of Section~\ref{subsec: Definition of the bundle MPjRn}). Recalling Definition~\ref{def: product measures}, this shows that $\nu^{[k]}$ is the same as the integral over $\pi(Y_\text{reg})$ with respect to the Riemannian volume $\dx \Vol_{\pi(Y)}$ induced by the Euclidean metric on~$(\R^n)^k$. At this stage we know that, almost surely,
\begin{equation}
\label{eq: nu k is pi(Y)}
\forall \Phi \in L^\infty_c\parentheses*{\Omega^k}, \qquad \prsc*{\nu^{[k]}}{\Phi} = \int_{\pi(Y_\text{reg})} \Phi(\ux) \dx \Vol_{\pi(Y)}(\ux).
\end{equation}

\paragraph{Step 3: Comparing volumes.}
Let us introduce a Riemannian metric $g$ on $C_k[\Omega]$. It induces a volume measure $\dx \Vol_X$ on $X_\text{reg}$, and hence on $Y_\text{reg}$. Besides, let $z \in C_k[\Omega]$ and let $G \subset T_z\parentheses*{C_k[\Omega]}$ be a vector subspace of codimension $kr$. We define $J(G) = \det\parentheses*{\parentheses*{D_z\pi_{\vert G}}^* D_z\pi_{\vert G}}^\frac{1}{2}$, where the adjoint of $D_z\pi_{\vert G}$ is computed with respect to $g$ on $G$ and to the Euclidean metric on $(\R^n)^k$. Since $\pi$ is smooth, this defines a smooth non-negative function $J$ on the total space of the Grassmannian bundle $\gr[\strut T \parentheses*{C_k[\Omega]}]{kr} \to C_k[\Omega]$ of subspaces of codimension $kr$ in the tangent of $C_k[\Omega]$. Our interest in this map is that if $z \in Y_\text{reg}$ and $G = T_zY_\text{reg}$ then $J(G)$ is the Jacobian determinant of $D_z(\pi_Y)$, where the $\cC^\infty$-diffeomorphism $\pi_Y:Y_\text{reg} \to \pi(Y_\text{reg})$ is the restriction of $\pi$ on both sides.

Let $K \subset \Omega^k$ be compact and let us apply Equation~\eqref{eq: nu k is pi(Y)} to $\one_K$. Using the previous notation, the change of variables $\pi_Y$ yields:
\begin{equation*}
\prsc*{\nu^{[k]}}{\one_K} = \int_{Y_\text{reg}} \one_K(\pi(z)) J(T_zY_\text{reg}) \dx \Vol_X(z) \leq \int_{X_\text{reg}} \one_{\pi^{-1}(K)}(z) J(T_zX_\text{reg}) \dx \Vol_X(z).
\end{equation*}
By Condition~\ref{item: thm multijet pi} in Theorem~\ref{thm: multijet bundle}, the projection $\pi$ is proper; hence $\tilde{K}=\pi^{-1}(K)$ is compact. Since the bundle $\gr[\strut T \parentheses*{C_k[\Omega]}]{kr} \to C_k[\Omega]$ has compact fiber, its restriction over $\tilde{K} \subset C_k[\Omega]$ is compact. By continuity, the function $J$ is bounded on this compact set by some constant $C_K$. Finally, we have proved that, almost surely,
\begin{equation*}
\prsc*{\nu^{[k]}}{\one_K} \leq C_K \prsc*{\tilde{\nu}}{\one_{\tilde{K}}},
\end{equation*}
where $\tilde{\nu}$ is defined by integrating over $X_\text{reg}$ with respect to $\dx \Vol_X$. Taking expectation on both sides we get $\prsc*{\esp{\nu^{[k]}}}{\one_K} \leq C_K \prsc*{\esp{\tilde{\nu}}}{\one_{\tilde{K}}}$.

\paragraph{Step 4: Applying the Kac--Rice formula to multijets.}
Now, $X_\text{reg}$ is the regular part of the zero set $X$ of the Gaussian field $\mj_k(f,\cdot)$. We have checked at the beginning of the proof that $\mj_k(f,\cdot)$ satisfies the hypotheses of Proposition~\ref{prop: Kac-Rice expectation}. This Proposition yields that $\esp{\tilde{\nu}}$ is a Radon measure on $C_k[\Omega]$. Hence $\prsc*{\esp{\tilde{\nu}}}{\one_{\tilde{K}}}$ is finite, and so is $\prsc*{\esp{\nu^{[k]}}}{\one_K}$. Thus $\esp{\nu^{[k]}}$ is Radon on $\Omega^k$, which concludes the proof.
\end{proof}

We can now prove Theorem~\ref{thm: moments M}, which gives a criterion for the finiteness of the $p$-th moments of the linear statistics associated with a centered Gaussian field $s:M \to E$, where $E\to M$ is some vector bundle of rank $r$ over a Riemannian manifold $(M,g)$ without boundary of dimension $n \geq r$. The idea of the proof is to patch together the local results obtained by applying Theorem~\ref{thm: moments Rn local} in nice local trivializations.

\begin{proof}[Proof of Theorem~\ref{thm: moments M}]
Let $p \geq 1$ and let $s \in \Gamma^p(M,E)$ be a centered Gaussian field which is $\cC^p$ and $(p-1)$-non-degenerate.

\paragraph{Step 1: Existence of nice local trivializations.}
Let $x_0 \in M$, there exists an open neighborhood~$U$ of~$x_0$ and a local trivialization of $E$ over $U$ given by Lemma~\ref{lem: local picture}. In this trivialization, the Gaussian section $s$ corresponds to a centered Gaussian field $f:\Omega \to \R^r$ which is $\cC^p$ and $(p-1)$-non-degenerate. We denote by $x \in \Omega$ the image of $x_0$ in local coordinates.

Let $k \in \ssquarebrackets{1}{p}$ and let $\ux=(x,\dots,x) \in \Omega^k$. Since $\jet_{p-1}(f,x)$ is non-degenerate so is $\jet_{k-1}(f,x)$; see Definition~\ref{def: p non degeneracy}. Then, by Condition~\ref{item: thm multijet localness} in Theorem~\ref{thm: multijet bundle}, for all $z \in \pi^{-1}(\brackets{\ux}) \subset C_k[\Omega]$ the Gaussian vector $\mj_k(f,z) = \Theta_z(\jet_{k-1}(f,x)) \in \cMJ_k(\Omega,\R^r)_z$ is non-degenerate. By Condition~\ref{item: thm multijet pi} in Theorem~\ref{thm: multijet bundle}, the map $\pi$ is proper; hence $\pi^{-1}(\brackets{\ux})$ is compact. One the other hand, since $f$ is $\cC^p$, we know that $\mj_k(f,\cdot)$ is at least $\cC^1$. Thus $z \mapsto \det \var{\mj_k(f,z)}$ is continuous on $C_k[\Omega]$ and positive on the compact $\pi^{-1}(\brackets{\ux})$, and hence on some neighborhood $V_k$ of $\pi^{-1}(\brackets{\ux})$ in $C_k[\Omega]$.

Up to reducing $V_k$ we can assume that $V_k = \pi^{-1}(W_k)$, where $W_k$ is an open neighborhood of $\ux$ in $\Omega^k$. Otherwise, there would exist a sequence $\parentheses*{z_n}_{n \in \N} \in C_k[\Omega] \setminus V_k$ such that $\pi(z_n) \xrightarrow[n \to +\infty]{}\ux$. By properness of $\pi$, up to extracting a subsequence, we could assume that $z_n \xrightarrow[n \to +\infty]{}z$. By continuity $z \in \pi^{-1}(\brackets{\ux})$, which would be absurd. Since $W_k$ is a neighborhood of $\ux$ in $\Omega^k$, there exists an open neighborhood $\Upsilon_k$ of $x$ in $\Omega$ such that $(\Upsilon_k)^k \subset W_k$.

Let us define $\Upsilon = \cap_{k=1}^p \Upsilon_k$, which is an open neighborhood of $x$. For all $k \in \ssquarebrackets{1}{p}$, we have $C_k[\Upsilon] =\pi^{-1}(\Upsilon^k) \subset \pi^{-1}\parentheses*{(\Upsilon_k)^k} \subset \pi^{-1}(W_k)=V_k$ and $\mj_k(f,\cdot)$ is non-degenerate on $C_k[\Upsilon]$. Thus, up to replacing $\Omega$ by the smaller neighborhood $\Upsilon$ of $x$ in $\Omega$ and replacing $U$ by the corresponding neighborhood of $x_0$ on $M$, we can assume that the local trivialization given by Lemma~\ref{lem: local picture} is such that, for all $k \in \ssquarebrackets{1}{p}$, the Gaussian field $\mj_k(f,\cdot):C_k[\Omega] \to \cMJ_k(\Omega,\R^r)$ is non-degenerate.

\paragraph{Step 2: Reduction to the local case.}
Let $\phi \in L^\infty_c(M)$ and let $K$ denote its support. By compactness, there exists a finite family $\parentheses*{U_i}_{1=1}^m$ of open subsets such that $K \subset \bigcup_{i=1}^m U_i$ and each $U_i$ is the domain of nice trivialization of the type we built in the previous paragraph. Letting $U_0=M \setminus K$, there exists a smooth partition of unity $\parentheses*{\chi_i}_{i=0}^m$ subordinated to the open covering $\parentheses*{U_i}_{1=0}^m$ of $M$. Then $\phi = \sum_{i=1}^m \chi_i \phi$ by the definition of $K$.

Recall that $\nu$ is the measure from Definition~\ref{def: nu}. We have $\norm*{\prsc*{\nu}{\phi}} \leq \prsc*{\nu}{\norm{\phi}} = \sum_{i=1}^m \prsc{\nu}{\phi_i}$, where $\phi_i = \chi_i\norm{\phi}$ for all $i \in \ssquarebrackets{1}{m}$. Let $p \geq 1$, by Hölder's inequality we get:
\begin{align*}
\esp{\strut\norm*{\prsc*{\nu}{\phi}}^p} &\leq \esp{\parentheses*{\sum_{i=1}^m \prsc{\nu}{\phi_i}}^p} = \sum_{1 \leq i_1,\dots,i_p \leq m} \esp{\prod_{j=1}^p \prsc*{\nu}{\phi_{i_j}}} \leq \sum_{1 \leq i_1,\dots,i_p \leq m} \prod_{j=1}^p \esp{\prsc*{\nu}{\phi_{i_j}}^p}^\frac{1}{p}\\
&\leq m^p \max_{1 \leq i \leq m} \esp{\prsc*{\nu}{\phi_i}^p}.
\end{align*}
Thus, in order to prove Theorem~\ref{thm: moments M}, it is enough to prove that $\esp{\prsc*{\nu}{\phi}^p} < +\infty$ for any non-negative $\phi \in L^\infty_c(M)$ whose support is included in the domain of a nice trivialization.

\paragraph{Step 3: Local case.}
Let $U \subset M$ be an open subset over which we have a nice trivialization of $E$ and~$s$. That is, $U$ is as in Lemma~\ref{lem: local picture}, the section $s$ reads as $f:\Omega \to \R^r$ in local coordinates, and in addition we can assume that for all $k \in \ssquarebrackets{1}{p}$ the field $\mj_k(f,\cdot)$ is non-degenerate on $C_k[\Omega]$. Identifying objects on $U$ with their image in the local trivialization, we reduced our problem to proving that $\esp{\prsc{\nu}{\phi}^p}<+\infty$ for all non-negative $\phi \in L^\infty_c(\Omega)$. Note that $\nu$ is the measure of integration over $Z_\text{reg}$ with respect to the Riemannian volume measure $\dx \Vol_Z$ induced by the metric $g$. In order to apply Theorem~\ref{thm: moments Rn local}, we need to compare $\nu$ with $\tilde{\nu}$, which is the measure of integration over $Z_\text{reg}$ with respect to the Euclidean volume measure $\dx \Vol_Z^0$.

Let $\phi \in L^\infty_c(\Omega)$ be non-negative and let $K$ denote its compact support. Recalling Definition~\ref{def: Riemannian densities}, Lemma~\ref{lem: comparing volumes} shows that:
\begin{equation*}
\prsc{\nu}{\phi} = \int_{Z_\text{reg}} \phi(x) \dx\Vol_Z(x) = \int_{Z_\text{reg} \cap K} \phi(x)\gamma_r(x,\ker D_xf) \dx\Vol_Z^0(x).
\end{equation*}
Since $\gamma_r$ is continuous and $K \times \gr[\R^n]{r}$ is compact, the non-negative function $\gamma_r$ is bounded by some $C_K>0$ on this set. Thus $\prsc*{\nu}{\phi} \leq C_K \prsc*{\tilde{\nu}}{\phi}$. Since $f:\Omega \to \R^r$ satisfies the hypotheses of Theorem~\ref{thm: moments Rn local} and $\tilde{\nu}$ is the measure of integration over its zero set induced by the Euclidean metric, we have $\esp{\prsc*{\nu}{\phi}^p} \leq C_K^p \esp{\prsc{\tilde{\nu}}{\phi}^p} <+\infty$.
\end{proof}

We conclude this section with the proof of Theorem~\ref{thm: moments Rn}, which is a corollary of Theorem~\ref{thm: moments M}.

\begin{proof}[Proof of Theorem~\ref{thm: moments Rn}]
Let $f:\Omega \to \R^r$ be a centered Gaussian field which is $\cC^p$ and $(p-1)$-non-degenerate in the sense of Definition~\ref{def: p non degeneracy}. Then $s=(f,\Id)$ is a random section of the trivial bundle $\R^r \times \Omega \to \Omega$. This $s$ is also a $\cC^p$ and $(p-1)$-non-degenerate centered Gaussian field. Its vanishing locus (as a section) is the same as the vanishing locus of $f$. Hence, the result follows from applying Theorem~\ref{thm: moments M} to $s$.
\end{proof}


\section{Multijets adapted to a differential operator}
\label{sec: Multijets adapted to a differential operator}

In Theorem~\ref{thm: multijet bundle} we defined multijets such that, over the configuration space $\config \subset C_p[\R^n]$, the $p$-multijet $\mj_p(f,\ux)$ reads as $\parentheses{f(x_1),\dots,f(x_p)}$ in the natural trivialization $\tau$ (see the end of Section~\ref{subsec: Definition of the bundle MPjRn}). Thus $\mj_p(f,\ux)$ is a way to patch together the $0$-jets of $f$ at $x_i$ into a smooth object that does not degenerate along $\Delta_p$. In this section, we explain how a similar construction allows us to build a multijet that patches together the $k$-jets of $f$ at $x_i$, and more generally the values at $x_i$ of~$\cD f$, where $\cD$ is a differential operator. In Section~\ref{subsec: Differential operator} we recall the definition of a differential operator. Then we define a multijet adapted to a given differential operator in Section~\ref{subsec: Multijets adapted to D}. Finally, in Section~\ref{subsec: Finiteness of moments for critical points}, we prove Theorem~\ref{thm: moments critical points}.


\subsection{Differential operator}
\label{subsec: Differential operator}

In this section, we recall a few fact about differential operators. In the following, we use the multi-index notation introduced in Section~\ref{subsec: Spaces of functions, sections and jets}.

\begin{dfn}[Differential operator]
\label{def: differential operator}
Let $\Omega \subset \R^n$ be open, let $q,r \geq 1$ and $d \geq 0$. A \emph{differential operator of order at most $d$} is a linear map $\cD:\cC^d(\Omega,\R^q) \to \cC^0(\Omega,\R^r)$ such that there exist continuous functions $\parentheses*{a_{ij\alpha}}_{1 \leq i \leq r; 1 \leq j \leq q; \norm{\alpha}\leq d}$ on $\Omega$ such that for all $f=(f_1,\dots,f_q)\in \cC^d(\Omega,\R^q)$,
\begin{equation}
\label{eq: differential operator}
\cD(f):x \longmapsto \parentheses*{\sum_{j=1}^q \sum_{\norm{\alpha}\leq d} a_{ij\alpha}(x) \partial^\alpha f_j(x)}_{1 \leq i \leq r}.
\end{equation}
More generally, let $M$ be a manifold of dimension $n$ and let $E \to M$ and $F\to M$ be two vector bundles of ranks $q$ and $r$ respectively, we say that $\cD:\Gamma^d(M,E) \to \Gamma^0(M,F)$ is a \emph{differential operator of order at most $d$} if around any point $x \in M$ there exist a chart and local trivializations of $E$ and $F$ such that $\cD$ is of the form~\eqref{eq: differential operator} in the corresponding local coordinates. We say that $\cD$ is of \emph{order $d \in \N$} if it is of order at most $d$ and not of order at most $d-1$. If $s \in \Gamma^d(M,E)$ and $x \in M$, we write $\cD s = \cD(s)$ and $\cD_x s=\cD(s)(x)$ for simplicity.
\end{dfn}

\begin{rem}
\label{rem: differential operator}
Let us make some important comments.
\begin{itemize}
\item If $\cD:\Gamma^d(M,E) \to \Gamma^0(M,F)$ is a differential operator of order at most $d$, then it is of the form~\eqref{eq: differential operator} in any set of local coordinates on $M$, $E$ and $F$.
\item An equivalent definition of a differential operator of order at most $d$ is that it factors linearly through the bundle of $d$-jets. That is, there exists $L \in \Gamma^0\parentheses*{M,\cJ_d(M,E)^* \otimes F}$ such that $\cD_x s = L(x)\jet_d(s,x) \in F_x$ for all $s \in \Gamma^d(M,E)$ and $x \in M$.
\end{itemize}
\end{rem}

In the following we always assume that $M$, $E$, $F$ and $L$ are smooth. In particular, the functions $\parentheses*{a_{ij\alpha}}$ appearing in the local expression~\eqref{eq: differential operator} of $\cD$ are smooth. This implies that if $s \in \Gamma^{d+l}(M,E)$ then $\cD s \in \Gamma^{l}(M,F)$.

\begin{ex}
\label{ex: differential operator}
The main examples we have in mind are the following.
\begin{itemize}
\item The differential $D:\cC^1(M) \to \Gamma^0(M,T^*M)$ is a differential operator of order $1$.
\item For all $k \in \N$, the jet map $\jet_k:\Gamma^k(M,E) \to \Gamma^0\parentheses*{M,\cJ_k(M,E)}$ is a differential operator of order $k$ corresponding to $L(x)$ being the identity of $\cJ_k(M,E)_x$ for all $x \in M$.
\item If $M$ is equipped with a Riemannian metric, the Laplace--Beltrami operator $\Delta$ acting on $\cC^2(M)$ is a differential operator of order $2$.
\item If $\nabla$ is a connection on $E \to M$ then $\nabla:\Gamma^1(M,E) \to \Gamma^0\parentheses*{M,T^*M \otimes E}$ is a differential operator of order $1$. Indeed, in a local frame $(e_1,\dots,e_q)$ of $E$ and local coordinates $(x_1,\dots,x_n)$ on $M$ the covariant derivative of $s= \sum_{j=1}^q f_j e_j \in \Gamma^1(M,E)$ at $x$ is given by:
\begin{equation*}
\nabla_x s = \sum_{i=1}^n\sum_{j=1}^q \parentheses*{\strut \partial_i f_j(x)+ \sum_{k=1}^q\mu_{ijk}(x)f_k(x)} dx_i \otimes e_j(x),
\end{equation*}
where the $\parentheses*{\mu_{ijk}}$ are defined by the relations $\nabla e_k = \sum_{i=1}^n\sum_{j=1}^q \mu_{ijk} dx_i \otimes e_j$ for all $k \in \ssquarebrackets{1}{q}$.
\end{itemize}
\end{ex}


\subsection{Multijets adapted to \texorpdfstring{$\cD$}{}}
\label{subsec: Multijets adapted to D}

The purpose of this section is to explain how to modify the construction of Section~\ref{sec: Definition of the multijet bundles} in order to define a multijet bundle adapted to a given differential operator.

Let $n,q$ and $r \geq 1$, we consider a differential operator $\cD:\cC^d(\R^n,\R^q) \to \cC^0(\R^n,\R^r)$ of order~$d$. As in Remark~\ref{rem: differential operator}, there exists a section $L$ of $\cJ_d(\R^n,\R^q)^* \otimes \R^r$ such that for any $f \in \cC^d(\R^n,\R^q)$ and $x \in \R^n$, we have $\cD_x f = L(x)\jet_d(f,x)$. We assume that for all $x \in \R^n$ the linear map $\cD_x :\cC^d(\R^n,\R^q) \to \R^r$ is surjective, which is equivalent to $L(x):\cJ_d(\R^n,\R^q)_x \to \R^r$ being surjective. Moreover, we assume that $L$ is smooth. In this context, we have the following analogue of Theorem~\ref{thm: multijet bundle}. It holds in particular if $\cD = \jet_k$ or $\cD =D$ is the differential.

\begin{thm}[Existence of multijets adapted to \texorpdfstring{$\cD$}{}]
\label{thm: multijet bundle D}
Let $\cD:\cC^d(\R^n,\R^q) \to \cC^0(\R^n,\R^r)$ be a differential operator of order $d$ as above. Let $p \geq 1$, there exist a smooth manifold $C_p^\cD[\R^n]$ of dimension $np$ without boundary and a smooth vector bundle $\cMJ_p^\cD(\R^n) \to C_p^\cD[\R^n]$ of rank $rp$ with the following properties.

\begin{enumerate}

\item \label{item: thm multijet D pi} There exists a smooth proper surjection $\pi:C_p^\cD[\R^n] \to (\R^n)^p$ such that $\pi^{-1}\parentheses*{\config}$ is a dense open subset of $C_p^\cD[\R^n]$, and $\pi$ restricted to $\pi^{-1}\parentheses*{\config}$ is a $\cC^\infty$-diffeomorphism onto $\config$.

\item \label{item: thm multijet D mj} There exists a map $\mj_p^\cD:\cC^{(d+1)p-1}(\R^n,\R^r) \times C_p^\cD[\R^n] \to \cMJ_p^\cD(\R^n)$ such that:
\begin{itemize}
\item for all $z \in C_p^\cD[\R^n]$, the map $\mj_p^\cD(\cdot,z):\cC^{(d+1)p-1}(\R^n,\R^r) \to \cMJ_p^\cD(\R^n)_z$ is surjective;
\item for all $f \in \cC^{l+(d+1)p-1}(\R^n,\R^r)$, the section $\mj_p^\cD(f,\cdot)$ of $\cMJ_p^\cD(\R^n) \to C_p^\cD[\R^n]$ is~$\cC^l$.
\end{itemize}

\item \label{item: thm multijet D ev} Let $z \in C_p^\cD[\R^n]$ be such that $\pi(z)=(x_1,\dots,x_p) \notin \Delta_p$, then for all $f \in \cC^{(d+1)p-1}(\R^n,\R^q)$:
\begin{equation*}
\mj_p^\cD(f,z) = 0 \iff \forall i \in \ssquarebrackets{1}{p},\ \cD_{x_i}f=0.
\end{equation*}

\item \label{item: thm multijet D localness} Let $z \in C_p^\cD[\R^n]$, let $\cI = \cI\parentheses{\pi(z)}$ and let $\parentheses{y_I}_{I \in \cI} = \iota_{\cI}^{-1}\parentheses{\pi(z)} \in (\R^n)^{\cI} \setminus \Delta_{\cI}$, there exists a linear surjection $\Theta_z^\cD : \prod_{I \in \cI} \cJ_{(d+1)\norm{I}-1}(\R^n,\R^q)_{y_I} \to \cMJ_p^\cD(\R^n)_z$ such that
\begin{equation*}
\forall f \in \cC^{(d+1)p-1}(\R^n,\R^r), \qquad \mj_p^\cD(f,z) = \Theta_z^\cD\parentheses*{\parentheses*{\jet_{(d+1)\norm{I}-1}(f,y_I)}_{I \in \cI}}.
\end{equation*}

\end{enumerate}

\end{thm}

\begin{proof}
The proof follows the same strategy as what we did in Sections~\ref{sec: Evaluation maps and their kernels} and~\ref{sec: Definition of the multijet bundles} in order to prove Theorem~\ref{thm: multijet bundle}. Let us sketch its main steps.

Let $\ux =(x_1,\dots,x_p) \in (\R^n)^p$ and let $\hux = \parentheses{\ux,\dots,\ux} \in (\R^n)^{(d+1)p}$. Let $f \in \cC^{(d+1)p-1}(\R^n,\R^q)$, the polynomial map $K(f,\hux) \in \R_{(d+1)p-1}[X]\otimes \R^q$ is defined as in Definition~\ref{def: Kergin polynomial}. For all $i \in \ssquarebrackets{1}{p}$, since $x_i$ appears with multiplicity $(d+1)$ in $\hux$, the map $K(f,\hux)$ has the same $d$-jet as $f$ at $x_i$. Hence $\cD_{x_i}f = L(x_i)\jet_d(f,x_i) = L(x_i)\jet_d\parentheses*{K(f,\hux),x_i}= \cD_{x_i}\parentheses*{K(f,\hux)}$.

If $\ux \notin \Delta_p$, let us define $\ev_{\ux}^\cD:P \mapsto \parentheses*{\cD_{x_i}P}_{1 \leq i \leq p}$ from $\R_{(d+1)p-1}[X] \otimes \R^q$ to $(\R^r)^p$. Since we assumed that $L(x_i)$ is surjective for all $i \in \ssquarebrackets{1}{p}$, the previous interpolation result proves that $\ev_{\ux}^\cD$ is surjective. Then, as in Equation~\eqref{eq: def GI tilde}, for all non-empty $I \subset \ssquarebrackets{1}{p}$ we define
\begin{equation*}
\cG_I^\cD(\ux) = \ker \ev_{\ux_I}^\cD \in \gr[{\R_{(d+1)\norm{I}-1}[X]\otimes \R^q}]{r\norm{I}}.
\end{equation*}
We also define $\cG^\cD(\ux) = \cG^\cD_{\ssquarebrackets{1}{p}}(\ux)$.

Following the same strategy as in Section~\ref{sec: Definition of the multijet bundles}, we denote by $\Sigma_\cD$ the graph of $\parentheses*{\cG_I^\cD}_{I \subset \ssquarebrackets{1}{p}}$ defined on $\config$. We define $C_p^\cD[\R^n]$ as a resolution of the singularities of the algebraic variety
\begin{equation*}
\bar{\Sigma_\cD} \subset (\R^n)^p \times \prod_{\emptyset \neq I \subset \ssquarebrackets{1}{p}} \gr[{\R_{(d+1)\norm{I}-1}[X]\otimes \R^q}]{r\norm{I}}.
\end{equation*}
The manifold $C_p^\cD[\R^n]$ satisfies the analogue of Corollary~\ref{cor: existence of the basis CpRn}. In particular the maps $\cG_I^\cD$ with $I \subset \ssquarebrackets{1}{p}$ extend smoothly to $C_p^\cD[\R^n]$. Then we define the $p$-multijet bundle adapted to $\cD$ as:
\begin{equation*}
\cMJ_p^\cD(\R^n) = \parentheses*{\parentheses*{\R_{(d+1)p-1}[X] \otimes \R^q} \times C_p^\cD[\R^n]}/ \cG^\cD
\end{equation*}
over $C_p^\cD[\R^n]$, similarly to Definition~\ref{def: vector bundle of multijets}. Given a function $f \in \cC^{(d+1)p-1}(\R^n,\R^q)$, we define its $p$-multijet adapted to $\cD$ at $z \in C_p^\cD[\R^n]$ as:
\begin{equation*}
\mj_p^\cD(f,z) = K(f,\hat{\pi(z)}) \mod \cG^\cD(z).
\end{equation*}
Then, following the same steps as in Section~\ref{sec: Definition of the multijet bundles}, one can check that the objects we just defined satisfy the conditions in Theorem~\ref{thm: multijet bundle D}.
\end{proof}

As before, thanks to the localness condition in Theorem~\ref{thm: multijet bundle D} (Condition~\ref{item: thm multijet D localness}), the multijet $\mj_p^\cD(f,z)$ makes sense even if $f$ is only defined and $\cC^{(d+1)\norm{I}-1}$ near $y_I$, for all $I \in \cI(\pi(z))$. Hence, the following definition makes sense.

\begin{dfn}[Multijets adapted to $\cD$]
\label{def: multijets D}
Let $\Omega \subset \R^n$ be open. We define $C_p^\cD[\Omega] = \pi^{-1}(\Omega^p)$ and denote by $\cMJ_p^\cD(\Omega)$ the restriction of $\cMJ_p^\cD(\R^n)$ to $C_p^\cD[\Omega]$. We call $\cMJ_p^\cD(\Omega) \to C_p^\cD[\Omega]$ the \emph{bundle of $p$-multijets adapted to $\cD$}. Let $f:\Omega\to \R^q$ be of class $\cC^{(d+1)p-1}$, we call the section $\mj_p^\cD(f,\cdot)$ of $\cMJ_p^\cD(\Omega)$ the \emph{$p$-multijet of~$f$  adapted to $\cD$}.
\end{dfn}


\subsection{Finiteness of moments for critical points}
\label{subsec: Finiteness of moments for critical points}

The purpose of this section is to prove Theorem~\ref{thm: moments critical points}. More generally we prove an analogous result for the zero set of $\cD s$, where $s$ is a section of a vector bundle $E \to M$ and $\cD$ is a differential operator; see Theorem~\ref{thm: moments M D}. This is done by adapting what we did in Section~\ref{sec: Application to zeros of Gaussian fields} to this framework.

Let $\parentheses*{M,g}$ be Riemannian manifold of dimension $n \geq 1$ without boundary. Let $E \to M$ (resp.~$F \to M$) be a smooth vector bundle of rank $q \geq 1$ (resp.~$r \in \ssquarebrackets{1}{n}$). We consider a differential operator $\cD:\Gamma^d(M,E) \to \Gamma^0(M,F)$ of order $d \geq 0$, corresponding to a smooth section $L \in \Gamma^\infty(M,\cJ_d(M,E)^* \otimes F)$; see Remark~\ref{rem: differential operator}. Thanks to this smoothness assumption we have $\cD:\Gamma^{d+l}(M,E) \to \Gamma^l(M,F)$ for all $l \geq 0$. Finally we assume that $L(x):\cJ_d(M,E)_x \to F_x$ (or equivalently $\cD_x:\Gamma^d(M,E) \to F_x)$ is surjective for all $x \in M$.

Let $s:M \to E$ be a centered Gaussian field on $M$ with values in $E$ in the sense of Definition~\ref{def: Gaussian field}. We assume that $s$ is $\cC^{d+1}$ and $d$-non-degenerate, so that $\jet_d(s,\cdot)$ is a centered Gaussian field with values in $\cJ_d(M,E)$ which is $\cC^1$ and non-degenerate. Then $\cD s \in \Gamma^1(M,F)$ is a centered Gaussian field with values in $F$, which is non-degenerate because of the surjectivity assumption on~$L$. Everything we did in Sections~\ref{subsec: Bulinskaya lemma and Kac--Rice formula for the expectation} and~\ref{subsec: Factorial moment measures and Kac--Rice densities} applies to $\cD s$. In particular, $\cD s$ satisfies the weak Bulinskaya Lemma (see Proposition~\ref{prop: weak Bulinskaya}). Hence its vanishing locus is almost surely a union of codimension-$r$ submanifold of $M$ and a negligible singular set. We denote by $\nu_\cD$ the random Radon measure on $M$ defined by integrating over the zero set of $\cD s$. The formal definition is similar to Definition~\ref{def: nu}.

\begin{ex}
\label{ex: critical points}
Let us assume that $E = \R \times M$ is trivial. Then we can identify $\Gamma^1(M,E)$ with $\cC^1(M)$ and consider the differential $D:\cC^1(M) \to \Gamma^0\parentheses*{M,T^*M}$, which is a differential operator of order $1$. Let $f:M \to \R$ be a $\cC^2$ and $1$-non-degenerate centered Gaussian field. Then $D f$ is a non-degenerate $\cC^1$ centered Gaussian field on $M$ with values in $T^*M$, the vanishing locus of $D f$ is the set of critical points of $f$, and $\nu_D$ is the counting measure this random set.
\end{ex}

We can now state the analogue of Theorem~\ref{thm: moments Rn local} in this context, bearing in mind that Proposition~\ref{prop: relation between moments measures and densities} applies to $\nu_\cD$.

\begin{thm}[Finiteness of moments for \texorpdfstring{$\nu_\cD$}{}, local version]
\label{thm: moments D local}
Let $\Omega \subset \R^n$ be open and $f:\Omega \to \R^q$ be a centered Gaussian field. Let $r \in \ssquarebrackets{1}{n}$. Let $\cD:\cC^d(\Omega,\R^q) \to \cC^0(\Omega,\R^r)$ be a differential operator of order $d$ satisfying the previous hypotheses and $\nu_\cD$ denote the measure of integration over the zero set of $\cD f$. Let $p \geq 1$, if $f$ is $\cC^{(d+1)p}$ and the Gaussian field $\mj_k^\cD(f,\cdot):C_k^\cD[\Omega] \to \cMJ_k^\cD(\Omega)$ is non-degenerate for all $k \in \ssquarebrackets{1}{p}$, then the four equivalent statements in Proposition~\ref{prop: relation between moments measures and densities} hold for~$\nu_\cD$.
\end{thm}

\begin{proof}
Under these hypotheses, for all $k \in \ssquarebrackets{1}{p}$ the Gaussian field $\mj_k^\cD(f,\cdot)$ is at least $\cC^1$. Then the proof is the same as that of Theorem~\ref{thm: moments Rn local}.
\end{proof}

\begin{thm}[Finiteness of moments for zeros of \texorpdfstring{$\cD s$}{}]
\label{thm: moments M D}
In the setting introduced at the beginning of this section, let $s:M \to E$ be a centered Gaussian field and $\nu_\cD$ denote the measure of integration over the zero set of $\cD s$. Let $p\geq 1$, if~$s$ is $\cC^{(d+1)p}$ and $((d+1)p-1)$-non-degenerate then $\esp{\norm*{\prsc*{\nu_\cD}{\phi}}^p}<+\infty$ for all $\phi \in L^\infty_c(M)$.
\end{thm}

\begin{proof}
We deduce Theorem~\ref{thm: moments M D} from Theorem~\ref{thm: moments D local} in the same way that we deduced Theorem~\ref{thm: moments M} from Theorem~\ref{thm: moments Rn local}; see Section~\ref{subsec: Proof of finiteness of moments}.
\end{proof}


\section{Multijets of holomorphic maps}
\label{sec: Multijets of holomorphic maps}

The purpose of this section is to explain how to adapt what we did in Sections~\ref{sec: Divided differences and Kergin interpolation} to~\ref{sec: Application to zeros of Gaussian fields} to the case of holomorphic maps. Theorem~\ref{thm: moments Rn} asks for the $(p-1)$-non-degeneracy of the field $f$, that is, $\jet_{p-1}(f,x)$ needs to be non-degenerate for all $x$. If $f: \C^n \to \C$ is a centered holomorphic Gaussian field, then $\parentheses*{f(x),D_xf}$ is always degenerate. Indeed, identifying canonically $\C$ with $\R^2$, the differential $D_xf$ takes values in the subspace of $\mathcal{L}(\R^{2n},\R^2)$ consisting of $\R$-linear maps that are actually $\C$-linear. Thus, if we see the holomorphic field $f$ as a smooth field from $\R^{2n}$ to $\R^2$, we cannot apply Theorem~\ref{thm: moments Rn}. From the point of view of multijets, the multijet $\mj_p(f,\cdot)$ of a holomorphic function $f:\C^n \to \C$ takes values in a strict sub-bundle of $\cMJ_p(\R^{2n},\R^2)$, which is similar to what happens for jet bundles. Thus, the field $\mj_p(f,\cdot)$ associated with a holomorphic Gaussian field $f$ is necessarily degenerated and Theorem~\ref{thm: moments Rn local} does not apply. To remedy this situation, we define in Section~\ref{subsec: Definition of the holomorphic multijet bundles} a multijet bundle adapted to holomorphic maps. Then, in Section~\ref{subsec: Application to the zeros of holomorphic Gaussian fields}, we use this holomorphic multijet to prove Theorem~\ref{thm: moments holo}.


\subsection{Definition of the holomorphic multijet bundles}
\label{subsec: Definition of the holomorphic multijet bundles}

In this section, we define a multijet bundle for holomorphic maps. Our main result is an equivalent of Theorem~\ref{thm: multijet bundle} in this context. Let us first introduce some notation.

\begin{dfn}[Spaces of holomorphic maps]
\label{def: spaces of holomorphic maps}
We define the following spaces.
\begin{itemize}
\item We denote by $\C_d[X]$ the space of complex polynomials of degree at most $d$ in $n$ variables, where $X=(X_1,\dots,X_n)$ is multivariate.
\item If $M$ and $N$ are two complex manifolds, we denote by $\cO(M,N)$ the space of holomorphic maps from $M$ to $N$. If $N=\C$, we simply write $\cO(M)$.
\item If $E \to M$ is a holomorphic vector bundle, we denote by $\cJ_k^\C(M,E) \to M$ the holomorphic bundle of $k$-jets of holomorphic sections of $E$. If $E = V \times M$ is trivial with fiber $V$, we denote its holomorphic $k$-jet bundle by $\cJ_k^\C(M,V) \to M$. If $V=\C$, we simply write $\cJ_k^\C(M) \to M$. Given a holomorphic section $s$ of $E$, we denote by $\jet_k^\C(s,x)$ its holomorphic $k$-jet at $x \in M$.
\end{itemize}
\end{dfn}

\begin{thm}[Existence of holomorphic multijet bundles]
\label{thm: holomorphic multijet bundle}
Let $n \geq 1$ and $p \geq 1$ and let $V$ be a complex vector space of dimension $r\geq 1$. There exist a complex manifold $C_p^\C[\C^n]$ of dimension $np$ and a holomorphic vector bundle $\cMJ_p^\C(\C^n,V) \to C_p^\C[\C^n]$ of rank $rp$ with the following properties.

\begin{enumerate}

\item \label{item: thm holo multijet pi} There exists a holomorphic proper surjection $\pi:C_p^\C[\C^n] \to (\C^n)^p$ such that $\pi^{-1}\parentheses*{\configC}$ is a dense open subset of $C_p^\C[\C^n]$, and $\pi$ restricted to $\pi^{-1}\parentheses*{\configC}$ is a biholomorphism onto $\configC$.

\item \label{item: thm holo multijet mj} There exists a map $\mj_p^\C:\cO(\C^n,V) \times C_p^\C[\C^n] \to \cMJ_p^\C(\C^n,V)$ such that:
\begin{itemize}
\item for all $z \in C_p^\C[\C^n]$, the linear map $\mj_p^\C(\cdot,z):\cO(\C^n,V) \to \cMJ_p^\C(\C^n,V)_z$ is surjective;
\item for all $f \in \cO(\C^n,V)$, the section $\mj_p^\C(f,\cdot)$ of $\cMJ_p^\C(\C^n,V) \to C_p^\C[\C^n]$ is holomorphic.
\end{itemize}

\item \label{item: thm holo multijet ev} Let $z \in C_p^\C[\C^n]$ be such that $\pi(z)=(x_1,\dots,x_p) \notin \Delta_p$ then for all $f \in \cO(\C^n,V)$ we have:
\begin{equation*}
\mj_p^\C(f,z) = 0 \iff \forall i \in \ssquarebrackets{1}{p}, f(x_i)=0.
\end{equation*}

\item \label{item: thm holo multijet localness} Let $z \in C_p^\C[\C^n]$, let $\cI = \cI\parentheses{\pi(z)}$ and let $\parentheses{y_I}_{I \in \cI} = \iota_{\cI}^{-1}\parentheses{\pi(z)} \in (\C^n)^{\cI} \setminus \Delta_{\cI}$, there exists a linear surjection $\Theta_z^\C : \prod_{I \in \cI} \cJ_{\norm{I}-1}^\C(\C^n,V)_{y_I} \to \cMJ_p^\C(\C^n,V)_z$ such that
\begin{equation*}
\forall f \in \cO(\C^n,V), \qquad \mj_p^\C(f,z) = \Theta_z^\C\parentheses*{\parentheses*{\jet_{\norm{I}-1}^\C(f,y_I)}_{I \in \cI}}.
\end{equation*}

\end{enumerate}

\end{thm}

\begin{proof}
The proof follows the same steps as what we did in Sections~\ref{sec: Divided differences and Kergin interpolation}, \ref{sec: Evaluation maps and their kernels} and~\ref{sec: Definition of the multijet bundles} to prove Theorem~\ref{thm: multijet bundle}. In the following, we sketch how the proof of Theorem~\ref{thm: multijet bundle} adapts to the holomorphic case.

\paragraph{Step 1: Divided differences and Kergin interpolation.}
Let us consider $f \in \cO\parentheses*{\C^n}$ and $\ux = (x_0,\dots,x_k) \in (\C^n)^{k+1}$. The divided difference $f[x_0,\dots,x_k]$ from Definition~\ref{def: divided differences} still makes sense. Since $f$ is holomorphic, it is now a symmetric $\C$-multilinear form on $\C^n$, that depends linearly on $f$ and is holomorphic with respect to $\ux$. As explained in~\cite[Prop.~5.1]{Ker1980}, the Kergin interpolating polynomial is well-behaved with respect to holomorphic maps. Given $f \in \cO(\C^n)$ and $\ux \in (\C^n)^p$, Equation~\eqref{eq: definition K} defines $K(f,\ux) \in \C_{p-1}[X]$ that interpolates the values of $f\squarebrackets*{\ux_I}$ for all non-empty $I \subset \ssquarebrackets{1}{p}$. The equivalent of Lemma~\ref{lem: regularity of the Kergin polynomial} is true, in the sense that $K(\cdot,\ux)$ is $\C$-linear and $K(f,\cdot)$ is holomorphic from $(\C^n)^p$ to $\C_{p-1}[X]$.

In this complex framework, the equivalent of Lemma~\ref{lem: compatibility Kergin interpolation} holds, that is: for all $\ux \in (\C^n)^p$ the map $P \mapsto \parentheses*{K(P,\ux_I)}_{I \in \cI(\ux)}$ is surjective from $\C_{p-1}[X]$ to $\prod_{I \in \cI(\ux)}\C_{\norm{I}-1}[X]$. Note however that the proof we gave of Lemma~\ref{lem: compatibility Kergin interpolation} does not adapt to the holomorphic setting since it uses bump functions. Here, we deduce the surjectivity of $\parentheses*{K(\cdot,\ux_I)}_{I \in \cI(\ux)}$ from a general amplitude result in algebraic geometry. Let $\cI = \cI(\ux)$ and $\parentheses{y_I}_{I \in \cI} = \iota_{\cI}^{-1}(\ux)$. For all $I \in \cI$, let $P_I \in \C_{\norm{I}-1}[X]$. Multiplying $P_I$ by the right power of $X_0$ yields a homogeneous polynomial $\tilde{P}_I \in \C_{p-1}^\text{hom}[X_0,\dots,X_n]$, that is, a global holomorphic section of the line bundle $\cO(p-1) \to \C\P^n$. Recall that $\cO(p-1)$ is the $(p-1)$-th tensor power of the hyperplane line bundle $\cO(1)\to \C\P^n$. Since $\cO(1)$ is very ample, the bundle $\cO(p-1)$ is $(p-1)$-jet ample; see~\cite[Cor.~2.1]{BS1993}. This means that there exists $\tilde{P}\in \C_{p-1}^\text{hom}[X_0,\dots,X_n]$ with the same $(\norm{I}-1)$-jet as $\tilde{P}_I$ at $y_I$ for all $I \in \cI$, where we see $\C^n$ as a standard affine chart in $\C\P^n$. Then, for all $I \in \cI$, the polynomial $P = \tilde{P}(1,X_1,\dots,X_n) \in \C_{p-1}[X]$ has the same $(\norm{I}-1)$-jet (i.e.,~the same Taylor polynomial of order $\norm{I}-1$) as $P_I$ at $y_I$. Thus $\parentheses*{K(P,\ux_I)}_{I \in \cI}=\parentheses*{P_I}_{I \in \cI}$.

\paragraph{Step 2: Kernel of the evaluation and resolution of singularities.}
As in Definition~\ref{def: evaluation map}, we define a complex evaluation map by $\ev_{\ux}^\C:f \mapsto \parentheses*{f(x_1),\dots,f(x_p)}$, where $\ux \in (\C^n)^p$. If $\ux \notin \Delta_p$, this map is surjective from $\C_{p-1}[X]$ to $\C^p$. Hence we can define $\cG_I^\C(\ux) = \ker \ev_{\ux_I}^{\C} \in \gr[{\C_{\norm{I}-1}[X]}]{\norm{I}}$ for all non-empty $I \subset \ssquarebrackets{1}{p}$, where the Grassmannian is now the Grassmannian of complex subspaces of codimension $\norm{I}$. Then, everything we did in Sections~\ref{sec: Evaluation maps and their kernels} and~\ref{sec: Definition of the multijet bundles} works in the holomorphic setting after replacing $\R$-linear objects by $\C$-linear ones.

We define $\Sigma_\C$ as the graph of $\parentheses*{\cG_I^\C}_{I \subset \ssquarebrackets{1}{p}}$ from $\configC$ to $\prod_{\emptyset \neq I \subset\ssquarebrackets{1}{p}} \gr[{\C_{\norm{I}-1}[X]}]{\norm{I}}$ and $C_p^\C[\C^n]$ as a resolution of the singularities of its closure $\bar{\Sigma_\C}$ in $(\C^n)^p \times \prod_{\emptyset \neq I \subset\ssquarebrackets{1}{p}} \gr[{\C_{\norm{I}-1}[X]}]{\norm{I}}$. The resolution of singularities is a result from algebraic geometry which holds over fields of characteristic $0$. In particular, in Proposition~\ref{prop: resolution of singularities} and Corollary~\ref{cor: existence of the basis CpRn}, we can replace ``smooth'' by ``algebraic'' everywhere. The same results hold over $\C$, in which case algebraic implies holomorphic. Thus, $C_p^\C[\C^n]$ is a complex manifold of dimension $np$, which satisfies the equivalent of Corollary~\ref{cor: existence of the basis CpRn} with ``smooth'' replaced by ``holomorphic''.

\paragraph{Step 3: Definition of the holomorphic multijet bundles.} Everything we did in Sections~\ref{subsec: Definition of the bundle MPjRn}, \ref{subsec: Localness of multijets} and~\ref{subsec: Multijets of vector valued maps} adapts to the holomorphic setting. It is enough to replace $\cC^k$ functions by holomorphic ones and to write the linear arguments over $\C$ instead of $\R$. We can then define the holomorphic vector bundle of $p$-multijets of holomorphic functions on $\C^n$ by
\begin{equation}
\label{eq: def holomorphic MJ}
\cMJ_p^\C(\C^n) = \parentheses*{\C_{p-1}[X] \times C_p^\C[\C^n]}/\cG^\C
\end{equation}
and the $p$-multijet of $f \in \cO(\C^n)$ by $\mj_p^\C(f,z) = K(f,\pi(z)) \mod \cG^\C(z)$ for all $z \in C_p^\C[\C^n]$. If $V$ is a complex vector space of finite dimension, we define as in Definition~\ref{def: multijet bundle of vector valued maps}
\begin{equation}
\label{eq: def holomorphic MJ V}
\cMJ_p^\C(\C^n,V) = \cMJ_p^\C(\C^n) \otimes V.
\end{equation}
Then we define the $p$-multijet of $f \in \cO(\C^n,V)$ as in Definition~\ref{def: multijet of a map}. If $\parentheses{v_1,\dots,v_r}$ is a basis of $V$ and $f= \sum_{i=1}^r f_i v_i$ is holomorphic, then $\mj_p^\C(f,z) = \sum_{i=1}^r \mj_p^\C(f_i,z)\otimes v_i$ for all $z \in C_p^\C[\C^n]$. As in Lemma~\ref{lem: multijet intrinsic}, this definition does not depend on a choice of basis. This defines the holomorphic $p$-multijet that we are looking for.
\end{proof}

As in the real case, thanks to Condition~\ref{item: thm holo multijet localness} in Theorem~\ref{thm: holomorphic multijet bundle}, the holomorphic multijet $\mj_p^\C(f,z)$ of $f$ at $z \in C_p^\C[\C^n]$ only depends on the germ of $f$ near the $x_i$, where $(x_i)_{1 \leq i \leq p}=\pi(z)$. Thus, we can define a holomorphic multijet bundle over any open subset of $\C^n$.

\begin{dfn}[Holomorphic multijets]
\label{def: holo multijets}
Let $\Omega \subset \C^n$ be open, we denote by $C_p^\C[\Omega] = \pi^{-1}(\Omega^p)$ and by $\cMJ_p^\C(\Omega,V) \to C_p^\C[\Omega]$ the restriction of $\cMJ_p^\C(\C^n,V)$ to $C_p[\Omega]$. If $V=\C$, we drop it from the notation and write $\cMJ_p^\C(\Omega) \to C_p^\C[\Omega]$. Let $f\in \cO(\Omega,V)$, we call the section $\mj_p^\C(f,\cdot)$ of $\cMJ_p^\C(\Omega,V)$ the \emph{holomorphic $p$-multijet of~$f$}.
\end{dfn}


\subsection{Application to the zeros of holomorphic Gaussian fields}
\label{subsec: Application to the zeros of holomorphic Gaussian fields}

In this section, we explain how the holomorphic multijets defined in Section~\ref{subsec: Definition of the holomorphic multijet bundles} allow us to prove Theorem~\ref{thm: moments holo}, and the analogue of Theorem~\ref{thm: moments Rn local} for holomorphic Gaussian fields. We start by recalling a few facts about complex Gaussian vectors; see~\cite[Chap.~2]{AHSE1995}.

A random variable $X \in \C$ is called a \emph{centered complex Gaussian} if its real and imaginary parts are independent real centered Gaussian variables of the same variance, i.e.,~there exists $\lambda \geq 0$ such that $X = X_\Re + i X_\Im$ with $\parentheses{X_\Re,X_\Im} \sim \gauss{\lambda \Id}$ in $\R^2$.

\begin{dfn}[Complex Gaussian vector]
\label{def: complex Gaussian vector}
We say that a random vector $X$ in a finite-dimensional complex vector space $V$ is a \emph{centered Gaussian} if for all $\eta \in V^*$, the complex variable $\eta(X)$ is a centered complex Gaussian.

If $V$ is equipped with a Hermitian inner product $\prsc{\cdot}{\cdot}$ and we define $v^* = \prsc{v}{\cdot}$ then the \emph{variance} of $X$ is the non-negative Hermitian operator $\varC{X} = \esp{X \otimes X^*}$. We say that $X$ is \emph{non-degenerate} if $\varC{X}$ is positive-definite.
\end{dfn}

\begin{rem}
\label{rem: complex Gaussian vector}
As in the real case, the Gaussianity and non-degeneracy of $X$ do not depend on $\prsc{\cdot}{\cdot}$, but the variance operator does.
\end{rem}

A centered complex Gaussian vector $X$ in $\parentheses*{V,\prsc{\cdot}{\cdot}}$ is completely determined by its variance. For example, if $\varC{X} = \Lambda$ is positive-definite, then $X$ admits the density $v \mapsto \frac{1}{\det\parentheses{\pi \Lambda}}e^{-\prsc{v}{\Lambda^{-1}v}}$ with respect to the Lebesgue measure on $V$. We denote by $\gaussC{\Lambda}$ the centered complex Gaussian distribution of variance $\Lambda$. Then $X \sim \gaussC{\Lambda}$ in $\C^n$ if and only if its real and imaginary part satisfy $\parentheses{X_\Re,X_\Im} \sim \gauss{\frac{1}{2}\parentheses*{\begin{smallmatrix}
\Re(\Lambda) & \Im(\Lambda) \\ - \Im(\Lambda) & \Re(\Lambda)
\end{smallmatrix}}}$ in $\R^{2n}$.

Let $E \to M$ be a holomorphic vector bundle over a complex manifold $M$, we denote by $H^0(M,E)$ the vector space of global holomorphic sections of $E \to M$.

\begin{dfn}[Holomorphic Gaussian field]
\label{def: holomorphic Gaussian field}
We say that a random section $s \in H^0(M,E)$ is a \emph{centered holomorphic Gaussian field} if for all $m \geq 1$ and all $x_1,\dots,x_m$ the random vector $\parentheses{s(x_1),\dots,s(x_m)}$ is a centered complex Gaussian. We say that this field is \emph{non-degenerate} if $s(x)$ is non-degenerate for all $x \in M$.
\end{dfn}

Note that if $s \in H^0(M,E)$ is a centered holomorphic Gaussian field then, for all $k \in \N$, its holomorphic $k$-jet $\jet_k^\C(s,\cdot)$ defines a centered holomorphic Gaussian field with values in $\cJ_k^\C(M,E)$.

\begin{dfn}[$p$-non-degeneracy for holomorphic fields]
\label{def: p non-degeneracy holomorphic}
Let $p \geq 1$, we say that the centered holomorphic Gaussian field $s \in H^0(M,E)$ is \emph{$p$-non-degenerate} if the centered complex Gaussian $\jet_p^\C(s,x) \in \cJ_p^\C(M,E)_x$ is non-degenerate for all $x \in M$.
\end{dfn}

As in the real framework, we need the following definition.

\begin{dfn}[Complex Jacobian]
\label{def: complex Jacobian}
Let $L:V \to V'$ be a $\C$-linear map between Hermitian spaces and let $L^*$ denote its adjoint map. The \emph{complex Jacobian} of $L$ is defined as $\jacC{L}=\det\parentheses{LL^*}$.
\end{dfn}

\begin{rem}
\label{rem: complex Jacobian}
If we see $V$, $V'$ and $L$ as $\R$-linear objects and we equip $V$ and $V'$ with the Euclidean structures induced by their Hermitian inner products, then the real and complex Jacobians are related by $\jacC{L} = \jac{L}^2$; see~\cite[Thm.~A.5]{AHSE1995}.
\end{rem}

Let us consider a complex manifold $M$ of complex dimension $n$ equipped with a Riemannian metric $g$ and a holomorphic vector bundle $E \to M$ of complex rank $r \in \ssquarebrackets{1}{n}$. In the following, we denote by $\nabla$ a connection on $E$ which is compatible with the complex structure. As in the real case, the choice of this connection will not matter. Let $s \in H^0(M,E)$ be a centered holomorphic Gaussian field on $M$ with values in $E$ and let $Z= s^{-1}(0)$ denote its vanishing locus. We will always assume that $s$ is non-degenerate. In this setting, the random section $s$ satisfies a strong Bulinskaya-type lemma.

\begin{prop}[Holomorphic Bulinskaya lemma]
\label{prop: holomorphic Bulinskaya}
Almost surely the following set is empty
\begin{equation*}
\brackets*{x \in M \mvert s(x)=0 \ \text{and} \ \jacC{\nabla_xs}=0}.
\end{equation*}
In particular, the zero set $Z$ is almost surely a (possibly empty) complex submanifold of complex codimension $r$ in $M$.
\end{prop}

\begin{proof}
It is enough to check the result locally. On an open subset $\Omega \subset M$ over which $E$ is trivial, we can consider $s$ as a non-degenerate smooth centered Gaussian field from $\Omega$ to $\C^r \simeq \R^{2r}$. Then, the local result follows from~\cite[Thm.~7]{LS2019}.
\end{proof}

Let us consider $Z$ as random submanifold of real codimension $2r$ in the Riemannian manifold $(M,g)$ of real dimension $2n$. The metric $g$ induces a $(2n-2r)$-dimensional Riemannian volume $\dx\Vol_Z$ on $Z$ and we can define $\nu$ as in Definition~\ref{def: nu}, bearing in mind that $Z=Z_\text{reg}$. Then Propositions~\ref{prop: Kac-Rice expectation}, \ref{prop: Kac--Rice moments} and~\ref{prop: relation between moments measures and densities} hold for the holomorphic field $s$ and the associated linear statistics $\prsc{\nu}{\phi}$ with $\phi \in L^\infty_c(M)$. More generally, everything we did in Sections~\ref{subsec: Bulinskaya lemma and Kac--Rice formula for the expectation}, \ref{subsec: Factorial moment measures and Kac--Rice densities} and~\ref{subsec: Proof of finiteness of moments} adapts to the holomorphic setting.

\begin{rem}
\label{rem: complex Kac-Rice density}
In Definition~\ref{def: rhop}, the function $\rho_p$ is defined in terms of real Jacobians and the variance of $\parentheses{s(x_1),\dots,s(x_p)}$ seen as a real Gaussian vector. One can check that another expression of $\rho_p(x_1,\dots,x_p)$ is the following, which is more natural in our holomorphic framework:
\begin{equation*}
\forall (x_1,\dots,x_p) \notin \Delta_p,\quad \rho_p\parentheses*{x_1,\dots,x_p}=\frac{\espcond{\strut \prod_{i=1}^p \jacC{\nabla_{x_i}s}}{\forall i \in \ssquarebrackets{1}{p}, s(x_i)=0}}{\det\parentheses*{\pi \varC{\strut s(x_1),\dots,s(x_p)}\strut}}.
\end{equation*}
\end{rem}

We can now state the equivalent of Theorem~\ref{thm: moments Rn local} for holomorphic Gaussian fields. Let $\Omega \subset \C^n$ be open. Recall that $\cMJ_p^\C(\Omega,\C^r) \to C_p^\C[\Omega]$ is defined in Definition~\ref{def: holo multijets} as the restriction over $C_p^\C[\Omega] \subset C_p^\C[\C^n]$ of the vector bundle $\cMJ_p^\C(\C^n,\C^r) \to C_p^\C[\C^n]$ from Theorem~\ref{thm: holomorphic multijet bundle}.

\begin{thm}[Finiteness of moments for holomorphic fields, local version]
\label{thm: moments holomorphic local}
Let $f:\Omega \to \C^r$ be a centered holomorphic Gaussian field and $\nu$ be as in Definition~\ref{def: nu}. Let $p \geq 1$, if for all $k \in \ssquarebrackets{1}{p}$ the holomorphic Gaussian field $\mj_k^\C(f,\cdot)\in H^0\parentheses*{C_k^\C[\Omega],\cMJ_k^\C(\Omega,\C^r)}$ is non-degenerate, then the four equivalent statements in Proposition~\ref{prop: relation between moments measures and densities} hold.
\end{thm}

\begin{proof}
The proof of Theorem~\ref{thm: moments Rn local} relies mostly on two facts that are valid for all $k \in \ssquarebrackets{1}{p}$. First, on the open dense subset $\Omega^k \setminus \Delta_k \subset C_k[\Omega]$, the zero set of $\mj_k(f,\cdot)$ is the same as that of $(x_1,\dots,x_k) \mapsto (f(x_1),\dots,f(x_k))$. And second, the field $\mj_k(f,\cdot)$ is non-degenerate on $C_k[\Omega]$, so that we can apply the Kac--Rice formula (Proposition~\ref{prop: Kac-Rice expectation}) to the $k$-multijet.

These two facts are still true in the present holomorphic setting. Hence, the same proof as that of Theorem~\ref{thm: moments Rn local} yields the result.
\end{proof}

We deduce from this result the equivalent of Theorem~\ref{thm: moments M} for a centered holomorphic Gaussian field $s$ on a complex manifold $M$ of dimension $n$ with values in a holomorphic vector bundle $E$ of rank $r \in \ssquarebrackets{1}{n}$. Theorem~\ref{thm: moments holo} is a special case of the following.

\begin{thm}[Finiteness of moments for zeros of holomorphic Gaussian sections]
\label{thm: moments M holo}
Let $p\geq 1$, let $s \in H^0(M,E)$ be a centered holomorphic Gaussian field and let $\nu$ be as in Definition~\ref{def: nu}. If~$s$ is $(p-1)$-non-degenerate then $\esp{\norm{\prsc{\nu}{\phi}}^p}<+\infty$ for all $\phi \in L^\infty_c(M)$.
\end{thm}

\begin{proof}
We deduce Theorem~\ref{thm: moments M holo} from Theorem~\ref{thm: moments holomorphic local} in the same way that we deduced Theorem~\ref{thm: moments M} from Theorem~\ref{thm: moments Rn local}; see Section~\ref{subsec: Proof of finiteness of moments}.
\end{proof}

\begin{rem}
\label{rem: Bleher Shiffman Zelditch}
In particular, Theorem~\ref{thm: moments M holo} proves the local integrability of the $p$-points correlation functions studied in~\cite{BSZ2000a} and their scaling limit.
\end{rem}

\bibliographystyle{amsplain}
\bibliography{MultijetBundlesFinitenessMomentsZerosGaussianFields}

\end{document}